\documentclass[a4paper]{amsart}

\usepackage{amscd,amsthm,amsfonts,amssymb,amsmath,yhmath}
\usepackage[colorlinks=true]{hyperref}
\usepackage{fullpage}
\usepackage[all]{xy}
\usepackage{mathrsfs}
\usepackage{bm}
\usepackage{tikz}
\usepackage{bbm}
\usepackage{graphics}
\usepackage{enumerate}
\usepackage[utf8]{inputenc}

\newcommand{\BA}{{\mathbb {A}}}\newcommand{\BC}{{\mathbb {C}}}\newcommand{\BD}{{\mathbb {D}}}
\newcommand{\BG}{{\mathbb {G}}}\newcommand{\BH}{{\mathbb {H}}}

\newcommand{\BQ}{{\mathbb {Q}}}\newcommand{\BR}{{\mathbb {R}}}

\newcommand{\BZ}{{\mathbb {Z}}}

\newcommand{\bfA}{{\mathbf {A}}}
\newcommand{\bfG}{{\mathbf {G}}}\newcommand{\bfH}{{\mathbf {H}}}
\newcommand{\bfL}{{\mathbf {L}}}
\newcommand{\bfM}{{\mathbf {M}}}\newcommand{\bfN}{{\mathbf {N}}}\newcommand{\bfP}{{\mathbf {P}}}
\newcommand{\bfQ}{{\mathbf {Q}}}\newcommand{\bfR}{{\mathbf {R}}}\newcommand{\bfT}{{\mathbf {T}}}
\newcommand{\bfX}{{\mathbf {X}}}

\newcommand{\CC}{{\mathcal {C}}}

\newcommand{\CO}{{\mathcal {O}}}
\newcommand{\CS}{{\mathcal {S}}}

\newcommand{\msf}{\mathscr{F}}
\newcommand{\msl}{\mathscr{L}}
\newcommand{\msp}{\mathscr{P}}
\newcommand{\mst}{\mathscr{T}}

\newcommand{\fa}{{\mathfrak{a}}} \newcommand{\fc}{{\mathfrak{c}}} \newcommand{\fd}{{\mathfrak{d}}}
\newcommand{\fe}{{\mathfrak{e}}} \newcommand{\fg}{{\mathfrak{g}}} \newcommand{\fh}{{\mathfrak{h}}}
 \newcommand{\fk}{{\mathfrak{k}}} \newcommand{\fl}{{\mathfrak{l}}}
\newcommand{\fm}{{\mathfrak{m}}} \newcommand{\fn}{{\mathfrak{n}}}\newcommand{\fo}{{\mathfrak{o}}} \newcommand{\fp}{{\mathfrak{p}}}
 \newcommand{\fr}{{\mathfrak{r}}}\newcommand{\fs}{{\mathfrak{s}}} \newcommand{\ft}{{\mathfrak{t}}}
 \newcommand{\fv}{{\mathfrak{v}}}

\newcommand{\ab}{{\mathrm{ab}}}                      \newcommand{\Ad}{{\mathrm{Ad}}}              \newcommand{\ad}{{\mathrm{ad}}}

                      \newcommand{\bs}{\backslash}

                      		\newcommand{\Cent}{{\mathrm{Cent}}}

\newcommand{\diag}{{\mathrm{diag}}}                            
                  			\newcommand{\der}{{\mathrm{der}}}

			\newcommand{\el}{\mathrm{ell}}

\newcommand{\Gal}{{\mathrm{Gal}}}

\newcommand{\Hom}{{\mathrm{Hom}}}

\newcommand{\Lie}{{\mathrm{Lie}}}                           
                       	\newcommand{\da}{\leftrightarrow}

			\newcommand{\Nrd}{{\mathrm{Nrd}}}		\newcommand{\Norm}{{\mathrm{Norm}}}

\newcommand{\ov}{\overline}

                  \newcommand{\ra}{\rightarrow}    
\newcommand{\rk}{\mathrm{rk}}

         	\newcommand{\Prd}{{\mathrm{Prd}}}

                   \newcommand{\reg}{{\mathrm{reg}}}               \newcommand{\Res}{{\mathrm{Res}}}
				 \newcommand{\rs}{{\mathrm{rs}}}

                                  \newcommand{\Supp}{{\mathrm{Supp}}}
 \newcommand{\sk}{\medskip}                      \newcommand{\s}{\sk\noindent}
\newcommand{\scn}{{\mathrm{sc}}} 

                   \newcommand{\tr}{{\mathrm{tr}}}                 
\newcommand{\Trd}{{\mathrm{Trd}}}			\newcommand{\Tran}{{\mathrm{Tran}}}

\newcommand{\vol}{{\mathrm{vol}}}

\newcommand{\wt}{\widetilde}                        \newcommand{\wh}{\widehat}

\newtheorem{thm}{Theorem}[section]
\newtheorem{coro}[thm]{Corollary}
\newtheorem{lem}[thm]{Lemma}
\newtheorem{prop}[thm]{Proposition}

\newtheorem{defn}[thm]{Definition}

\theoremstyle{definition}

\theoremstyle{remark}
\newtheorem{remark}[thm]{Remark}

\numberwithin{equation}{subsection}
\def\mat(#1,#2,#3,#4){
  \begin{pmatrix}
  #1 & #2 \\ #3 & #4
  \end{pmatrix}
}

\begin{document}
\title{On certain identities between Fourier transforms of weighted orbital integrals on infinitesimal symmetric spaces of Guo-Jacquet}
\author{Huajie Li}
\date{\today}
\maketitle

\begin{abstract}
In an infinitesimal variant of Guo-Jacquet trace formulae, the regular semi-simple terms are expressed as noninvariant weighted orbital integrals on two global infinitesimal symmetric spaces. We prove some relations between the Fourier transforms of invariant weighted orbital integrals on the corresponding local infinitesimal symmetric spaces. These relations should be useful in the noninvariant comparison of the infinitesimal variant of Guo-Jacquet trace formulae. 
\end{abstract}

\tableofcontents


\section{\textbf{Introduction}}

Inspired by Jacquet's new proof \cite{MR868299} of Waldspurger's well-known result \cite{MR783511} on the central values of automorphic $L$-functions for $GL_2$, Guo-Jacquet have suggested comparison of two relative trace formulae in \cite{MR1382478} in order to generalise this theorem to higher ranks. This approach has also been followed by Feigon-Martin-Whitehouse \cite{MR3805647} via simple trace formulae. However, if one wants to remove the restrictive conditions in {\it{loc. cit.}}, some additional terms in the Guo-Jacquet trace formula other than relative orbital integrals can not be neglected. 

Our starting point is an infinitesimal analogue of Guo-Jacquet trace formulae and their comparison. It means that we first work on the tangent space of a symmetric space (called an infinitesimal symmetric space). One reason for this is that at the infinitesimal level, the spectral side of the relative trace formula is replaced by the Fourier transform of the geometric side where the harmonic analysis is simpler. Another reason is that the comparison of trace formulae for infinitesimal symmetric spaces is expected to imply the comparison of the original relative trace formulae for symmetric spaces. For example, one may consult Zhang's proof of the transfer of relative local orbital integrals \cite{MR3414387}. 

We have established an infinitesimal variant of Guo-Jacquet trace formulae in \cite{MR4424024, MR4350885}, where the main (namely regular semi-simple) terms are explicit weighted orbital integrals. These distributions should be the first ones to be studied and compared after orbital integrals. However, some new difficulties arise since these distributions are noninvariant. Instead of making the trace formula invariant as Arthur did (see \cite{MR625344} for example), we would like to follow Labesse's proposal \cite{MR1339717} of noninvariant comparison which seems more direct. For example, we have established the weighted fundamental lemma for infinitesimal Guo-Jacquet trace formulae in \cite{MR4350885} as a noninvariant and infinitesimal avatar of Guo's fundamental lemma \cite{MR1382478}. The strategy of noninvariant comparison has also been adopted by Chaudouard in \cite{MR2164623, MR2332352} on the stable base change. These works provide some indications to our work. 

Let us recall some basic objects in the local setting. Let $E/F$ be a quadratic extension of non-archimedean local fields of characteristic zero. Let $\eta$ be the quadratic character of $F/NE^\times$ attached to $E/F$, where $NE^\times$ denotes the norm of $E^\times$. The first symmetric pair is $(G,H)=(GL_{2n}, GL_n\times GL_n)$. Let $\fs\simeq\fg\fl_n\oplus\fg\fl_n$ be the corresponding infinitesimal symmetric space. Denote by $\fs_\rs$ the set of regular semi-simple elements in $\fs$ (see Section \ref{ssec:def-rss}). Let $M$ be an $\omega$-stable (see Section \ref{ssec:def-omega-stable}) Levi subgroup of $G$, and $X\in(\fm\cap\fs_\rs)(F)$. Let $f$ be a locally constant and compactly supported function on $\fs(F)$. We define the weighted orbital integral $J_M^G(\eta, X, f)$ by \eqref{defnoninvwoi1}. We have proved in \cite{MR4681295} that its Fourier transform is represented by a locally constant function $\hat{j}_M^G(\eta, X,\cdot)$ on $\fs_\rs(F)$. We have also defined the $(H,\eta)$-invariant weighted orbital integrals $I_M^G(\eta, X, f)$ in {\it{loc. cit.}} by Arthur's standard method, whose Fourier transform is represented by a locally constant function $\hat{i}_M^G(\eta, X,\cdot)$ on $\fs_\rs(F)$. The second symmetric pair is $(G', H')$, where $G'$ is the group of invertible elements in a central simple algebra over $F$ containing $E$, and $H'$ is the centraliser of $E^\times$ in $G'$. It is inspired by the related local conjecture of Prasad and Takloo-Bighash \cite{MR2806111} and more general than Guo-Jacquet's original setting. Denote by $\fs'$ the corresponding infinitesimal symmetric space. For a Levi subgroup $M'$ of $H'$ and $Y\in(\wt{\fm'}\cap\fs'_\rs)(F)$ (see Section \ref{ssec:sym2-I-II}), we similarly obtain local constant functions $\hat{j}_{M'}^{H'}(Y,\cdot)$ and $\hat{i}_{M'}^{H'}(Y,\cdot)$ on $\fs'_\rs(F)$. 

The functions $\hat{j}_M^G(\eta, X, \cdot)$ is decomposed as their invariant analogues $\hat{i}_M^G(\eta, X, \cdot)$ and weight functions $v_M^G$. The decomposition for the functions $\hat{j}_{M'}^{H'}(Y, \cdot)$ is similar. In order to obtain relations between $\hat{j}_M^G(\eta, X, \cdot)$ and $\hat{j}_{M'}^{H'}(Y, \cdot)$, which is part of the noninvariant comparison of the infinitesimal variant of Guo-Jacquet trace formulae, we shall focus on the relations between $\hat{i}_M^G(\eta, X, \cdot)$ and $\hat{i}_{M'}^{H'}(Y, \cdot)$ in this paper. There is an injection $M'\mapsto M$ from the set of Levi subgroups of $H'$ into the set of $\omega$-stable Levi subgroups of $G$ (see Section \ref{ssec:mat_Levi}). We fix such a matching pair of Levi subgroups. We define the notion of matching orbits between $\fs_\rs(F)$ and $\fs'_\rs(F)$ by Definition \ref{defbyinv}. There is also a refined notion of $M$-matching orbits (see Definition \ref{defbyinvM}). For $X=\mat(0,A,B,0)\in\fs_\rs(F)$, we define the transfer factor $\kappa(X):=\det(A)$ and the sign $\eta(X):=\eta(\det(AB))$. Our main result is as follows. 

\begin{thm}[see Corollary \ref{parcommutecor1} and Proposition \ref{vancommute}]
\leavevmode
\begin{enumerate}
	\item Let $X\in(\fm\cap\fs_\rs)(F)$ and $Y\in(\wt{\fm'}\cap\fs'_\rs)(F)$ have $M$-matching orbits. Let $U\in\fs_\rs(F)$ and $V\in\fs'_\rs(F)$ have matching orbits. Then we have the equality
$$ \gamma_{\psi}(\fh(F))^{-1} \kappa(X)\kappa(U)\hat{i}_M^G(\eta, X, U)=\gamma_{\psi}(\fh'(F))^{-1} \hat{i}_{M'}^{H'}(Y, V), $$
where $\gamma_\psi(\fh(F))$ and $\gamma_{\psi}(\fh'(F))$ are Weil constants (see Section \ref{ssec:def-weil-const}). 

	\item Let $X\in(\fm\cap\fs_\rs)(F)$ and $U\in\fs_\rs(F)$. If $\eta(X)\neq\eta(U)$, then 
$$ \hat{i}_M^G(\eta, X, U)=0. $$
\end{enumerate}
\end{thm}

This theorem generalises some of the main results in \cite{MR3414387} to the weighted context. As in {\it{loc. cit.}}, we use Waldspurger's global method on the endoscopic transfer \cite{MR1440722} to show (1) and a local method to show (2). To show (1), we define a notion of matching weighted orbital integrals (see Definition \ref{defMass}) and prove that this property commutes with Fourier transform under acceptable restriction (see Theorem \ref{thmcommute}). Then we may extract the relations between $\hat{i}_M^G(\eta, X, \cdot)$ and $\hat{i}_{M'}^{H'}(Y, \cdot)$ with the help of Labesse's lemma (see Lemma \ref{lemlabesse}). These steps are close to those in \cite{MR2164623}. However, there is an important distinction. While the weighted fundamental lemma for inner forms is tautological in {\it{loc. cit.}}, the vanishing condition of Lemma \ref{wfl} here is subtler. It makes the comparison of global trace formulae by Waldspurger's method, which itself is a simple case of the noninvariant comparison, even trickier. We translate our definition of matching orbits into the language of cohomology (see Sections \ref{seccohcri} and \ref{secmatllevi}) and use abelian Galois cohomology (see \cite{MR1401491, MR1695940}) to go through some technical difficulties. 

This paper is organised as follows. We introduce some notations and recall some preliminaries in Section \ref{secnotpre}. Then we define the notion of matching orbits and give a cohomological criterion in Section \ref{secmatorbi}. Our main results are stated in Section \ref{secsta}. The rest of the paper is devoted to the proof Proposition \ref{parcommute} by Waldspurger's global method. We recall limit formulae of $\hat{i}_M^G(\eta, X, \cdot)$ and $\hat{i}_{M'}^{H'}(Y, \cdot)$, the weighted fundamental lemma and an infinitesimal variant of Guo-Jacquet trace formulae respectively in Sections \ref{seclim}, \ref{secwfle} and \ref{secinfiGJ}. We explain the construction of test functions and the globalisation of local data respectively in Sections \ref{secconsr} and \ref{secapp}. These results are prepared for our final proof in Section \ref{secfinalproof}. 

\s{\textbf{Acknowledgement. }}I sincerely thank Pierre-Henri Chaudouard for generously sharing his insight on this problem and allowing me to include some of his ideas in this paper. This work would not have been completed without his constant support. In particular, he suggested the cohomological criterion of matching orbits and shared his notes with me. It allowed me to handle a more general setting motivated by the conjecture of Prasad and Takloo-Bighash. 

This is a completely revised version of the last part of my thesis \cite{li:tel-03226143} at the Université de Paris. Part of this article was revised when I was a postdoc at the Aix-Marseille Université and the Max Planck Institute for Mathematics. This work was supported by grants from Région Ile-de-France. The project leading to this publication has received funding from Excellence Initiative of Aix–Marseille University–A*MIDEX, a French ``Investissements d’Aveni'' programme. The author is grateful to Max Planck Institute for Mathematics in Bonn for its hospitality and financial support. 


\section{\textbf{Notation and preliminaries}}\label{secnotpre}

\subsection{Groups}

\subsubsection{}\label{ssec:gp_gen}

Let $F$ be a local field of characteristic zero or a number field. Denote by $\CO_F$ the ring of integers of $F$. Let $E$ be a quadratic extension of $F$. If $F$ is a local (resp. global) field, denote by $\eta$ the quadratic character of $F^\times/NE^\times$ (resp. $\BA^\times/F^\times$) attached to $E/F$, where $NE^\times=N_{E/F} E^\times$ denotes the norm of $E^\times$ in $F^\times$ (resp. $\BA=\BA_F$ denotes the ring of ad\`eles of $F$). 

Let $G$ be a (connected) reductive group over $F$. Denote by $\rk_F(G)$ the $F$-rank of $G$. Let $\wh{G}$ be the Langlands dual of $G$ which is a complex reductive group. For a topological group $H$, denote by $H^\circ$ the neutral component of $H$ and let $\pi_0(H):=H/H^\circ$. All algebraic groups (except $\wh{G}$) and varieties are assumed to be defined over $F$ until further notice. Denote by $G_\ad$ the adjoint group of $G$, by $G_\der$ the derived subgroup of $G$ and by $G_\scn$ the simple connected cover of $G_\der$. Denote by $Z_G$ the centre of $G$ and by $C_G:=G/G_\der=Z_G/Z_{G_\der}$ the cocentre of $G$. Fix an algebraic closure $\ov{F}$ of $F$. Let $\Gamma:=\Gal(\ov{F}/F)$. If $\Gamma$ acts on a set $S$, denote by $S^\Gamma$ the subset of $S$ consisting of all $\Gamma$-fixed points. For an $F$-variety $V$, we sometimes abuse notation and also write $V$ for $V(\ov{F})$ when there is no confusion. 

We use a minuscule Fraktur letter to denote the Lie algebra of its corresponding algebraic group. For example, we write $\fg:=\Lie(G)$. Denote by $\Ad$ the adjoint action of $G$ on itself or $\fg$. If $G$ acts on an $F$-variety $V$ and $X\in V(F)$, denote by $G_X$ the centraliser of $X$ in $G$. If $\fv$ is an $F$-subvariety of $\fg$, denote by $\fv_X$ the centraliser of $X\in\fg(F)$ in $\fv$. If $\theta$ is an automorphism on $G$, denote by $G^\theta$ the subgroup of $\theta$-fixed points of $G$. 

Fix a Levi $F$-factor $M_0$ of a minimal parabolic $F$-subgroup of $G$. By a Levi subgroup of $G$, we mean a Levi $F$-factor of some parabolic $F$-subgroup of $G$. For a semi-standard (namely containing $M_0$) Levi subgroup $M$ of $G$, denote by $\msf^G(M)$, $\msp^G(M)$ and $\msl^G(M)$ the sets of parabolic $F$-subgroups of $G$ containing $M$, parabolic $F$-subgroups of $G$ with Levi factor $M$ and Levi subgroups of $G$ containing $M$ respectively. For $P\in\msf^G(M_0)$, denote by $M_P$ the unique Levi factor containing $M_0$ and by $N_P$ the unipotent radical. Let $\ov{P}$ be the parabolic subgroup opposite to $P$. 

For $M\in\msl^G(M_0)$, define the Weyl group of $(G,M)$ by 
$$ W^G(M):=\Norm_{G(F)}(M)/M(F). $$
In particular, we also write $W_0^G:=W^G(M_0)$. For $M, L\in\msl^G(M_0)$, denote 
$$ \Tran_G(M,L):=\{w\in W_0^L\bs W_0^G: \Ad(w)(M)\subseteq L\}. $$

Denote by $A_G$ the maximal $F$-split central torus of $G$. Let $X(G)_F$ be the group of $F$-rational characters of $G$. Define the $\BR$-linear space 
$$ \fa_G:=\Hom_\BZ(X(G)_F,\BR), $$
whose dual space is denoted by $\fa_G^\ast$. Fix a scalar product on $\fa_{M_0}$ which is invariant under the action of $W_0^G$, from which we deduce Haar measures on all subspaces of $\fa_{M_0}$. Denote by $\fa_M^G$ the orthogonal complement of $\fa_G$ in $\fa_M$. 

Let $D$ a central division algebra over $F$. Denote by $\deg(D)$ the degree of $D$, i.e., $\dim_F(D)=\deg(D)^2$. Denote by $GL_{n, D}$ the reductive group over $F$ whose $F$-points are $GL_n(D)$. For $x\in \fg\fl_n(D)$, we write $\Nrd (x), \Trd (x)$ and $\Prd_x$ for its reduced norm, reduced trace and reduced characteristic polynomial respectively. If $D=F$, we also write them as $\det(x)$, $\tr(x)$ and $\chi_x$ respectively. 

\subsubsection{}\label{ssec:stdmaxcpt}

Now suppose that $F$ is a local field of characteristic zero. Denote by $|\cdot|_F$ the normalised absolute value on $F$. Define a homomorphism $H_G: G(F)\ra\fa_G$ by 
$$ \langle H_G(x),\chi\rangle=\log(|\chi(x)|_F) $$
for all $x\in G(F)$ and $\chi\in X(G)_F$. Fix a maximal compact subgroup $K=K_G$ of $G(F)$ which is admissible relative to $M_0$ in the sense of \cite[p. 9]{MR625344}. Unless otherwise stated (see Section \ref{ssec:wt_function}), we shall choose the standard maximal compact subgroup when $G(F)=GL_n(D)$, where $D$ is a central division algebra over a finite field extension of $F$. That is to say, if $F$ is non-archimedean, $K=GL_n(\CO_D)$ with $\CO_D$ being the ring of integers of $D$ (see \cite[p. 191]{MR1344916}), while if $F$ is archimedean, $K$ is the unitary group with respect to some hermitian form (see \cite[p. 199]{MR1344916}). For $P\in\msf^G(M_0)$, we may extend the function $H_{M_P}$ to a map $H_P: G(F)\ra\fa_{M_P}$ using the decomposition $G(F)=M_P(F)N_P(F)K$. 

Fix the Haar measure on $K$ such that $\vol(K)=1$. In particular, we emphasise that this convention applies to the case where $G$ is a torus. For $P\in\msf^G(M_0)$, fix a Haar measure on $N_P(F)$ such that
$$ \int_{N_P(F)} \exp (2\rho_{\ov{P}}(H_{\ov{P}}(n))) dn =1, $$
where $\rho_{\ov{P}}$ is the half of the sum of roots (with multiplicity) associated to $\ov{P}$. We equip $\fn_P(F)$ with the corresponding Haar measure via the exponential map. For $M\in\msl^G(M_0)$, there are compatible Haar measures on $G(F)$ and $M(F)$ in the sense of \cite[(1.1), p. 12]{MR1114210} such that for all $P\in\msp^G(M)$ and all continuous and compactly supported function $f$ on $G(F)$, we have the equality 
$$ \int_{G(F)} f(x) dx = \int_{M(F)\times N_P(F)\times K} f(mnk) dkdndm. $$
We shall choose such measures. 

Let $V$ be an $F$-linear space of finite dimension. If $F$ is non-archimedean, denote by $\CC_c^\infty(V)=\CS(V)$ the space of locally constant, compactly supported and complex-valued functions on $V$. If $F$ is archimedean, denote by $\CS(V)$ the space of Schwartz functions on $V$. For $f\in\CS(V)$, denote by $\Supp(f)$ its support. 

Fix a continuous and nontrivial unitary character $\psi: F\ra\BC^\times$. Let $\langle\cdot,\cdot\rangle$ be a non-degenerate symmetric bilinear form on $\fg(F)$ which is invariant under conjugation. Let $\fv$ be an $F$-linear subspace of $\fg(F)$, on which the restriction of $\langle\cdot,\cdot\rangle$ is non-degenerate; examples for such $\fv$ include $\fm(F)$ with $M\in\msl^G(M_0)$. It is equipped with the unique self-dual Haar measure with respect to $\psi(\langle\cdot,\cdot\rangle)$. For $f\in\CS(\fv)$, define its Fourier transform $\hat{f}\in\CS(\fv)$ by 
$$ \forall X\in \fv, \hat{f}(X):=\int_{\fv} f(Y) \psi(\langle X,Y \rangle) dY. $$

Let $M\in\msl^G(M_0)$ and $Q\in\msf^G(M)$. For $x\in G(F)$, we define 
$$ v_P(\lambda,x):=e^{-\lambda(H_P(x))}, \forall \lambda\in i\fa_M^\ast, P\in\msp^G(M). $$
By \cite[p. 40-41]{MR625344}, this is a $(G,M)$-family in the sense of \cite[p. 36]{MR625344}. By \cite[Lemma 6.2]{MR625344}, we obtain Arthur's weight function 
\begin{equation}\label{eqweifun}
 v_M^Q(x):=\lim_{\lambda\ra 0} \sum_{\{P\in\msp^G(M): P\subseteq Q\}} v_P(\lambda,x) \theta_P^Q(\lambda)^{-1}, \forall x\in G(F) 
\end{equation}
where $\theta_P^Q(\lambda)$ is defined in \cite[p.15]{MR625344}. 

Given an $F$-linear space of finite dimension equipped with a non-degenerate symmetric bilinear form $q(\cdot,\cdot)$ and a Haar measure, we denote by $\gamma_\psi(q)$ the Weil constant given in \cite[\S 14, Th\'{e}or\`{e}me 2]{MR0165033}. 


\subsection{Symmetric pairs}\label{secgensympai}

\subsubsection{}\label{ssec:def-rss}

Let $F$ be a local field of characteristic zero or a number field. A symmetric pair in the sense of \cite[Definition 7.1.1]{MR2553879} is a triple $(G,H,\theta)$ where $H\subseteq G$ are a pair of reductive groups, and $\theta$ is an involution of $G$ such that $H=G^\theta$. Let $\fs$ be the tangent space at the neutral element of the symmetric space $S:=G/H$. We shall always view $\fs$ as a subspace of $\fg$. Thus 
$$ \fs=\{X\in\fg:(d\theta)(X)=-X\}, $$
on which $H$ acts by the restriction of $\Ad$. By \cite[Lemma 7.1.9]{MR2553879}, there exists a $G$-invariant $\theta$-invariant non-degenerate symmetric bilinear form on $\fg$. 

An element $X\in\fs$ is said to be semi-simple if $\Ad(H)(X)$ is Zariski closed in $\fs$. If $F$ is a local field of characteristic zero, $X\in\fs(F)$ is semi-simple if and only if $\Ad(H(F))(X)$ is closed in $\fs(F)$ in the analytic topology by \cite[Fact A, p. 108-109]{MR1375304}. We say that an element $X\in\fs$ is regular if $H_X$ has minimal dimension. Denote by $\fs_\rs$ the subset of $\fs$ consisting of regular semi-simple elements in $\fs$. 

\subsubsection{}\label{ssec:def-weil-const}

Now suppose that $F$ is a local field of characteristic zero. A Cartan subspace of $\fs$ is defined as a maximal abelian subspace $\fc\subseteq\fs$ defined over $F$ consisting of semi-simple elements. Denote by $\mst^\fs$ the set of Cartan subspaces of $\fs$. Fix a (finite) set of representatives $\mst_0^\fs$ for $H(F)$-conjugacy classes in $\mst^\fs$. Let $\fc\in\mst^\fs$. We write $\fc_\reg:=\fc\cap\fs_\rs$. Denote by $T_\fc$ the centraliser of $\fc$ in $H$. Define the Weyl group 
\begin{equation}\label{eq:defweylcartan}
 W(H,\fc):=\Norm_{H(F)}(\fc)/T_\fc(F). 
\end{equation}
For $\fc_1, \fc_2\in\mst^\fs$, denote by $W(H,\fc_1,\fc_2)$ the set of isomorphisms from $\fc_1$ onto $\fc_2$ induced by $\Ad(x)$ for some $x\in H(F)$. In particular, $W(H,\fc,\fc)$ is nothing but $W(H,\fc)$ (viewed as a set). For $X\in \fc(F)$, define the Weyl discriminant factor 
$$ |D^\fs(X)|_F:=|\det(\ad(X)|_{\fh/\ft_\fc\oplus\fs/\fc})|_F^{1/2}. $$

Let $\langle\cdot,\cdot\rangle$ be a $G$-invariant $\theta$-invariant non-degenerate symmetric bilinear form on $\fg$. For any $F$-linear subspace $\fv$ of $\fg(F)$ such that the restriction of $\langle \cdot,\cdot\rangle$ on $\fv$ is non-degenerate, denote by $\gamma_\psi(\fv)$ the Weil constant associated to $\fv$ (see the end of Section \ref{ssec:stdmaxcpt}). Let $\fc\in\mst^\fs$. For $X,Y\in(\fc\cap\fs_\rs)(F)$, we define a bilinear form $q_{X,Y}$ on $\fh(F)/\ft_{\fc}(F)$ by 
$$ q_{X,Y}(Z,Z'):=\langle [Z,X],[Y,Z'] \rangle. $$
It is non-degenerate and symmetric and we have $q_{X,Y}=q_{Y,X}$. Write  
\begin{equation}\label{gammaXY}
\gamma_\psi(X,Y):=\gamma_\psi(q_{X,Y}). 
\end{equation}

\subsection{The case of $(G,H)$}\label{symmpair1}

\subsubsection{}

We refer to \cite{MR4424024, MR4681295} for more details of the following facts about the first symmetric pair. 

\subsubsection{}\label{ssec:def-omega-stable}

Let $F$ be a local field of characteristic zero or a number field. Let $G=G_n:=GL_{2n}$ and denote by $H=H_n:=GL_n\times GL_n$ its subgroup via diagonal embedding. In fact, $H$ is the subgroup of fixed points of the involution $\Ad(\omega_0)$ on $G$, where 
$$ \omega_0:=\mat(1_n,,,-1_n)\in G(F). $$
We also write $\fs_n:=\fs$. We shall embed $G$ into $\fg$ in the standard way. For an $F$-subvariety $\fv$ of $\fg$, we write $\fv^\times:=\fv\cap G$. From \cite[Proposition 4.2]{MR3414387}, we know $\fs_\rs\subseteq\fs^\times$ in this case. Let $\langle\cdot,\cdot\rangle$ be the non-degenerate symmetric bilinear form on $\fg(F)$ defined by
\begin{equation}\label{bilform1}
 \langle X,Y\rangle:=\tr(XY),\forall X,Y\in\fg(F), 
\end{equation}
which is invariant under the adjoint action of $G(F)$ and $\Ad(\omega_0)$. 

Let $M_0$ be the group of diagonal matrices in $G$. Set 
$$ \omega:=\mat(0,1_n,1_n,0)\in G(F). $$ 
For $P\in\msf^G(M_0)$, we say that $P$ is ``$\omega$-stable'' if $\omega\in P$. Denote by $\msf^{G,\omega}(M_0)$ the subset of $\msf^G(M_0)$ consisting of $\omega$-stable parabolic subgroups. For $M\in\msl^G(M_0)$, we say that $M$ is ``$\omega$-stable'' if $M=M_P$ for some $P\in\msf^{G,\omega}(M_0)$. Denote by $\msl^{G,\omega}(M_0)$ the subset in $\msl^G(M_0)$ consisting of $\omega$-stable Levi subgroups. Let $A_n$ be the group of diagonal matrices in $GL_n$. There is a bijection between $\msl^{GL_n}(A_n)$ and  $\msl^{G,\omega}(M_0)$ induced by 
\begin{equation}\label{eq:notationM_n}
 M_n\mapsto M=\mat(\fm_n,\fm_n,\fm_n,\fm_n)^\times. 
\end{equation}
Given $M\in\msl^{G,\omega}(M_0)$, we shall always write $M_n$ for the preimage of $M$ under this bijection. Notice that if $M\in\msl^{G,\omega}(M_0)$, for $L\in\msl^G(M)$ and $Q\in\msf^G(M)$, we have $L\in\msl^{G,\omega}(M_0)$ and $Q\in\msf^{G,\omega}(M_0)$. There is also a bijection between $\msf^{GL_n}(A_n)$ and  $\msf^{G,\omega}(M_0)$ induced by 
$$ P_n\mapsto P=\mat(\fp_n,\fp_n,\fp_n,\fp_n)^\times. $$ 
Given $P\in\msf^{G,\omega}(M_0)$, we shall always write $P_n$ for the preimage of $P$ under this bijection. For $\fc\in\mst^\fs$ and $L\in\msl^{G,\omega}(M_0)$, define 
\begin{equation}\label{eq:weylcarlevi1}
 W(H,\fc,\fl\cap\fs):=\bigsqcup_{\fc_2\in\mst_0^{\fl\cap\fs}} W(H,\fc,\fc_2). 
\end{equation}

Let $M\in\msl^{G,\omega}(M_0)$. We say that an element $X\in(\fm\cap\fs_\rs)(F)$ (resp. a Cartan subspace $\fc\subseteq\fm\cap\fs$) is $M$-elliptic if $A_M$ is the maximal $F$-split torus in the torus $H_X$ (resp. in $T_\fc$). Denote by $(\fm\cap\fs_\rs)(F)_\el$ the set of $M$-elliptic elements in $(\fm\cap\fs_\rs)(F)$. Write $M_H:=M\cap H$. Denote by $\Gamma_\el((\fm\cap\fs_\rs)(F))$ the set of $M_H(F)$-conjugacy classes in $(\fm\cap\fs_\rs)(F)_\el$. Denote by $\mst_\el^{\fm\cap\fs}$ the subset of $\mst_0^{\fm\cap\fs}$ consisting of representatives that are $M$-elliptic Cartan subspaces. For $X\in\fs_\rs(F)$, $X\in\fs_\rs(F)_\el$ if and only if $\chi_X(\lambda)=p(\lambda^2)$ for some irreducible polynomial $p(\lambda)\in F[\lambda]$ of degree $n$. 

\subsubsection{}

Now suppose that $F$ is a local field of characteristic zero. Let $P\in\msf^{G,\omega}(M_0)$. Then 
$$ \fm_P=\mat(\fm_n, \fm_n, \fm_n, \fm_n) \text{ and } \fn_P=\mat(\fn_n, \fn_n, \fn_n, \fn_n), $$
where we denote $M_n:=M_{P_n}$ and $N_n:=N_{P_n}$. We shall choose the same Haar measure for any of the four copies in $\fm_P(F)$ or $\fn_P(F)$ under these identifications. For $f\in\CS(\fs(F))$, we define a function $f_P^\eta\in\CS((\fm_P\cap\fs)(F))$ by
$$ f_P^\eta(Z):=\int_{K_H\times(\fn_P\cap\fs)(F)} f(\Ad(k^{-1})(Z+U)) \eta(\det(k)) dUdk $$
for all $Z\in(\fm_P\cap\fs)(F)$. Recall that $(\hat{f})_P^\eta=\wh{f_P^\eta}$, and we shall denote it by $\hat{f}_P^\eta$ without confusion. 

Let $M\in\msl^{G,\omega}(M_0)$ and $Q\in\msf^G(M)$. For $f\in \CC_c^\infty(\fs(F))$ and $X\in (\fm\cap\fs_\rs)(F)$, define the weighted orbital integral 
\begin{equation}\label{defnoninvwoi1}
 J_M^Q(\eta, X, f):=|D^\fs(X)|_F^{1/2} \int_{H_X(F)\bs H(F)} f(\Ad(x^{-1})(X)) \eta(\det(x)) v_M^Q(x) dx. 
\end{equation}
For $X=\mat(0,A,B,0)\in\fs_\rs(F)$, we define a transfer factor (see \cite[Definition 5.7]{MR3414387})
\begin{equation}\label{eq:def_trans_fac}
 \kappa(X):=\eta(\det(A)) 
\end{equation}
which satisfies 
$$ \kappa(\Ad(x^{-1})(X))=\eta(\det(x))\kappa(X) $$
for all $x\in H(F)$. It follows that the function $\kappa(\cdot)J_M^Q(\eta, \cdot, f)$ is constant on $\Ad(M_H(F))(X)$. 

\subsubsection{}

Now suppose additionally that $F$ is non-archimedean. Let $M\in\msl^{G,\omega}(M_0)$ and $X\in (\fm\cap\fs_\rs)(F)$. In \cite[\S8.1]{MR4681295}, we deduce from $J_M^G(\eta,X,\cdot)$ an $(H,\eta)$-invariant distribution $I_M^G(\eta, X, \cdot)$ on $\fs(F)$. By \cite[Propositions 7.2 and 8.1]{MR4681295}, there are unique locally constant functions $\hat{j}_M^G(\eta,X,\cdot)$ and $\hat{i}_M^G(\eta,X,\cdot)$ on $\fs_\rs(F)$ representing their Fourier transforms $\hat{J}_M^G(\eta,X,\cdot)$ and $\hat{I}_M^G(\eta,X,\cdot)$ respectively. That is to say, for all $f\in\CC_c^\infty(\fs(F))$, we have 
$$ \hat{J}_M^G(\eta, X, f):=J_M^G(\eta, X, \hat{f})=\int_{\fs(F)} f(U) \hat{j}_M^G(\eta, X, U) |D^\fs(U)|_F^{-1/2} dU $$
and
$$ \hat{I}_M^G(\eta, X, f):=I_M^G(\eta, X, \hat{f})=\int_{\fs(F)} f(U) \hat{i}_M^G(\eta, X, U) |D^\fs(U)|_F^{-1/2} dU. $$

\subsection{The case of $(G',H')$}\label{symmpair2}

\subsubsection{}

We refer to \cite{MR4350885, MR4681295} for more details of the following facts about the second symmetric pair. 

\subsubsection{}\label{ssec:sym2-I-II}

Let $F$ be a local field of characteristic zero or a number field. Let $E$ be a quadratic extension of $F$. Let $\fg'$ be a central simple algebra over $F$ with a fixed $F$-algebra embedding $E\hookrightarrow\fg'(F)$. Let $\fh':=\Cent_{\fg'}(E)$ be the centraliser of $E$ in $\fg'$. Then $\fh'(F)$ is a central simple algebra over $E$ by the double centraliser theorem. Denote by $G':={\fg'}^\times$ (resp. $H':={\fh'}^\times$) the group of invertible elements in $\fg'$ (resp. $\fh'$). Let $\alpha\in E\bs F$ be such that $\alpha^2\in F$, so $E=F(\alpha)$. In fact, $H'$ is the subgroup of fixed points of the involution $\Ad(\alpha)$ on $G'$. Denote by $\fs'$ the corresponding tangent space of $G'/H'$ at the neutral element. For an $F$-subvariety $\fv'\subseteq\fg'$, we write ${\fv'}^\times:=\fv'\cap G'$. By the base change to $\ov{F}$, we see $\fs'_\rs\subseteq{\fs'}^\times$ in this case. Let $\langle\cdot,\cdot\rangle$ be the non-degenerate symmetric bilinear form on $\fg'(F)$ defined by
\begin{equation}\label{bilform2}
 \langle X,Y\rangle:=\Trd(XY),\forall X,Y\in\fg'(F), 
\end{equation}
which is invariant under the adjoint action of $G'(F)$ and $\Ad(\alpha)$. 

By the Wedderburn-Artin theorem, $G'$ is isomorphic to $GL_{r,D}$ where $r:=\rk_F(G')$ for some central division algebra $D$ over $F$ such that $r\deg(D)$ is even. By the Noether-Skolem theorem, up to conjugation by $G'(F)$, the embedding $H'\hookrightarrow G'$ is isomorphic to one of the two cases below (see \cite[\S2.1 and \S3.1]{MR3958071} and \cite[\S3.4]{MR4350885}). 

\textbf{Case I}: if there is an embedding $E\ra D$ as $F$-algebras, then 
$$ (G',H')\simeq(GL_{r,D},\Res_{E/F} GL_{r,D'}), $$
where $D':=\Cent_D(E)$ is a central division algebra over $E$ of degree $\frac{\deg(D)}{2}$. 
Let $M'_0\simeq (\Res_{E/F} \BG_{m,D'})^r$ (resp. $M'_{\wt{0}}\simeq (\BG_{m,D})^r$) be the subgroup of diagonal elements in $H'$ (resp. $G'$). For $M'\in\msl^{H'}(M'_0)$, denote  by $\wt{M'}$ the unique element in $\msl^{G'}(M'_{\wt{0}})$ such that $\wt{M'}\cap H'=M'$. The map $M'\mapsto\wt{M'}$ induces a bijection between $\msl^{H'}(M'_0)$ and $\msl^{G'}(M'_{\wt{0}})$, and we can identify $A_{M'}$ with $A_{\wt{M'}}$. 

\textbf{Case II}: if there is no embedding $E\ra D$ as $F$-algebras, then 
$$ (G',H')\simeq(GL_{r,D},\Res_{E/F} GL_{\frac{r}{2},D\otimes_F E}), $$
where $D\otimes_F E$ is a central division algebra over $E$ of degree $\deg(D)$. 
Let $M'_0\simeq (\Res_{E/F} \BG_{m,D\otimes_F E})^\frac{r}{2}$ (resp. $M'_{\wt{0}}\simeq (\BG_{m,D})^r$) be the subgroup of diagonal elements in $H'$ (resp. $G'$). Denote by $\msl^{G'}(M'_{\wt{0}}, M'_0)$ the subset of elements in $\msl^{G'}(M'_{\wt{0}})$ containing $M'_0$. For $M'\in\msl^{H'}(M'_0)$, denote  by $\wt{M'}$ the unique element in $\msl^{G'}(M'_{\wt{0}}, M'_0)$ such that $\wt{M'}\cap H'=M'$. The map $M'\mapsto\wt{M'}$ induces a bijection between $\msl^{H'}(M'_0)$ and $\msl^{G'}(M'_{\wt{0}}, M'_0)$, and we can identify $A_{M'}$ with $A_{\wt{M'}}$. 

If $\rk_F(G')=r$, we also write $G'_r:=G'$, $H'_r:=H'$ and $\fs'_r:=\fs'$. Notice that $\rk_F(H'_r)=r$ in \textbf{Case I} (resp. $=\frac{r}{2}$ in \textbf{Case II}). 
For $P'\in\msf^{H'}(M'_0)$, denote by $\wt{P'}$ the unique element in $\msf^{G'}(\wt{M'_0})$ such that $\wt{P'}\cap H'=P'$. The map $P'\mapsto\wt{P'}$ induces a bijection between $\msf^{H'}(M'_0)$ and $\msf^{G'}(\wt{M'_0})$ in both of \textbf{Case I} and \textbf{Case II}. Let $\tau\in D^\times$ in \textbf{Case I} (resp. $\tau\in GL_2(D)$ in \textbf{Case II}) be an element such that $\Ad(\alpha)(\tau)=-\tau$. Let $P'\in\msf^{H'}(M'_0)$. By \cite[Proposition 3.12]{MR4350885}, we have 
$$ \fm_{\wt{P'}}\cap\fs'=\fm_{P'}\tau=\tau\fm_{P'} $$
and 
$$ \fn_{\wt{P'}}\cap\fs'=\fn_{P'}\tau=\tau\fn_{P'}. $$
For $\fc'\in\mst^{\fs'}$ and $L'\in\msl^{H'}(M'_0)$, define 
\begin{equation}\label{eq:weylcarlevi2}
 W(H',\fc',\wt{\fl'}\cap\fs'):=\bigsqcup_{\fc'_2\in\mst_0^{\wt{\fl'}\cap\fs'}} W(H',\fc',\fc'_2). 
\end{equation}

Let $M'\in\msl^{H'}(M'_0)$. We say that an element $Y\in(\wt{\fm'}\cap\fs'_\rs)(F)$ (resp. a Cartan subspace $\fc'\subseteq\wt{\fm'}\cap\fs'$) is $M'$-elliptic if $A_{M'}$ is the maximal $F$-split torus in the torus $H'_Y$ (resp. in $T_{\fc'}$). Denote by $(\wt{\fm'}\cap\fs'_\rs)(F)_\el$ the set of $M'$-elliptic elements in $(\wt{\fm'}\cap\fs'_\rs)(F)$. Denote by $\Gamma_\el((\wt{\fm'}\cap\fs'_\rs)(F))$ the set of $M'(F)$-conjugacy classes in $(\wt{\fm'}\cap\fs'_\rs)(F)_\el$. Denote by $\mst_\el^{\wt{\fm'}\cap\fs'}$ the subset of $\mst_0^{\wt{\fm'}\cap\fs'}$ consisting of representatives that are $M'$-elliptic Cartan subspaces. 

\begin{lem}\label{lemsym2ell}
For $Y\in\fs'_\rs(F)$, $Y\in\fs'_\rs(F)_\el$ if and only if $\Prd_Y(\lambda)=p(\lambda^2)$ for some irreducible polynomial $p(\lambda)\in F[\lambda]$ of degree $\frac{r\deg(D)}{2}$. 
\end{lem}

\begin{proof}
Let $n=r\deg(D)$ and define $\fs$ as in Section \ref{symmpair1}. By Lemma \ref{lem74} below, if $Y\in\fs'_\rs(F)_\el$, then there exists $X\in\fs_\rs(F)_\el$ such that $X\da Y$ in the sense of Definition \ref{defbyinv}. Then the assertion about $\Prd_Y$ is a consequence of its analogue for $\fs$. 

By \cite[Lemma 9.1]{MR4350885}, if $Y\notin\fs'_\rs(F)_\el$, then $Y$ is $H'(F)$-conjugate to an element $Z\in(\wt{\fm'}\cap\fs'_\rs)(F)$ for some $M'\in\msl^{H'}(M'_0), M'\subsetneq H'$. Thus $\Prd_Y=\Prd_Z$ cannot satisfy the desired property. 
\end{proof}

\begin{coro}\label{corsym2ell}
Let $Y\in\fs'_\rs(F)$ with $\Prd_Y(\lambda)=\prod\limits_{i=1}^m p_i(\lambda^2)$ where $p_i(\lambda)\in F[\lambda]$ is irreducible for $1\leq i\leq m$. Then for each $i$, 
\begin{enumerate}
	\item there exists $r_i\in\BZ_{>0}$ such that $\deg(p_i)=\frac{r_i\deg(D)}{2}$; 

	\item there exists $Y_i\in\fs'_{r_i,\rs}(F)_\el$ such that $\Prd_{Y_i}(\lambda)=p_i(\lambda^2)$. 
\end{enumerate}
\end{coro}

\begin{proof}
By \cite[Lemma 9.1]{MR4350885}, there exists $M'\in\msl^{H'}(M'_0)$ and $Z\in(\wt{\fm'}\cap\fs'_\rs)(F)_\el$ which is $H'(F)$-conjugate to $Y$. The assertion results from $\Prd_Y=\Prd_Z$ and Lemma \ref{lemsym2ell}. 
\end{proof}

\subsubsection{}

Now suppose that $F$ is a local field of characteristic zero. Let $P'\in\msf^{H'}(M'_0)$. We shall choose the same Haar measures on $(\fm_{\wt{P'}}\cap\fs')(F)$ and $(\fn_{\wt{P'}}\cap\fs')(F)$ as those on $\fm_{P'}(F)$ and $\fn_{P'}(F)$ respectively using above identifications induced by $\tau$. Such Haar measures depend on the choice of $\tau$. For $f'\in\CC_c^\infty(\fs'(F))$, we define a function $f'_{P'}\in\CC_c^\infty((\fm_{\wt{P'}}\cap\fs')(F))$ by
$$ f'_{P'}(Z):=\int_{K_{H'}\times(\fn_{\wt{P'}}\cap\fs')(F)} f'(\Ad(k^{-1})(Z+U)) dUdk $$
for all $Z\in(\fm_{\wt{P'}}\cap\fs')(F)$. Recall that $(\hat{f'})_{P'}=\wh{f'_{P'}}$, and we shall denote it by $\hat{f}'_{P'}$ without confusion. 

Let $M'\in\msl^{H'}(M'_0)$ and $Q'\in\msf^{H'}(M')$. For $f'\in \CC_c^\infty(\fs'(F))$ and $Y\in(\wt{\fm'}\cap\fs'_\rs)(F)$, define the weighted orbital integral
\begin{equation}\label{defnoninvwoi2}
 J_{M'}^{Q'}(Y, f'):=|D^{\fs'}(Y)|_F^{1/2} \int_{H'_Y(F)\bs H'(F)} f'(\Ad(x^{-1})(Y)) v_{M'}^{Q'}(x) dx. 
\end{equation}

\subsubsection{}

Now suppose additionally that $F$ is non-archimedean. Let $M'\in\msl^{H'}(M'_0)$ and $Y\in(\wt{\fm'}\cap\fs'_\rs)(F)$. In \cite[\S8.2]{MR4681295}, we deduce from $J_{M'}^{H'}(Y,\cdot)$ an $H'$-invariant distribution $I_{M'}^{H'}(Y, \cdot)$ on $\fs'(F)$. By \cite[Propositions 7.10 and 8.6]{MR4681295}, there are unique locally constant functions $\hat{j}_{M'}^{H'}(Y,\cdot)$ and $\hat{i}_{M'}^{H'}(Y,\cdot)$ on $\fs'_\rs(F)$ representing their Fourier transforms $\hat{J}_{M'}^{H'}(Y,\cdot)$ and $\hat{I}_{M'}^{H'}(Y,\cdot)$ respectively. That is to say, for all $f'\in\CC_c^\infty(\fs'(F))$, we have 
$$ \hat{J}_{M'}^{H'}(Y, f'):=J_{M'}^{H'}(Y, \hat{f'})=\int_{\fs'(F)} f'(V) \hat{j}_{M'}^{H'}(Y, V) |D^{\fs'}(V)|_F^{-1/2} dV $$
and
$$ \hat{I}_{M'}^{H'}(Y, f'):=I_{M'}^{H'}(Y, \hat{f'})=\int_{\fs'(F)} f'(V) \hat{i}_{M'}^{H'}(Y, V) |D^{\fs'}(V)|_F^{-1/2} dV. $$

\subsection{Labesse's lemma}\label{ssec:lab_lem}

Let $F$ be a non-archimedean local field of characteristic zero and $E$ be a quadratic extension of $F$. Let $\eta$ be the quadratic character of $F^\times/NE^\times$ attached to $E/F$. 

For $f\in\CC_c^\infty(\fs(F))$, we say that the weighted orbital integrals of $f$ vanish for nontrivial weights if for all $M\in\msl^{G,\omega}(M_0), Q\in\msf^G(M)-\msp^G(M)$ and $X\in(\fm\cap\fs_\rs)(F)$, we have 
$$ J_M^Q(\eta, X, f)=0. $$ 
Suppose that $f$ satisfies this condition. By definition \cite[(8.1.1)]{MR4681295}, for all $M\in\msl^{G,\omega}(M_0)$ and $X\in(\fm\cap\fs_\rs)(F)$, we have the equality
$$ J_M^G(\eta, X, \hat{f})=I_M^G(\eta, X, \hat{f}). $$

For $f'\in\CC_c^\infty(\fs'(F))$, we say that the weighted orbital integrals of $f'$ vanish for nontrivial weights if for all $M'\in\msl^{H'}(M'_0), Q'\in\msf^{H'}(M')-\msp^{H'}(M')$ and $Y\in(\wt{\fm'}\cap\fs'_\rs)(F)$, we have 
$$ J_{M'}^{Q'}(Y, f')=0. $$ 
Suppose that $f'$ satisfies this condition. By definition \cite[(8.2.1)]{MR4681295}, for all $M'\in\msl^{H'}(M'_0)$ and $Y\in(\wt{\fm}\cap\fs'_\rs)(F)$, we have the equality
$$ J_{M'}^{H'}(Y, \hat{f'})=I_{M'}^{H'}(Y, \hat{f'}). $$

\begin{lem}\label{lemlabesse}
\leavevmode
\begin{enumerate}
	\item Let $f\in\CC_c^\infty(\fs(F))$ be such that $\Supp(f)\subseteq\fs_\rs(F)$. Then there exists $\phi\in\CC_c^\infty(\fs(F))$ such that

\begin{enumerate}
	\item $\Supp(\phi)\subseteq\Ad(H(F))(\Supp(f))$; 

	\item for all $X\in\fs_\rs(F)$, $J_G^G(\eta, X, \phi)=J_G^G(\eta, X, f)$; 

	\item the weighted orbital integrals of $\phi$ vanish for nontrivial weights. 
\end{enumerate}

	\item Let $f'\in\CC_c^\infty(\fs'(F))$ be such that $\Supp(f')\subseteq\fs'_\rs(F)$. Then there exists $\phi'\in\CC_c^\infty(\fs'(F))$ such that

\begin{enumerate}
	\item $\Supp(\phi')\subseteq\Ad(H'(F))(\Supp(f'))$; 

	\item for all $Y\in\fs'_\rs(F)$, $J_{H'}^{H'}(Y, \phi')=J_{H'}^{H'}(Y, f')$; 

	\item the weighted orbital integrals of $\phi'$ vanish for nontrivial weights. 
\end{enumerate}
\end{enumerate}
\end{lem}

\begin{proof}
This is a relative and infinitesimal analogue of Labesse's lemma \cite[Lemma I.7.1]{MR1339717}. We can imitate the proof in {\it{loc. cit.}} with the preparation in \cite{MR4681295} except that we have to take care of the character $\eta$. For completeness, we include a proof for the case of $(G,H)$ here. 

Let $M\in\msl^{G,\omega}(M_0), Q\in\msf^G(M)$ and $X\in(\fm\cap\fs_\rs)(F)$. There exists $L\in\msl^{G,\omega}(M_0), L\subseteq M$ and $y\in M_H(F)$ such that $Z:=\Ad(y^{-1})(X)\in(\fl\cap\fs_\rs)(F)_\el$. For $\phi\in\CC_c^\infty(\fs(F))$, by the descent formula (cf. \cite[Proposition 4.1.(5)]{MR4681295} and \cite[Lemma I.6.2]{MR1339717}), we have 
$$ J_M^Q(\eta, X, \phi)=\eta(\det(y)) J_M^Q(\eta, Z, \phi)=\eta(\det(y)) \sum_{L'\in\msl^{M_Q}(L)} d_L^{M_Q}(M, L') J_L^{Q'}(\eta, Z, \phi) $$
where $d_L^{M_Q}(M, L')\in\BR_{\geq0}$ is defined in \cite[p.356]{MR928262} and $Q'\in\msp^G(L')$ is given in \cite[Lemma I.1.2]{MR1339717}. In particular, both of $L'$ and $Q'$ are $\omega$-stable. It is known that $d_L^{M_Q}(M, L')\neq0$ if and only if $\fa_L^M\oplus\fa_L^{L'}=\fa_L^{M_Q}$. In this case, if $Q\notin\msp^G(M)$, then $L\neq L'=M_{Q'}$. Therefore, to check the condition (c), it suffices to check that $J_M^Q(\eta, X, \phi)=0$ for all $M\in\msl^{G,\omega}(M_0), Q\in\msf^G(M)-\msp^G(M)$ and $X\in(\fm\cap\fs_\rs)(F)_\el$. 

Let $\fc\subseteq\fm\cap\fs$ be an $M$-elliptic Cartan subspace. It is known that $v_M^Q$ vanishes on $M_H(F)K_H$ unless $Q\in\msp^G(M)$, in which case $v_M^Q=1$. By \cite[Lemma 7.9.(1)]{MR4681295}, we have 
$$ \Norm_{H(F)}(\fc)\subseteq M_H(F) W_0^H, $$
so $v_M^Q$ vanishes on $\Norm_{H(F)}(\fc) K_H=T_\fc(F) W(H,\fc) K_H$ unless $Q\in\msp^G(M)$. Hence, there exists a compactly supported function $\alpha_\fc$ on $T_\fc(F)\bs\Norm_{H(F)}(\fc) K_H\subseteq T_\fc(F)\bs H(F)$ which is left $W(H,\fc)$-invariant and right $K_H$-invariant such that 
\begin{enumerate}[(i)]
	\item we have 
	$$ \int_{T_\fc(F)\bs H(F)} \alpha_\fc(x) dx=1; $$
	
	\item for all $Q\in\msf^G(M)-\msp^G(M)$, 
	$$ \int_{T_\fc(F)\bs H(F)} \alpha_\fc(x) v_M^Q(x) dx=0. $$
\end{enumerate}
If $\fc'=\Ad(h)(\fc)\subseteq\fm'\cap\fs$ is an $M'$-elliptic Cartan subspace for some $h\in H(F)$ and $M'\in\msl^{G,\omega}(M_0)$, we can and we shall choose $\alpha_{\fc'}$ and $\alpha_\fc$ in a compatible way that $\alpha_{\fc'}(hx)=\alpha_\fc(x)$. 

Consider the map 
$$ T_\fc(F)\bs H(F) \times \fc_\reg(F) \ra \fs_\rs(F) $$
defined by $(x, X)\mapsto\Ad(x^{-1})(X)$, which induces an étale covering from its source onto its image whose fibres are orbits under $W(H,\fc)$. From $\Supp(f)\subseteq\fs_\rs(F)$ and \cite[Proposition 4.1.(2)]{MR4681295}, we know that $\kappa(\cdot) J_G^G(\eta, \cdot, f)$ is a $W(H,\fc)$-invariant, locally constant and compactly supported function on $\fc_\reg(F)$. Then the function $\varphi_\fc: T_\fc(F)\bs H(F) \times \fc_\reg(F)\ra\BC$ defined by 
$$ \varphi_\fc(x,X):=\alpha_\fc(x) \kappa(X) \int_{T_\fc(F)\bs H(F)} f(\Ad(z^{-1})(X)) \eta(\det(z)) dz $$ 
is $W(H,\fc)$-invariant. If $\fc'=\Ad(h)(\fc)\subseteq\fm'\cap\fs$ is an $M'$-elliptic Cartan subspace for some $h\in H(F)$ and $M'\in\msl^{G,\omega}(M_0)$, we see that 
$$ \varphi_{\fc'}(hx, \Ad(h)(X))=\varphi_\fc(x,X). $$
Therefore, there exists $\varphi\in\CC_c^\infty(\fs(F))$ supported on $\fs_\rs(F)$ and whose restriction to $\Ad(H(F))(\fc_\reg(F))$ is such that 
$$ \varphi(\Ad(x^{-1})(X))=\varphi_\fc(x,X), $$
and this function $\varphi$ is independent of the choice of $\fc$ in its $H(F)$-conjugacy class. 

Let $\phi:=\kappa\varphi\in\CC_c^\infty(\fs(F))$ which is supported on $\fs_\rs(F)$. It is clear from the construction that $\phi$ satisfies the condition (a). Now, let us check (b). Up to $H(F)$-conjugation, it suffices to consider $X\in\fc_\reg(F)$ where $\fc\subseteq\fm\cap\fs$ is an $M$-elliptic Cartan subspace with $M\in\msl^{G,\omega}(M_0)$. In this case, we have 
\[\begin{split}
 J_G^G(\eta,X,\phi)=|D^\fs(X)|_F^{1/2} \kappa(X) \int_{T_\fc(F)\bs H(F)} \varphi(\Ad(x^{-1})(X)) dx  \\ 
 =|D^\fs(X)|_F^{1/2} \int_{T_\fc(F)\bs H(F)} \alpha_\fc(x) dx \int_{T_\fc(F)\bs H(F)} f(\Ad(z^{-1})(X)) \eta(\det(z)) dz = J_G^G(\eta,X,f), 
\end{split}\]
which verifies the condition (b). Finally, we check (c). From the discussion at the beginning, it suffices to consider $M\in\msl^{G,\omega}(M_0), Q\in\msf^G(M)-\msp^G(M)$ and $X\in(\fm\cap\fs_\rs)(F)_\el$. Let $\fc:=\fs_X\subseteq\fm\cap\fs$, which is an $M$-elliptic Cartan subspace. We have 
\[\begin{split}
 J_M^Q(\eta, X, \phi)=|D^\fs(X)|_F^{1/2} \kappa(X) \int_{T_\fc(F)\bs H(F)} \varphi(\Ad(x^{-1})(X)) v_M^Q(x) dx  \\ 
 =|D^\fs(X)|_F^{1/2} \int_{T_\fc(F)\bs H(F)} \alpha_\fc(x) v_M^Q(x) dx \int_{T_\fc(F)\bs H(F)} f(\Ad(z^{-1})(X)) \eta(\det(z)) dz =0, 
\end{split}\]
which verifies the condition (c). 
\end{proof}

\begin{lem}\label{lem:Lab_I.6.4}
\leavevmode
\begin{enumerate}
	\item Let $f\in\CC_c^\infty(\fs(F))$ and $Q\in\msf^{G,\omega}(M_0)$. Suppose that the weighted orbital integrals of $f$ vanish for nontrivial weights. Then for all $L\in\msl^{M_Q,\omega}(M_0)$, $P\in\msf^{M_Q}(L)-\msp^{M_Q}(L)$ and $X\in(\fl\cap\fs_\rs)(F)$, we have 
	$$ J_L^P(\eta, X, f_Q^\eta)=0. $$
	
	\item Let $f'\in\CC_c^\infty(\fs'(F))$ and $Q'\in\msf^{H'}(M'_0)$. Suppose that the weighted orbital integrals of $f'$ vanish for nontrivial weights. Then for all $L'\in\msl^{M_{Q'}}(M'_0)$, $P'\in\msf^{M_{Q'}}(L')-\msp^{M_{Q'}}(L')$ and $Y\in(\wt{\fl'}\cap\fs'_\rs)(F)$, we have 
	$$ J_{L'}^{P'}(Y, f'_{Q'})=0. $$
\end{enumerate}
\end{lem}

\begin{proof}
This is a consequence of a relative and infinitesimal analogue of \cite[Lemma I.6.4.(i)]{MR1339717}. For completeness, we include a proof for the case of $(G,H)$ here. 

By \cite[Proposition 4.1.(4)]{MR4681295}, we have 
\begin{equation}\label{eq:pardes1}
 J_L^P(\eta, X, f_Q^\eta)=J_L^{M_P}(\eta, X, (f_Q^\eta)_P^\eta). 
\end{equation}
Note that 
\[\begin{split}
	(f_Q^\eta)_P^\eta=\int_{K_{M_{Q_H}}\times (\fn_P\cap\fs)(F)} f_Q^\eta(\Ad(g^{-1})(Z+V))\eta(\det(g)) dV dg \\ 
	=\int_{K_{M_{Q_H}}\times (\fn_P\cap\fs)(F)} \int_{K_H\times(\fn_Q\cap\fs)(F)} f(\Ad(k^{-1})(\Ad(g^{-1})(Z+V)+U)) \eta(\det(k)) \eta(\det(g)) dUdk dV dg.
\end{split}\]
Using changes of variables $U\mapsto \Ad(g^{-1})(U)$ and $k\mapsto g^{-1}k$, since $\vol(K_{M_{Q_H}})=1$, we obtain 
$$ (f_Q^\eta)_P^\eta=f_{PN_Q}^\eta, $$
where $PN_Q\in\msf^{G,\omega}(M_0)\cap\msp^G(M_P)$. By \cite[Proposition 4.1.(4)]{MR4681295} again, we have 
\begin{equation}\label{eq:pardes2}
 J_L^{M_P}(\eta, X, (f_Q^\eta)_P^\eta)=J_L^{M_P}(\eta, X, f_{PN_Q}^\eta)=J_L^{PN_Q}(\eta, X, f). 
\end{equation}
For $P\notin\msp^{M_Q}(L)$, we have $PN_Q\notin\msp^G(L)$. Thus $J_L^{PN_Q}(\eta, X, f)=0$ by our assumption on $f$. 
\end{proof}


\section{\textbf{Matching of orbits}}\label{secmatorbi}

Let $F$ be a local field of characteristic zero or a number field. Let $E$ be a quadratic extension of $F$. Assume that $\dim_{\ov{F}}(G)=\dim_{\ov{F}}(G')$, i.e., $2n=r\deg(D)$. 

\subsection{Definition by invariants}

\subsubsection{Matching orbits}\label{ssec:matorb}

Denote by $\bfA^n$ the affine space over $F$ of dimension $n$. By \cite[Proposition 3.3]{MR4424024}, the map 
$$ \fs\ra\bfA^n, X\mapsto\chi_X $$
defines a categorical quotient $\fs//H$ over $F$, where $\chi_X$ denotes the characteristic polynomial of $X\in\fg$. By \cite[Proposition 3.5]{MR4350885}, the map 
$$ \fs'\ra\bfA^n, Y\mapsto\Prd_Y $$
defines a categorical quotient $\fs'//{H'}$ over $F$, where $\Prd_Y$ denotes the reduced characteristic polynomial of $Y\in\fg'$. Therefore, we can identify 
$$ \fs//H\simeq\bfA^n\simeq\fs'//{H'}. $$
By \cite[Proposition 3.3]{MR4350885}, it induces an injection from the set of $H'(F)$-orbits in $\fs'_\rs(F)$ into the set of $H(F)$-orbits in $\fs_\rs(F)$. 

\begin{defn}\label{defbyinv}
Let $X\in\fs_\rs(F)$ and $Y\in\fs'_\rs(F)$. If $\chi_X=\Prd_Y$, we say that $X$ and $Y$ have matching orbits and write $X\da Y$. For $X\in\fs_\rs(F)$, if there is an element $Y\in\fs'_\rs(F)$ such that $X\da Y$, we also say that $X$ comes from $\fs'_\rs(F)$. 
\end{defn}

\subsubsection{Matching Levi subgroups}\label{ssec:mat_Levi}

There is an injection $M'\mapsto M$ from $\msl^{H'}(M'_0)$ into $\msl^{G,\omega}(M_0)$ such that the restriction of $G'_{\ov{F}}\simeq G_{\ov{F}}$ induces an isomorphism $\wt{M'}_{\ov{F}}\simeq M_{\ov{F}}$, where $\wt{M'}\in\msl^{G'}(M'_{\wt{0}})$ is associated to $M'$ as in Section \ref{ssec:sym2-I-II}. Let $M'\in\msl^{H'}(M'_0)$. We shall always denote by $M$ the image of $M'$ under this injection, and say that $M'$ and $M$ are matching Levi subgroups. There are compatible isomorphisms 
\begin{equation}\label{eq:M-fact}
 (M,M_H, \fm\cap\fs)\simeq \prod_{i=1}^\ell (G_{n_i},H_{n_i}, \fs_{n_i}) 
\end{equation}
and 
\begin{equation}\label{eq:M'-fact}
 (\wt{M'},M', \wt{\fm'}\cap\fs')\simeq \prod_{i=1}^\ell (G'_{r_i},H'_{r_i},\fs'_{r_i})
\end{equation}
such that the restriction of $G'_{\ov{F}}\simeq G_{\ov{F}}$ induces an isomorphism 
$$ (G_{n_i},H_{n_i}, \fs_{n_i})_{\ov{F}}\simeq (G'_{r_i},H'_{r_i},\fs'_{r_i})_{\ov{F}} $$
for each $1\leq i\leq \ell$. Denote 
\begin{equation}\label{eq:def_e_M'}
 e_{M'}:=((-1)^{r_i})_{1\leq i\leq\ell}\in\{\pm1\}^\ell. 
\end{equation}

For $L\in\msl^G(M)\subseteq\msl^{G,\omega}(M_0)$, it matches a unique element in $\msl^{H'}(M')$ to be denoted by $L'$. The map $L\mapsto L'$ induces a bijection between $\msl^G(M)$ and $\msl^{H'}(M')$. For $Q\in\msf^G(M)\subseteq\msf^{G,\omega}(M_0)$, we shall always denote by $Q'$ the unique element in $\msf^{H'}(M')$ such that the restriction of $G'_{\ov{F}}\simeq G_{\ov{F}}$ induces an isomorphism $\wt{Q'}_{\ov{F}}\simeq Q_{\ov{F}}$, where $\wt{Q'}\in\msf^{G'}(\wt{M'_0})$ is associated to $Q'$ as in Section \ref{ssec:sym2-I-II}. The map $Q\mapsto Q'$ induces a bijection between $\msf^G(M)$ and $\msf^{H'}(M')$. 

\subsubsection{$M$-matching orbits}\label{ssec:M-mat_orb}

Fix a pair of matching Levi subgroups $M'\in\msl^{H'}(M'_0)$ and $M\in\msl^{G,\omega}(M_0)$ as in Section \ref{ssec:mat_Levi}. Let $X\in(\fm\cap\fs_\rs)(F)$ and $Y\in(\wt{\fm'}\cap\fs'_\rs)(F)$. Write 
\begin{equation}\label{eq:X-fact}
 X=\diag(X_1,\cdots,X_\ell)\in\bigoplus_{i=1}^\ell (\fg\fl_{n_i}\oplus\fg\fl_{n_i})(F) 
\end{equation}
and 
$$ Y=\diag(Y_1,\cdots,Y_\ell)\in\bigoplus_{i=1}^\ell \fs'_{r_i}(F). $$

\begin{defn}\label{defbyinvM}
If $X_i$ and $Y_i$ have matching orbits for each $1\leq i\leq \ell$, we say that $X$ and $Y$ have $M$-matching orbits and write $X\overset{M}{\da}Y$. For $X\in(\fm\cap\fs_\rs)(F)$, if there is an element $Y\in(\wt{\fm'}\cap\fs'_\rs)(F)$ such that $X\overset{M}{\da} Y$, we also say that $X$ comes from $(\wt{\fm'}\cap\fs'_\rs)(F)$. 
\end{defn}

\begin{remark}
The condition $X\overset{M}{\da}Y$ implies $X\da Y$, but the converse fails in general. 
\end{remark}

\begin{prop}\label{prop:equivcomefromLevi}
For $X\in(\fm\cap\fs_\rs)(F)$, it comes from $\fs'_\rs(F)$ if and only if it comes from $(\wt{\fm'}\cap\fs'_\rs)(F)$. 
\end{prop}

\begin{proof}
This is a consequence of Corollary \ref{corsym2ell}. 
\end{proof}

\subsubsection{Criterion via the Kottwitz sign}

Now suppose that $F$ is a local field of characteristic zero. Denote by $e(G')$ the Kottwitz sign of $G'$ in the sense of \cite{MR697075}. 

\begin{lem}\label{lem:kottsign}
We have $e(G')=(-1)^r$. 
\end{lem}

\begin{proof}
Let $\frac{i}{2n}=\frac{i_0}{\deg(D)}\in\BQ\cap[0,1)$ be the invariant of $\fg'(F)$, where $i_0$ and $\deg(D)$ are coprime. Since $2n=r\deg(D)$ is even, by \cite[Corollary (7)]{MR697075}, $e(G')=-1$ if and only if $i$ is odd. We have $i=\frac{2ni_0}{\deg(D)}=ri_0$. If $i$ is odd, then $r$ is odd. Conversely, if $r$ is odd, then $\deg(D)$ is even, so $i_0$ is odd, which implies that $i$ is odd. 
\end{proof}

For $X=\mat(0,A,B,0)\in\fs_\rs(F)$, we define its sign by 
$$ \eta(X):=\eta(\det(AB)). $$

\begin{prop}\label{equivdefcomefrom}
Let $X\in\fs_\rs(F)_\el$. Then the following conditions are equivalent: 
\begin{enumerate}
	\item $X$ comes from $\fs'_\rs(F)$; 

	\item $\eta(X)=(-1)^r$;  

	\item $\eta(X)=e(G')$. 
\end{enumerate}
\end{prop}

\begin{proof}
See \cite[Claim 2.7.1]{MR4226986} for (1)$\Leftrightarrow$(2). Lemma \ref{lem:kottsign} implies (2)$\Leftrightarrow$(3). 
\end{proof}

Let $M'\in\msl^{H'}(M'_0)$ and $M\in\msl^{G,\omega}(M_0)$ be a pair of matching Levi subgroups as in Section \ref{ssec:mat_Levi}. Let $X\in(\fm\cap\fs_\rs)(F)$ as in \eqref{eq:X-fact}. Denote 
\begin{equation}\label{eq:def_eta_M}
 \eta_M(X):=(\eta(X_i))_{1\leq i\leq\ell}\in\{\pm1\}^\ell. 
\end{equation}

\begin{defn}\label{defcomepotfrom}
If $\eta_M(X)=e_{M'}$, we say that $X$ comes potentially from $(\wt{\fm'}\cap\fs'_\rs)(F)$. 
\end{defn}

\begin{coro}\label{corequivdefcomefrom}
\leavevmode
\begin{enumerate}
	\item If $X$ comes from $(\wt{\fm'}\cap\fs'_\rs)(F)$, then it comes potentially from $(\wt{\fm'}\cap\fs'_\rs)(F)$. 

	\item If $X\in(\fm\cap\fs_\rs)(F)_\el$ and comes potentially from $(\wt{\fm'}\cap\fs'_\rs)(F)$, then it comes from $(\wt{\fm'}\cap\fs'_\rs)(F)$. 
\end{enumerate}
\end{coro}

\begin{proof}
It results from Corollary \ref{corsym2ell} and Proposition \ref{equivdefcomefrom}. 
\end{proof}

\subsection{Centralisers of semi-simple elements}

\subsubsection{Classification of $E\hookrightarrow\fg'(F)$}\label{ssec:classEg'}

Recall that $E=F(\alpha)$ with $\alpha^2\in F$. Set 
$$ \alpha_0:=\mat(0,\alpha^2 1_n,1_n,0)\in G(F). $$
There is an $F$-algebra embedding $i: E\hookrightarrow\fg(F)$ defined by $i(\alpha)=\alpha_0$. Denote 
$$ H_0:=G_{\alpha_0}\simeq\Res_{E/F}GL_{n,E}. $$
We have 
$$ \fh_0=\left\{\mat(A,\alpha^2 C,C,A): A,C\in\fg\fl_n\right\}. $$
Denote by $\fs_0$ the corresponding tangent space of $G/H_0$ at the neutral element. Then 
$$ \fs_0=\left\{\mat(A,-\alpha^2 C,C,-A): A,C\in\fg\fl_n\right\}. $$
Recall that $\omega_0=\mat(1_n,0,0,-1_n)\in\fs_0(F)$. Then 
$$ \omega_0^2=1_{2n} \text{ and } \fs_0=\fh_0\omega_0=\omega_0\fh_0. $$
The action $\Ad(\omega_0)$ induces the Galois involution on $H_0$ and $\fh_0$. Consider pairs of the form $(\fg',j)$ where $\fg'$ is a central simple algebra over $F$ of degree $2n$ and $j: E\hookrightarrow\fg'(F)$ is an $F$-algebra embedding. Denote 
$$ \alpha':=j(\alpha)\in G'(F). $$ 

\begin{lem}\label{lem:isopair}
Such a pair $(\fg',j)$ is always isomorphic to the pair $(\fg,i)_{\ov{F}}$ over $\ov{F}$. 
\end{lem}

\begin{proof}
In fact, since $\Prd_{\alpha'}$ is defined over $F$ and ${\alpha'}^2-\alpha^2=0$, we deduce that 
$$ \Prd_{\alpha'}(\lambda)=(\lambda^2-\alpha^2)^n=\chi_{\alpha_0}(\lambda)\in F[\lambda]. $$
But both $\alpha'$ and $\alpha_0$ are semi-simple in the classical sense, so they are conjugate over $\ov{F}$ when we fix any isomorphism $\fg'_{\ov{F}}\simeq\fg_{\ov{F}}$. 
\end{proof}

\begin{prop}\label{prop:bookinvol}
There is a natural bijection between $H^1(F, H_0/Z_G)$ and the set of $F$-isomorphism classes of pairs $(\fg',j)$. 
\end{prop}

\begin{proof}
It results from \cite[(29.13)]{MR1632779} and Lemma \ref{lem:isopair}. 
\end{proof}


By Proposition \ref{prop:bookinvol}, there is an isomorphism $\varphi: \fg'_{\ov{F}}\ra\fg_{\ov{F}}$ such that 
$$ \varphi(\alpha')=\alpha_0 $$ 
and that 
$$ \varphi\circ\sigma\circ\varphi^{-1}\circ\sigma^{-1}=\Ad(u_\sigma) $$
for all $\sigma\in\Gamma$, where $u_\sigma$ is a Galois $1$-cocycle with values in $H_0/Z_G$. 

\subsubsection{Link to the split case}

Set 
$$ \alpha_1:=\mat(\alpha1_n,,,-\alpha1_n)\in G(E). $$
Recall that $H=G_{\omega_0}=G_{\alpha_1}$ and that $\omega=\mat(0,1_n,1_n,0)\in\fs(F)$. We have 
$$ \omega^2=1_{2n} \text{ and } \fs=\fh\omega=\omega\fh. $$
The action $\Ad(\omega)$ induces an involution on $H$ and $\fh$. Let $y\in G$ be such that 
\begin{equation}\label{formula1142}
 \Ad(y)\circ\varphi(\alpha')=\alpha_1, 
\end{equation}
i.e., $\Ad(y)(\alpha_0)=\alpha_1$. Recall that $H'=G'_{\alpha'}$. The morphism $\Ad(y)\circ\varphi$ induces isomorphisms $H'_{\ov{F}}\ra H_{\ov{F}}$ and $\fs'_{\ov{F}}\ra\fs_{\ov{F}}$. For all $\sigma\in\Gamma$, we have 
$$ \Ad(\sigma(y))\circ\sigma\circ\varphi(\alpha')=\sigma(\alpha_1)=\varepsilon_\sigma\alpha_1, $$
where $\varepsilon_\sigma$ denotes the quadratic character of $\Gamma$ associated to $E/F$. Since 
$$ \sigma\circ\varphi(\alpha')=\Ad(u_\sigma^{-1})\circ\varphi\circ\sigma(\alpha')=\Ad(u_\sigma^{-1})\circ\varphi(\alpha')=\Ad(u_\sigma^{-1})\circ\Ad(y^{-1})(\alpha_1), $$
we obtain 
$$ \Ad(yu_\sigma\sigma(y)^{-1})(\alpha_1)=\varepsilon_\sigma\alpha_1. $$
For all $\sigma\in\Gamma$, define $w_\sigma:=1_{2n}$ if $\sigma\in\Gal(\ov{F}/E)$ and $w_\sigma:=\omega$ otherwise. Then $\Ad(w_\sigma)(\alpha_1)=\varepsilon_\sigma\alpha_1$. We deduce that 
\begin{equation}\label{formula1143}
 yu_\sigma\sigma(y)^{-1}w_\sigma\in H. 
\end{equation}
The elements $y$ verifying \eqref{formula1142} form an $H$-torsor of the form $Hv$ where 
\begin{equation}\label{formula1144}
v:=\mat(1_n,\alpha1_n,1_n,-\alpha1_n)\in G(E). 
\end{equation}
We easily check that
$\sigma(v)=w_\sigma v$ for all $\sigma\in\Gamma$ and that $\omega v=v\omega_0$. 

\subsubsection{Isomorphisms over $\ov{F}$ of centralisers}

Let $Y\in\fs'(F)$ be a semi-simple element. There exists $y\in Hv$ and a semi-simple element $X\in\fs(F)$ such that 
$$ \Ad(y)\circ\varphi(Y)=X. $$
Then $\Ad(y)\circ\varphi$ induces isomorphisms $H'_Y\ra H_X$ and $\fs'_{Y}\ra \fs_X$ over $\ov{F}$. For all $\sigma\in\Gamma$, we see that 
\begin{equation}\label{eq:adj_cocycle}
 (\Ad(y)\circ\varphi)\circ\sigma\circ(\Ad(y)\circ\varphi)^{-1}\circ\sigma^{-1}=\Ad(y)\circ(\varphi\circ\sigma\circ\varphi^{-1}\circ\sigma^{-1})\circ(\sigma\circ\Ad(y^{-1})\circ\sigma^{-1})=\Ad(yu_\sigma\sigma(y)^{-1}). 
\end{equation}
Since 
\[\begin{split}
 &(\Ad(y)\circ\varphi)\circ\sigma\circ(\Ad(y)\circ\varphi)^{-1}\circ\sigma^{-1}(X)=(\Ad(y)\circ\varphi)\circ\sigma\circ(\Ad(y)\circ\varphi)^{-1}(X) \\ 
 =&(\Ad(y)\circ\varphi)\circ\sigma(Y)=(\Ad(y)\circ\varphi)(Y)=X, 
\end{split}\]
we obtain 
\begin{equation}\label{eq:G_X}
 yu_\sigma\sigma(y)^{-1}\in G_X. 
\end{equation}
Combining this with \eqref{formula1143}, we have 
\begin{equation}\label{eq:G_XcapHw}
 yu_\sigma\sigma(y)^{-1}\in Hw_\sigma\cap G_X. 
\end{equation}
Since $w_\sigma$ normalises $H$ (resp. $\fs$), we see that $yu_\sigma\sigma(y)^{-1}$ normalises $H_X$ (resp. $\fs_X$). 

\subsubsection{Case of invertible elements: inner twists}

We first consider the particular case where $Y\in{\fs'}^\times(F)$ is a semi-simple element. In this case, by \cite[Proposition 2.1]{MR1394521}, up to conjugation by $H(F)$, we may suppose that 
$$ X=\mat(0,1_n,A,0) $$
where $A\in GL_n(F)$ is semi-simple in the classical sense. By \cite[Proposition 2.2]{MR1394521} or direct calculation, for such $X$, there is an isomorphism of representations 
$$ (H_X, \fs_X)\simeq(GL_{m,A}, \fg\fl_{m,A}). $$

\begin{lem}\label{lem:cent-intwist}
For such $Y$ and $X$, the isomorphism $\Ad(y)\circ\varphi$ over $\ov{F}$ from $H'_Y$ to $H_X$ is an inner twist. 
\end{lem}

\begin{proof}
We have \eqref{eq:adj_cocycle} for all $\sigma\in\Gamma$. Thus it suffices to show that for any $\sigma\in\Gamma$, there exists $h_\sigma\in H_X$ such that the automorphisms $\Ad(yu_\sigma\sigma(y)^{-1})$ and $\Ad(h_\sigma)$ on $H_X$ coincide. By \eqref{eq:G_XcapHw}, if $\sigma\in\Gal(\ov{F}/E)$, we can take $h_\sigma:=yu_\sigma\sigma(y)^{-1}$; otherwise, we can take $h_\sigma:=yu_\sigma\sigma(y)^{-1}X$. 
\end{proof}

\begin{lem}
\leavevmode
\begin{enumerate}		
	\item For any $X_0\in\fs^\times(F)$, the map $U\mapsto UX_0$ defines an $H_{X_0}$-equivariant isomorphism $\fh_{X_0}\ra\fs_{X_0}$ over $F$; 
	
	\item For any $Y_0\in{\fs'}^\times(F)$, the map $V\mapsto VY_0$ defines an $H'_{Y_0}$-equivariant isomorphism $\fh'_{Y_0}\ra\fs'_{Y_0}$ over $F$. 
\end{enumerate}
\end{lem}

\begin{proof}
Since $X_0$ is invertible, the map $U\mapsto UX_0^{-1}$ from $\fs_{X_0}$ to $\fh_{X_0}$ is the inverse for the map in (1). The proof of (2) is similar. 
\end{proof}

\subsubsection{Case of regular elements: isomorphisms over $F$}\label{ssec:centrss}

We now consider an even more special case. Let $X\in\fs_\rs(F)$ and $Y\in\fs'_\rs(F)$ be such that $X\da Y$. Then there exists $y\in Hv$ such that $\Ad(y)\circ\varphi(Y)=X$. 

\begin{lem}\label{lem74}
For such $Y$ and $X$, $\Ad(y)\circ\varphi$ induces isomorphisms $H'_{Y}\ra H_X$ and $\fs'_Y\ra \fs_X$ over $F$. 
\end{lem}

\begin{proof}
This is a generalisation of \cite[Lemma 7.4]{MR3414387}. Again we have \eqref{eq:adj_cocycle} for all $\sigma\in\Gamma$. For $X\in\fs_\rs$, we can deduce from \cite[Proposition 4.2]{MR3414387} that $X$ is also regular semi-simple in $G$ in the classical sense. By \eqref{eq:G_X}, the action $\Ad(yu_\sigma\sigma(y)^{-1})$ on $H_X$ (resp. $\fs_X$) is trivial, which implies that $\Ad(y)\circ\varphi$ induces an isomorphism $H'_{Y}\ra H_X$ (resp. $\fs'_Y\ra \fs_X$) over $F$. 
\end{proof}

Suppose that $F$ is a local field of characteristic zero. With regard to Lemma \ref{lem74}, we shall choose compatible Haar measures on $H'_Y(F)\ra H_X(F)$ and $\fs'_Y(F)\simeq\fs_X(F)$. 

\subsubsection{General case}

We return to the general case of a semi-simple element $Y\in\fs'(F)$. By \cite[Proposition 3.25]{MR4681295}, up to $H'(F)$-conjugation, we may suppose that 
$$ Y=\mat(B,,,0_{r-s}) $$
where $1\leq s\leq r$ is an integer in \textbf{Case I} (resp. an even integer in \textbf{Case II}) and $B\in{\fs'_s}^\times(F)$ is a semi-simple element with respect to the $H'_s$-action. For such Y, there is an isomorphism of representations 
$$(H'_Y, \fs'_Y)\simeq(H'_{s,B}, \fs'_{s,B})\times(H'_{r-s}, \fs'_{r-s}). $$
By \cite[Proposition 2.1]{MR1394521}, up to conjugation by $H(F)$, we may suppose that 
$$ X=
\left( \begin{array}{cccc}
0 & 0 & 1_{m} & 0 \\
0 & 0 & 0 & 0 \\
A & 0 & 0 & 0 \\
0 & 0 & 0 & 0 \\
\end{array} \right) $$
where $1\leq m\leq n$ is an integer and $A\in GL_m(F)$ is semi-simple in the classical sense. By \cite[Proposition 2.2]{MR1394521}, for such $X$, there is an isomorphism of representations 
$$ (H_X, \fs_X)\simeq(GL_{m,A}, \fg\fl_{m,A})\times(H_{n-m}, \fs_{n-m}). $$
Since $\Ad(y)\circ\varphi(Y)=X$, we have $2m=s\deg(D)$. By Lemma \ref{lem:cent-intwist}, $H'_{s,B}$ is an inner form of $GL_{m,A}$. This generalises \cite[Remark 5.5]{MR3414387} and is useful for the smooth transfer. 

\subsection{Cohomological criterion}\label{seccohcri}

\subsubsection{The Galois $1$-cocycle associated to $X$}\label{ssec:cocycletoX}

Let 
$$ X=\mat(0,1_n,A,0)\in\fs_\rs(F) $$
where $A\in GL_n(F)$ is regular semi-simple in the classical sense (see \cite[Proposition 4.2]{MR3414387}). Then $GL_{n,A}$ is a maximal $F$-torus in $GL_n$. Denote 
$$ X_0:=\Ad(v^{-1})(X)\in\fs_0(E), $$
where $v$ is defined by \eqref{formula1144}. Since $H_X=\{(x,x)\in H: x\in GL_{n,A}\}$, we see that $\Ad(v^{-1})$ acts trivially on $H_X$ and thus 
$$ H_{0,X_0}=\Ad(v^{-1})(H_X)=H_X\simeq GL_{n,A}. $$
Denote 
$$ H_A:=\Cent_H(\diag(A,A))=GL_{n,A}\times GL_{n,A}, $$
which is a maximal $F$-torus in $H$, and 
$$ T_X:=\Cent_{H_0}(\diag(A,A))=\Ad(v^{-1})(H_A). $$

\begin{lem}\label{lem:borel}
If $V$ is a closed subvariety of $H$ defined over $F$ and stabilised by $\Ad(\omega)$, then $\Ad(v^{-1})(V)$ is a closed subvariety of $H_0$ defined over $F$ and stabilised by $\Ad(\omega_0)$. 
\end{lem}

\begin{proof}
For all $\sigma\in\Gamma$, we have 
$$ \sigma(\Ad(v^{-1})(V))=\Ad(\sigma(v)^{-1})(V)=\Ad(v^{-1} w_\sigma^{-1})(V)=\Ad(v^{-1})(V). $$
By \cite[\S AG, Theorem 14.4]{MR1102012}, $\Ad(v^{-1})(V)$ is defined over $F$. The other assertions are obvious. 
\end{proof}

By Lemma \ref{lem:borel}, $T_X$ is a maximal $F$-torus in $H_0$. Notice that $H_X=H_A^{\omega}$ and that $H_{0,X_0}=T_X^{\omega_0}$. We also observe that 
$$ T_X\simeq\Res_{E/F} (GL_{n,A})_E $$ 
and that the inclusion $H_{0,X_0}\subseteq T_X$ is isomorphic to the inclusion $GL_{n,A}\subseteq\Res_{E/F} (GL_{n,A})_E$. For simplicity, we also write 
\begin{equation}\label{eq:defTandR}
 T:=T_X \text{ and } R:=H_{0,X_0}, 
\end{equation}
which are maximal $F$-tori of $H_0$ and $H_0^{\omega_0}$ respectively.  

\begin{lem}\label{lem1231}
There exists a unique Galois $1$-cocycle $t_\sigma$ with values in $T/R$ such that for all $\sigma\in\Gamma$, 
$$ \Ad(t_\sigma)\circ\sigma(X_0)=X_0. $$
\end{lem}

\begin{proof}
The uniqueness is obvious from the definition of $R$. It suffices to consider the existence. The cocycle condition is also automatic. It is thus enough to check the equality. If $\sigma\in\Gal(\ov{F}/E)$, then $\sigma(X_0)=X_0$, so it suffices to take $t_\sigma:=1$. Now suppose that $\sigma\notin\Gal(\ov{F}/E)$. Since 
$$ \sigma(X_0)=\Ad(\sigma(v)^{-1})\circ\sigma(X)=\Ad(v^{-1}\omega^{-1})(X), $$
it suffices to find an element $t_\sigma\in T$ such that $vt_\sigma v^{-1}\omega^{-1}\in G_X$. In fact, we can take $t_\sigma:=\Ad(v^{-1})\mat(1,,,A)$. 
\end{proof}

Since $A\in GL_n(F)$ is regular semi-simple, its characteristic polynomial $\chi_A$ is separable and factorised into a product $\chi_A=\prod_{i\in I}\chi_i$ of pairwise distinct monic irreducible polynomials over $F$. Then 
$$ GL_{n,A}\simeq \prod_{i\in I} \Res_{F_i/F} \BG_{m,F_i}, $$
where $F_i=F[\lambda]/(\chi_i(\lambda))$. Using the inflation-restriction exact sequence, Shapiro's lemma and Hilbert's Theorem 90 (see \cite[\S 29.A]{MR1632779} for example), we can identify 
$$ H^1(F,T/R)\simeq H^1(\Gal(E/F), T(E)/R(E)). $$

For all $\sigma\in\Gamma$ and $h':=\Ad(v^{-1})(h)\in T$ where $h\in H_A$, we see that 
$$ \sigma(h')=\Ad(\sigma(v)^{-1})\circ\sigma(h)=\Ad(v^{-1})\circ\Ad(w_\sigma)\circ\sigma(h). $$
Therefore, we can regard $T$ as the subgroup $H_A$ of $G$ equipped with the Galois action $\Ad(w_\sigma)\circ\sigma$. Then the inclusion $R(E)\subseteq T(E)$ is isomorphic to the diagonal inclusion $H_X(E)\subseteq H_A(E)$. Let $\sigma$ be the nontrivial element in $\Gal(E/F)$ and $s_\sigma:=(B,C)\in H_A(E)/H_X(E)$. By a straightforward calculation, we observe that 
\begin{enumerate}
	\item $s_\sigma$ defines a $1$-cocycle if and only if $s_\sigma\Ad(w_\sigma)\circ\sigma(s_\sigma)\in H_X(E)$, i.e., $BC^\sigma=CB^\sigma$; 
	
	\item $s_\sigma$ defines a $1$-coboundary if and only if $s_\sigma=(B,C)^{-1}\Ad(w_\sigma)\circ\sigma(B,C)$ for some $(B,C)\in H_A(E)/H_X(E)$, i.e., $s_\sigma=(C^\sigma B^{-1},B^\sigma C^{-1})\in H_A(E)/H_X(E)$. 
\end{enumerate}
Then $H^1(\Gal(E/F),T(E)/R(E))$ is the quotient of 
\[\begin{split}
 &\{(B,C)\in H_A(E)/H_X(E): BC^\sigma=CB^\sigma\}=\{(1,CB^{-1})\in H_A(E)/H_X(E): (CB^{-1})^\sigma=CB^{-1}\} \\ 
=&\{(1,B)\in H_A(E)/H_X(E): B^\sigma=B\}=\{(1,B)\in H_A(E)/H_X(E): B\in GL_{n,A}(F)\} \\
\end{split}\]
by 
\[\begin{split}
 &\{(C^\sigma B^{-1},B^\sigma C^{-1})\in H_A(E)/H_X(E): B,C\in GL_{n,A}(E)\} \\
=&\{(B^{-\sigma},B)\in H_A(E)/H_X(E): B\in GL_{n,A}(E)\}=\{(1,BB^\sigma)\in H_A(E)/H_X(E): B\in GL_{n,A}(E)\}. \\
\end{split}\]
That is to say, it is the quotient of $GL_{n,A}(F)$ by the group of norms of $GL_{n,A}(E)$. Under the identification of $T$ and $H_A$ with the twisted Galois action, the Galois $1$-cocycle $t_\sigma$ in Lemma \ref{lem1231} (see the proof) corresponds exactly to the class of $A$. 

Denote $E_i:=F_i\otimes_F E$ for $i\in I$. Let $I_0$ be the subset of $I$ consisting of $i$ such that $E_i$ is a field. Then 
\begin{equation}\label{eq:coh_opp_tor}
 H^1(F, T/R)\simeq\prod_{i\in I} F_i^\times/N_{E_i/F_i}(E_i^\times)=\prod_{i\in I_0} F_i^\times/N_{E_i/F_i}(E_i^\times). 
\end{equation}
If $F$ is a local field of characteristic zero, then $H^1(F,T/R)=(\BZ/2\BZ)^{I_0}$. 

\subsubsection{Criterion via Galois $1$-cocycles}

Note that $Z_G\subseteq R$. 

\begin{lem}\label{lem1251}
Let $t\in H^1(F,T/R)$ be the class of the Galois $1$-cocycle $t_\sigma$ in Lemma \ref{lem1231} associated to $X$. Let $u\in H^1(F,H_0/Z_G)$ be the class of the Galois $1$-cocycle $u_\sigma$ in Section \ref{ssec:classEg'} associated to $\fg'$. Then the following conditions are equivalent: 

\begin{enumerate}
	\item there exists $Y\in\fs'(F)$ and $h\in H_0$ such that $\Ad(h)\circ\varphi(Y)=X_0$; 
	
	\item there exists an element of $H^1(F,T/Z_G)$ which has images $t\in H^1(F,T/R)$ and $u\in H^1(F,H_0/Z_G)$ under the natural maps: 
$$   \xymatrix{ H^1(F,T/Z_G) \ar[r]  \ar[d] & H^1(F,T/R)  \\
 H^1(F,H_0/Z_G) & } $$
\end{enumerate}
\end{lem}

\begin{remark}
The condition (1) in the above lemma says exactly that $X$ comes from $\fs'_\rs(F)$. 
\end{remark}

\begin{proof}[Proof of Lemma \ref{lem1251}]
Given any $h\in H_0$, let $Y\in\fs'$ be the unique element such that $\Ad(h)\circ\varphi(Y)=X_0$. For all $\sigma\in\Gamma$, we have 
$$ \Ad(\sigma(h))\circ\sigma\circ\varphi(Y)=\sigma(X_0)=\Ad(t_\sigma^{-1})(X_0) $$ 
and 
$$ \sigma\circ\varphi(Y)=\Ad(u_\sigma^{-1})\circ\varphi\circ\sigma(Y). $$ 
But $Y\in\fs'(F)$ if and only if 
$$ \varphi\circ\sigma(Y)=\varphi(Y)=\Ad(h^{-1})(X_0) $$
for all $\sigma\in\Gamma$. We have shown that the condition (1) is equivalent to the following one: there exists $h\in H_0$ such that 
\begin{equation}\label{eq:1252}
 \Ad(\sigma(h)u_\sigma^{-1}h^{-1})(X_0)=\Ad(t_\sigma^{-1})(X_0) 
\end{equation}
for all $\sigma\in\Gamma$. But \eqref{eq:1252} means exactly 
$$ hu_\sigma\sigma(h)^{-1}\in Rt_\sigma. $$
If such an $h$ exists, then $hu_\sigma\sigma(h)^{-1}\in Rt_\sigma\subseteq T$ defines a Galois $1$-cocycle with values in $T/Z_G$ which satisfies the condition (2). Conversely, any Galois $1$-cocycle with values in $T/Z_G$ having image $u$ must be of the form $hu_\sigma\sigma(h)^{-1}$ where $h\in H_0$. If it also has image $t$, then by replacing $h$ with $t'h$ where $t'\in T$, we may suppose that $hu_\sigma\sigma(h)^{-1}\in Rt_\sigma$. 
\end{proof}

\subsubsection{Criterion via the Kottwitz sign revisited}\label{ssec:kotrev}

\paragraph{}

Note that $H_0^{\omega_0}=H^\omega\simeq GL_n$. We shall abuse notation and denote by $R_\der$ (resp. $T_\der$) the preimage of $R$ (resp. $T$) in $H_{0,\der}^{\omega_0}=H_{0,\scn}^{\omega_0}$ (resp. $H_{0,\der}=H_{0,\scn}$). We shall use the index ``$\ab$'' to denote the abelianised cohomology defined in \cite[Definition 2.2]{MR1401491} and \cite[\S1.6 and 1.8]{MR1695940}. By \cite[Proposition 1.6.7]{MR1695940}, we have abelianisation maps $H^1(F,H_0^{\omega_0}/Z_G)\ra H^1_\ab(F,H_0^{\omega_0}/Z_G)$ and $H^1(F,H_0/Z_G)\ra H^1_\ab(F,H_0/Z_G)$ which are surjective. If $F$ is a non-archimedean local field of characteristic zero, they are also injective by Kneser's theorem (see {\it{loc. cit.}}). By the short exact sequence of complexes 
$$ 1\ra[1\ra Z_G]\ra[R_\der\ra R]\ra[R_\der\ra R/Z_G]\ra 1, $$
we have 
$$ H^i_\ab(F,H_0^{\omega_0}/Z_G)=H^i_\ab(F, Z_G\ra H_0^{\omega_0})=H^i_\ab(F,R_\der\ra R/Z_G) $$
where $i=1, 2$. Similarly, we see that 
$$ H^1_\ab(F,H_0/Z_G)=H^1_\ab(F,T_\der\ra T/Z_G). $$
By the short exact sequence of complexes 
$$ 1\ra[R_\der\ra R/Z_G]\ra[T_\der\ra T/Z_G]\ra [T_\der/R_\der\ra T/R]\ra 1, $$
we have 
\begin{equation}\label{eq:H_ab-ra}
 H^1_\ab(F, H_0^{\omega_0}/Z_G\ra H_0/Z_G)=H^1_\ab(F,T_\der/R_\der\ra T/R). 
\end{equation}
From the theory of Galois hypercohomology (see \cite[Appendix A.1]{MR1687096}), we obtain the following commutative diagram with exact columns and rows. 

\begin{equation}\label{eq:commdiag}
   \xymatrix{ H^1(F, R/Z_G) \ar[r]  \ar[d] & H^1(F, T/Z_G) \ar[r]  \ar[d] & H^1(F, T/R) \ar[r]  \ar[d] & H^2(F,R/Z_G) \ar[d] \\
H^1_\ab(F,H_0^{\omega_0}/Z_G) \ar[r]  \ar[d] & H^1_\ab(F,H_0/Z_G) \ar[r] \ar[d] & H^1_\ab(F, H_0^{\omega_0}/Z_G\ra H_0/Z_G) \ar[r] \ar[d] & H^2_\ab(F,H_0^{\omega_0}/Z_G)	\\ 
H^2(F,R_\der) \ar[r]  & H^2(F,T_\der) \ar[r] & H^2(F,T_\der/R_\der) 	& } 
\end{equation}

Let $t\in H^1(F,T/R)$ and $u\in H^1(F,H_0/Z_G)$ be as in Lemma \ref{lem1251}. Denote by $\varepsilon_X$ the image of $t$ under the homomorphism 
$$ H^1(F,T/R)\ra H^1_\ab(F, H_0^{\omega_0}/Z_G\ra H_0/Z_G). $$
Denote by $\varepsilon'$ the image of $u$ under the morphism 
$$ H^1(F,H_0/Z_G)\ra H^1_\ab(F,H_0/Z_G)\ra H^1_\ab(F, H_0^{\omega_0}/Z_G\ra H_0/Z_G). $$

\begin{lem}\label{lem2131}
If $X$ comes from $\fs'_\rs(F)$, then $\varepsilon_X=\varepsilon'$. 
\end{lem}

\begin{proof}
It results from Lemma \ref{lem1251} and the above commutative diagram. 
\end{proof}

\begin{prop}\label{prop2231weak}
There is a natural identification 
$$ H^1_\ab(F, H_0^{\omega_0}/Z_G\ra H_0/Z_G)\simeq H^1(F,C_{H_0}/C_{H_0^{\omega_0}}). $$
\end{prop}

\begin{remark}
We have $C_{H_0}=\Res_{E/F} (C_{H_0^\omega})_E$ where $C_{H_0^\omega}\simeq C_{GL_n}\simeq \BG_m$. 
\end{remark}

\begin{proof}[Proof of Proposition \ref{prop2231weak}]
By \eqref{eq:H_ab-ra} and the short exact sequence 
$$ 1\ra[R_\der\ra T_\der] \ra[R\ra T] \ra[R/R_\der\ra T/T_\der] \ra1, $$
we have 
$$ H^1_\ab(F, H_0^{\omega_0}/Z_G\ra H_0/Z_G)=H^1(F, (T/T_\der)/(R/R_\der)). $$
But $T/T_\der=C_{H_0}$ and $R/R_\der=C_{H_0^{\omega_0}}$. 
\end{proof}

\begin{lem}\label{lem:cohopptorus}
Let $\BG_m^\ast:=\Res_{E/F} \BG_{m,E}/\BG_m$. Then 
$$ H^1(F,\BG_m^\ast)\simeq H^1(\Gal(E/F),\BG_m^\ast(E))\simeq F^\times/N_{E/F}(E^\times). $$
\end{lem}

\begin{proof}
This should be known and also included in Section \ref{ssec:cocycletoX}. In fact, the first equality is a consequence of the inflation-restriction exact sequence, Shapiro's lemma and Hilbert's Theorem 90. To obtain the second one, we can regard $\Res_{E/F} \BG_{m,E}$ as $\BG_m\times \BG_m$ equipped with the $\Gamma$-action $\Ad(w_\sigma)\circ\sigma$, where $w_\sigma:=1$ if $\sigma\in\Gal(\ov{F}/E)$ and $w_\sigma(x,y):=(y,x)$ otherwise. Then $\BG_m^\ast$ can be viewed as $\BG_m$ equipped with the $\Gamma$-action given by $\sigma\cdot x:=\sigma(x)$ if $\sigma\in\Gal(\ov{F}/E)$ and $\sigma\cdot x:=\sigma(x)^{-1}$ otherwise. We can conclude by a straightforward calculation. 
\end{proof}

\paragraph{}

In this paragraph, we assume that $F$ is a local field of characteristic zero. 

\begin{coro}\label{cor2231weak}
\leavevmode
\begin{enumerate}
	\item We have 
$$ H^1_\ab(F, H_0^{\omega_0}/Z_G\ra H_0/Z_G)=\BZ/2\BZ. $$

	\item The natural homomorphism  
$$ H^1(F,T/R)=(\BZ/2\BZ)^{I_0}\ra H^1_\ab(F, H_0^{\omega_0}/Z_G\ra H_0/Z_G)=\BZ/2\BZ $$
is the sum of components. 
\end{enumerate}
\end{coro}

\begin{proof}
The first assertion follows from Proposition \ref{prop2231weak} and Lemma \ref{lem:cohopptorus}. The second one is clear by the calculation in Section \ref{ssec:cocycletoX}. 
\end{proof}

\begin{coro}\label{idcar1}
The map 
$$ \fs_\rs(F)\ra H^1_\ab(F, H_0^{\omega_0}/Z_G\ra H_0/Z_G)=\BZ/2\BZ: X=\mat(0,A,B,0)\mapsto \varepsilon_{X'} $$
where $X':=\mat(0,1_n,AB,0)\in\fs_\rs(F)$ is understood as the map 
$$ \fs_\rs(F)\ra\{\pm1\}: X\mapsto\eta(X). $$ 
\end{coro}

\begin{proof}
Because of Corollary \ref{cor2231weak}, we may reduce ourselves to the case where $X\in\fs_\rs(F)_\el$. In the elliptic case, the assertion results from \cite[Lemmas 1.3 and 1.4]{MR1007299}. 
\end{proof}

\paragraph{}

Now assume additionally that $F$ is non-archimedean. 

\begin{prop}\label{prop:pdweak}
The group $H^1_\ab(F, H_0^{\omega_0}/Z_G\ra H_0/Z_G)$ is naturally isomorphic to the Pontryagin dual of the finite group $Z_{\wh{G}}[2]$ of elements $z\in Z_{\wh{G}}$ such that $z^2=1$. 
\end{prop}

\begin{proof}
By Proposition \ref{prop2231weak} and \cite[(6.3.1)]{MR757954}, the group $H^1_\ab(F, H_0^{\omega_0}/Z_G\ra H_0/Z_G)$ is naturally isomorphic to the Pontryagin dual of $\pi_0(\wh{C_{H_0}/C_{H_0^{\omega_0}}}^\Gamma)$. Using \cite[(1.8.1) and (1.8.3)]{MR757954}, we can identify $\wh{C_{H_0}/C_{H_0^{\omega_0}}}$ with the kernel of $\wh{C_{H_0}}=Z_{\wh{H_0}}\ra\wh{C_{H_0^{\omega_0}}}=Z_{\wh{H_0^{\omega_0}}}$. Notice that $Z_{\wh{H_0}}=Z_{\wh{H_0^{\omega_0}}}\times Z_{\wh{H_0^{\omega_0}}}$ on which $\Gamma$ acts via its quotient $\Gal(E/F)$ and the nontrivial element $\sigma\in\Gal(E/F)$ acts by exchanging two components. Since the homomorphism $Z_{\wh{H_0}}\ra Z_{\wh{H_0^{\omega_0}}}$ is given by the product of components, the kernel can be identified with $Z_{\wh{H_0^{\omega_0}}}\simeq Z_{\wh{G}}$, on which $\Gamma$ acts via its quotient $\Gal(E/F)$ and $\sigma(z)=z^{-1}$. Then the assertion follows. 
\end{proof}

\begin{prop}\label{prop2141}
Let $X=\mat(0,1_n,A,0)\in\fs_\rs(F)_\el$, i.e., $A$ is a regular semi-simple and elliptic element of $GL_n(F)$. Then $X$ comes from $\fs'_\rs(F)$ if and only if $\varepsilon_X=\varepsilon'$. 
\end{prop}

\begin{proof}
By Lemma \ref{lem2131}, it suffices to prove the reverse direction. Since $F$ is local and the tori $R/Z_G, R_\der$ and $T_\der$ are $F$-anisotropic, the groups $H^2(F,R/Z_G), H^2(F,R_\der), H^2(F,T_\der)$ and $H^2(F,T_\der/R_\der)$ vanish by \cite[Lemme 1.5.1]{MR1695940}. Moreover, since $F$ is non-archimedean, we also have $H^2_\ab(F,H_0^{\omega_0}/Z_G)=0$ (see {\it{loc. cit.}}). The commutative diagram \eqref{eq:commdiag} is thus simplified as follows. 
$$   \xymatrix{ H^1(F, R/Z_G) \ar[r]^f  \ar[d] & H^1(F, T/Z_G) \ar[r]  \ar[d] & H^1(F, T/R) \ar[r]  \ar[d] & 0 \\
H^1_\ab(F,H_0^{\omega_0}/Z_G) \ar[r]^g  \ar[d] & H^1_\ab(F,H_0/Z_G) \ar[r] \ar[d] & H^1_\ab(F, H_0^{\omega_0}/Z_G\ra H_0/Z_G) \ar[r] \ar[d] & 0	\\ 
0  & 0 & 0 	& } $$
Choose an arbitrary preimage $t'\in H^1(F, T/Z_G)$ of $t$. The image of $t'$ in $H^1_\ab(F,H_0/Z_G)$ is of the form $ug(u_1)$ where $u_1\in H^1_\ab(F,H_0^{\omega_0}/Z_G)$ because both of $u$ and this image are mapped to $\varepsilon'=\varepsilon_X\in H^1_\ab(F, H_0^{\omega_0}/Z_G\ra H_0/Z_G)$ by our assumption. Let $t_1\in H^1(F, R/Z_G)$ be a preimage of $u_1$. Then $t'f(t_1)^{-1}\in H^1(F, T/Z_G)$ has images $t$ and $u$. We may conclude by the bijectivity of the abelianisation map $H^1(F,H_0/Z_G)\ra H^1_\ab(F,H_0/Z_G)$ and Lemma \ref{lem1251}. 
\end{proof}

\begin{coro}\label{idcar2}
The morphism  
$$ H^1(F,H_0/Z_G)\ra H^1_\ab(F,H_0/Z_G)\ra H^1_\ab(F, H_0^{\omega_0}/Z_G\ra H_0/Z_G)=\BZ/2\BZ: u\mapsto\varepsilon' $$
is understood as the Kottwitz sign 
$$ H^1(F,H_0/Z_G)\ra H^1(F, G_\ad) \ra\{\pm1\}. $$
\end{coro}

\begin{proof}
It is a consequence of Proposition \ref{equivdefcomefrom}, Corollary \ref{idcar1} and Proposition \ref{prop2141}. 
\end{proof}

\subsection{Levi subgroups}\label{secmatllevi}

\subsubsection{}

Let $M\in\msl^{G,\omega}(M_0)$ be such that 
\begin{equation}\label{eq:M-fact-bis}
  (M,M_H, \fm\cap\fs)\simeq \prod_{i=1}^\ell (G_{n_i},H_{n_i}, \fs_{n_i}). 
\end{equation}
Denote $M_{H_0}:=M\cap H_0$. Because $v$ defined by \eqref{formula1144} belongs to $M(E)$, we have $M_{H_0}=\Ad(v^{-1})(M_H)$. Then $M_{H_0}$ is a Levi subgroup of $H_0$ defined over $F$. Consider $X=\mat(0,1_n,A,0)\in(\fm\cap\fs_\rs)(F)$. Since $H_A\subseteq M_H$, we have $T\subseteq M_{H_0}$ and $R\subseteq M_{H_0}^{\omega_0}$, where we use the notations \eqref{eq:defTandR}. Denote by $\varepsilon_X^M$ the image of $t\in H^1(F,T/R)$ (as in Lemma \ref{lem1251}) under the homomorphism 
$$ H^1(F,T/R)\ra H^1_\ab(F, M_{H_0}^{\omega_0}/Z_M\ra M_{H_0}/Z_M). $$
We have an obvious generalisation to the product form of several results in Section \ref{ssec:kotrev} whose proofs will be omitted. 

\begin{prop}[cf. Proposition \ref{prop2231weak}]\label{prop2231}
There is a natural identification 
$$ H^1_\ab(F, M_{H_0}^{\omega_0}/Z_M\ra M_{H_0}/Z_M)\simeq H^1(F,C_{M_{H_0}}/C_{M_{H_0}^{\omega_0}}). $$
\end{prop}

For $L\in\msl^G(M)$, the natural morphism 
$$ [C_{M_{H_0}^{\omega_0}}\ra C_{M_{H_0}}]\ra[C_{L_{H_0}^{\omega_0}}\ra C_{L_{H_0}}] $$
induces a homomorphism 
$$ H^1_\ab(F, M_{H_0}^{\omega_0}/Z_M\ra M_{H_0}/Z_M)\ra H^1_\ab(F, L_{H_0}^{\omega_0}/Z_L\ra L_{H_0}/Z_L) $$
by Proposition \ref{prop2231}. 

\subsubsection{}

In this section, we assume that $F$ is a local field of characteristic zero. For $X\in(\fm\cap\fs_\rs)(F)$, define $\eta_M(X)\in\{\pm1\}^\ell$ by \eqref{eq:def_eta_M}. 

\begin{coro}[cf. Corollary \ref{cor2231weak}]
\leavevmode
\begin{enumerate}
	\item We have  
$$ H^1_\ab(F, M_{H_0}^{\omega_0}/Z_M\ra M_{H_0}/Z_M)=(\BZ/2\BZ)^\ell. $$

	\item The natural homomorphism  
$$ H^1(F,T/R)=(\BZ/2\BZ)^{I_0}\ra H^1_\ab(F, M_{H_0}^{\omega_0}/Z_M\ra M_{H_0}/Z_M)=(\BZ/2\BZ)^\ell $$
is the sum of components within each factor $G_{n_i}$ of $M$. 

	\item Let $L\in\msl^G(M)$ with an isomorphism 
$$  (L,L_H, \fl\cap\fs)\simeq \prod_{j=1}^{\ell'} (G_{n'_j},H_{n'_j}, \fs_{n'_j}) $$
compatible with \eqref{eq:M-fact-bis} and the inclusion $M\subseteq L$. Then the natural homomorphism 
$$ H^1_\ab(F, M_{H_0}^{\omega_0}/Z_M\ra M_{H_0}/Z_M)=(\BZ/2\BZ)^\ell\ra H^1_\ab(F, L_{H_0}^{\omega_0}/Z_L\ra L_{H_0}/Z_L)=(\BZ/2\BZ)^{\ell'} $$
is the sum of components within each factor $G_{n'_j}$ of $L$. 
\end{enumerate}
\end{coro}

\begin{coro}[cf. Corollary \ref{idcar1}]\label{idcar3}
The map 
$$ (\fm\cap\fs_\rs)(F)\ra H^1_\ab(F, M_{H_0}^{\omega_0}/Z_M\ra M_{H_0}/Z_M)=(\BZ/2\BZ)^\ell: X=\mat(0,A,B,0)\mapsto \varepsilon_{X'}^M $$
where $X':=\mat(0,1_n,AB,0)\in(\fm\cap\fs_\rs)(F)$ is understood as the map 
$$ (\fm\cap\fs_\rs)(F)\ra\{\pm1\}^\ell: X\mapsto\eta_M(X). $$ 
\end{coro}

\subsubsection{}

Now assume additionally that $F$ is non-archimedean. 

\begin{prop}[cf. Proposition \ref{prop:pdweak}]\label{proppontdual}
\leavevmode
\begin{enumerate}
	\item The group $H^1_\ab(F, M_{H_0}^{\omega_0}/Z_M\ra M_{H_0}/Z_M)$ is naturally isomorphic to the Pontryagin dual of the finite group $Z_{\wh{M}}[2]$ of elements $z\in Z_{\wh{M}}$ such that $z^2=1$. 
	
	\item If $L\in\msl^G(M)$, then the natural homomorphism 
$$ H^1_\ab(F, M_{H_0}^{\omega_0}/Z_M\ra M_{H_0}/Z_M)\ra H^1_\ab(F, L_{H_0}^{\omega_0}/Z_L\ra L_{H_0}/Z_L) $$
is the dual of the embedding $Z_{\wh{L}}[2]\hookrightarrow Z_{\wh{M}}[2]$. 
\end{enumerate}
\end{prop}

\begin{coro}[cf. Corollary \ref{idcar2}]\label{idcar4}
The morphism 
$$ H^1(F,M_{H_0}/Z_M)\ra H^1_\ab(F,M_{H_0}/Z_M)\ra H^1_\ab(F, M_{H_0}^{\omega_0}/Z_M\ra M_{H_0}/Z_M)=(\BZ/2\BZ)^\ell $$
is understood as the Kottwitz sign within each factor $G_{n_i}$ of $M$ 
$$ H^1(F,M_{H_0}/Z_M)\ra H^1(F, M_\ad) \ra\{\pm1\}^\ell. $$
\end{coro}

\subsubsection{}

Return to the general setting where $F$ is a local field of characteristic zero or a number field. 

\begin{lem}\label{lem2242}
Let $A$ be a reductive group over $F$ and let $L$ be a Levi subgroup of $A$. The natural map 
$$ H^1(F,L)\ra H^1(F,A) $$
is injective. 
\end{lem}

\begin{proof}
Suppose that there are two Galois $1$-cocycles $u_\sigma$ and $v_\sigma$ with values in $L$ and of the same image in $H^1(F,A)$, i.e., there exists $g\in A$ such that $v_\sigma=g u_\sigma \sigma(g)^{-1}$ for all $\sigma\in\Gamma$. Let $A'$ be the inner form of $A$ over $F$ given by the Galois action $\Ad(u_\sigma)\circ\sigma$. Let $Q$ be a parabolic $F$-subgroup of $A$ with Levi factor $L$. Since $u_\sigma\in L$, we can use \cite[\S AG, Theorem 14.4]{MR1102012} to see that the pair $(Q,L)$ corresponds to a pair $(Q',L')$ defined over $F$ of $A'$. Similarly, we deduce from $v_\sigma\in L$ and $v_\sigma=g u_\sigma \sigma(g)^{-1}$ that the pair $(\Ad(g^{-1})(Q'), \Ad(g^{-1})(L'))$ of $A'$ is also defined over $F$. Let $(Q'_0,L'_0)$ be a pair formed by a minimal parabolic $F$-subgroup $Q'_0$ of $A'$ and a Levi $F$-factor $L'_0$ of $Q'_0$ such that $Q'_0\subseteq Q'$ and $L'_0\subseteq L'$. There exists $h\in A'(F)$ such that $Q'_0\subseteq\Ad(hg^{-1})(Q')$ and $L'_0\subseteq\Ad(hg^{-1})(L')$. Since two standard and conjugate pairs are equal, using \cite[\S IV.11, Theorem 11.16 and Proposition 11.23.(ii)]{MR1102012}, we obtain $hg^{-1}\in L'$. Write $g=lh$ with $l\in L$. Since $\Ad(u_\sigma)\circ\sigma(h)=h$, we have 
$$ v_\sigma=g u_\sigma \sigma(g)^{-1}=l (h u_\sigma \sigma(h)^{-1}) \sigma(l)^{-1}=l u_\sigma \sigma(l)^{-1} $$
for all $\sigma\in\Gamma$, which implies that $u_\sigma$ and $v_\sigma$ are equal in $H^1(F,L)$. 
\end{proof}

Let $M'\in\msl^{H'}(M'_0)$ and $M\in\msl^{G,\omega}(M_0)$ be a pair of matching Levi subgroups as in Section \ref{ssec:mat_Levi}. Recall that $e_{M'}\in\{\pm1\}^\ell$ is defined by \eqref{eq:def_e_M'}. It follows that $u\in H^1(F,H_0/Z_G)$ (as in Lemma \ref{lem1251}) belongs to the image of the injective map (see Lemma \ref{lem2242}) 
$$ H^1(F, M_{H_0}/Z_G)\hookrightarrow H^1(F,H_0/Z_G). $$ 
Let $u_M\in H^1(F, M_{H_0}/Z_G)$ be the preimage of $u$. Denote by $\bar{u}_M$ the image of $u_M$ under the injective map (since $H^1(F,Z_M/Z_G)=0$ by Hilbert's Theorem 90)
$$ H^1(F,M_{H_0}/Z_G)\hookrightarrow H^1(F,M_{H_0}/Z_M) $$
and by ${\varepsilon'}^M$ the image of $\bar{u}_M$ under the morphism 
$$ H^1(F,M_{H_0}/Z_M)\ra H^1_\ab(F,M_{H_0}/Z_M)\ra H^1_\ab(F, M_{H_0}^{\omega_0}/Z_M\ra M_{H_0}/Z_M). $$
Then $\bar{u}_M\in H^1(F,M_{H_0}/Z_M)$ is the class associated to $\wt{\fm'}$ under a product form of Proposition \ref{prop:bookinvol}. 

\begin{coro}\label{idcar5}
Suppose that $F$ is a non-archimedean local field of characteristic zero. We have ${\varepsilon'}^M=e_{M'}$ under the identification 
$$ H^1_\ab(F, M_{H_0}^{\omega_0}/Z_M\ra M_{H_0}/Z_M)=(\BZ/2\BZ)^\ell\simeq\{\pm1\}^\ell. $$ 
\end{coro}

\begin{proof}
It results from Corollary \ref{idcar4} and Lemma \ref{lem:kottsign}. 
\end{proof}


\section{\textbf{Statement of results}}\label{secsta}

\subsection{}

Let $F$ be a non-archimedean local field of characteristic zero and $E$ be a quadratic extension of $F$. Fix a pair of matching Levi subgroups $M'\in\msl^{H'}(M'_0)$ and $M\in\msl^{G,\omega}(M_0)$ in the sense of Section \ref{ssec:mat_Levi}. Recall that we denote by $\gamma_\psi(\fh(F))$ (resp. $\gamma_\psi(\fh'(F))$) the Weil constant associated to $\fh(F)$ (resp. $\fh'(F)$) and the restriction of \eqref{bilform1} (resp. \eqref{bilform2}). 

\begin{defn}\label{defparMass}
Let $f\in\CC_c^\infty(\fs(F))$ and $f'\in\CC_c^\infty(\fs'(F))$. We say that $f$ and $f'$ are partially $M$-associated if they satisfy the following condition: for all $L\in\msl^G(M)$ and all $Q\in\msf^G(L)$, if $X\in(\fl\cap\fs_\rs)(F)$ and $Y\in(\wt{\fl'}\cap\fs'_\rs)(F)$ have $L$-matching orbits (see Definition \ref{defbyinvM}), then
$$ \kappa(X) J_L^Q(\eta, X, f)=J_{L'}^{Q'}(Y, f'). $$
\end{defn}

\begin{defn}\label{defMass}
Let $f\in\CC_c^\infty(\fs(F))$ and $f'\in\CC_c^\infty(\fs'(F))$. We say that $f$ and $f'$ are $M$-associated if they are partially $M$-associated and satisfy the additional condition: for all $L\in\msl^G(M)$, $Q\in\msf^G(L)$ and $X\in(\fl\cap\fs_\rs)(F)$, we have 
$$ J_L^Q(\eta, X, f)=0 $$
unless $X$ comes potentially from $(\fm_{\wt{Q'}}\cap\fs'_\rs)(F)$ (see Definition \ref{defcomepotfrom}). 
\end{defn}

\begin{thm}\label{thmcommute}
Let $f\in\CC_c^\infty(\fs(F))$ and $f'\in\CC_c^\infty(\fs'(F))$ be partially $M$-associated and satisfy the following conditions. 
\begin{enumerate}[(a)]
	\item The weighted orbital integrals of $f$ and $f'$ vanish for nontrivial weights (see Section \ref{ssec:lab_lem}). 

	\item If $X\in\fs_\rs(F)$ does not come from $\fs'_\rs(F)$ (see Definition \ref{defbyinv}), then 
	$$ J_G^G(\eta,X,f)=0. $$
\end{enumerate}
Then $\gamma_{\psi}(\fh(F))^{-1} \hat{f}$ and $\gamma_{\psi}(\fh'(F))^{-1} \hat{f'}$ are $M$-associated. 
\end{thm}

\begin{remark}
The two conditions in the above theorem have the potential to be weakened, but they are enough for our purpose. 
\end{remark}

\begin{proof}[Proof of Theorem \ref{thmcommute}]
It suffices to combine Proposition \ref{parcommute}, Lemma \ref{lem:2implyadd} and Corollary \ref{corvancommute} below. 
\end{proof}

\begin{lem}\label{lem:2implyadd}
If $f\in\CC_c^\infty(\fs(F))$ satisfies the two conditions in Theorem \ref{thmcommute} (ignore $f'$ in the condition (a)), then it satisfies the additional condition in Definition \ref{defMass}. 
\end{lem}

\begin{proof}
Let $L\in\msl^G(M)$ and $X\in(\fl\cap\fs_\rs)(F)$. For $Q\in\msf^G(L)-\msp^G(L)$, we have $J_L^Q(\eta, X, f)=0$ by the condition (a). Now suppose $Q\in\msp^G(L)$. Then $v_L^Q=1$ and thus $J_L^Q(\eta, X, f)=J_G^G(\eta, X, f)$. If $X$ does not come potentially from $(\fm_{\wt{Q'}}\cap\fs'_\rs)(F)$, it follows from Corollary \ref{corequivdefcomefrom} and Proposition \ref{prop:equivcomefromLevi} that $X$ does not come from $\fs'_\rs(F)$. Then $J_L^Q(\eta, X, f)=J_G^G(\eta, X, f)=0$ by the condition (b). 
\end{proof}

\subsection{Identities between $\hat{i}_M^G$ and $\hat{i}_{M'}^{H'}$}

\begin{prop}\label{parcommute}
Let $f\in\CC_c^\infty(\fs(F))$ and $f'\in\CC_c^\infty(\fs'(F))$ be partially $M$-associated and satisfy the two conditions in Theorem \ref{thmcommute}. Then $\gamma_{\psi}(\fh(F))^{-1} \hat{f}$ and $\gamma_{\psi}(\fh'(F))^{-1} \hat{f'}$ are also partially $M$-associated. 
\end{prop}

The rest of this paper will be denoted to the proof of Proposition \ref{parcommute}. Its corollary below may be more useful for applications. 

\begin{coro}\label{parcommutecor1}
Let $X\in(\fm\cap\fs_\rs)(F)$ and $Y\in(\wt{\fm'}\cap\fs'_\rs)(F)$ be such that $X\overset{M}{\da} Y$. Let $U\in\fs_\rs(F)$ and $V\in\fs'_\rs(F)$ be such that $U\da V$. Then we have the equality
$$ \gamma_{\psi}(\fh(F))^{-1} \kappa(X)\kappa(U)\hat{i}_M^G(\eta, X, U)=\gamma_{\psi}(\fh'(F))^{-1} \hat{i}_{M'}^{H'}(Y, V). $$ 
\end{coro}

\begin{proof}
By Lemma \ref{lem74} applied to $V$ and $U$, we have an isomorphism $\varphi: \fs'_V(F)\ra\fs_U(F)$ such that $\varphi(V)=U$. Recall that $W(H,\fs_U)$ (resp. $W(H',\fs'_V)$) denotes the Weyl group associated to $\fs_U\in\mst^\fs$ (resp. $\fs'_V\in\mst^{\fs'}$) defined by \eqref{eq:defweylcartan}. Choose open compact neighbourhoods $\omega$ of $U$ in $(\fs_U\cap\fs_\rs)(F)$ and $\omega'$ of $V$ in $(\fs'_V\cap\fs'_\rs)(F)$ which are small enough such that 
\begin{enumerate}[(i)]
	\item the sets $i(\omega)$ where $i\in W(H, \fs_U)$ are mutually disjoint; 

	\item the sets $i'(\omega')$ where $i'\in W(H', \fs'_V)$ are mutually disjoint; 

	\item $\varphi(\omega')=\omega$; 

	\item $\hat{i}_M^G(\eta, X, \cdot)$ is constant on $\omega$; 

	\item $\hat{i}_{M'}^{H'}(Y, \cdot)$ is constant on $\omega'$; 
	
	\item $\kappa(\cdot)$ is constant on $\omega$. 
\end{enumerate}
Notice that the conditions (iv) and (v) are assured by \cite[Lemmas 8.3.(1) and 8.8.(1)]{MR4681295}. 

By the conditions (i) and (vi) on $\omega$ and Lemma \ref{lemlabesse}.(1), we can construct a function $f\in\CC_c^\infty(\fs(F))$ such that 
\begin{enumerate}[(i)]
	\item $\Supp(f)\subseteq\Ad(H(F))(\omega)$; 

	\item for all $Z\in\omega$, $\kappa(Z)J_G^G(\eta, Z, f)=1$; 

	\item the weighted orbital integrals of $f$ vanish for nontrivial weights. 
\end{enumerate}

By the condition (ii) on $\omega'$ and Lemma \ref{lemlabesse}.(2), we can construct a function $f'\in\CC_c^\infty(\fs'(F))$ such that 
\begin{enumerate}[(i)]
	\item $\Supp(f')\subseteq\Ad(H'(F))(\omega')$; 

	\item for all $Z'\in\omega'$, $J_{H'}^{H'}(Z', f')=1$; 

	\item the weighted orbital integrals of $f'$ vanish for nontrivial weights. 
\end{enumerate}

We see that $f$ and $f'$ are partially $M$-associated and satisfy the two conditions in Theorem \ref{thmcommute}. Thus by Proposition \ref{parcommute}, we have the equality 
$$ \gamma_{\psi}(\fh(F))^{-1} \kappa(X) J_M^G(\eta, X, \hat{f})=\gamma_{\psi}(\fh'(F))^{-1} J_{M'}^{H'}(Y, \hat{f'}). $$
By the condition (iii) on $f$ and the Weyl integration formula \cite[(7.1.2)]{MR4681295}, we have 
$$ J_M^G(\eta, X, \hat{f})=I_M^G(\eta, X, \hat{f})=\sum_{\fc\in\mst_0^\fs} |W(H,\fc)|^{-1}\int_{\fc_\reg(F)} J_G^G(\eta, Z, f) \hat{i}_M^G(\eta, X, Z) dZ. $$
By the conditions (i) and (ii) on $f$ and the condition (iv) on $\omega$, the last expression equals 
$$ \int_{\omega} J_G^G(\eta, Z, f) \hat{i}_M^G(\eta, X, Z) dZ=\vol(\omega)\kappa(U)\hat{i}_M^G(\eta, X, U). $$
Similarly, using the conditions on $f'$, \cite[(7.2.2)]{MR4681295} and the condition (iv) on $\omega'$, we obtain 
$$ J_{M'}^{H'}(Y, \hat{f'})=I_{M'}^{H'}(Y, \hat{f'})=\vol(\omega')\hat{i}_{M'}^{H'}(Y, V). $$ 
Since $\vol(\omega)=\vol(\omega')$ by the condition $\varphi(\omega')=\omega$ and our choice of Haar measures in Section \ref{ssec:centrss}, we deduce the desired equality. 
\end{proof}


\subsection{A vanishing statement}

Recall that $\omega=\mat(0,1_n,1_n,0)\in G(F)$. For $X\in\fs(F)$, denote 
$$ X^\omega:=\Ad(\omega)(X)\in\fs(F). $$
For $f\in\CC_c^\infty(\fs(F))$, define a function $f^\omega\in\CC_c^\infty(\fs(F))$ by 
$$ f^\omega(X):=f(X^\omega), \forall X\in\fs(F). $$ 

\begin{lem}\label{lempf8.3}
Let $f\in\CC_c^\infty(\fs(F))$ and $X\in(\fm\cap\fs_\rs)(F)$. Then we have 
\begin{enumerate}
	\item $(\hat{f})^\omega=\wh{f^\omega}$; 
	\item $J_M^G(\eta, X, f^\omega)=\eta(X)J_M^G(\eta, X, f)$. 
\end{enumerate}
\end{lem}

\begin{proof}
Analogous properties have been used in the proof of \cite[Lemma 8.3]{MR3414387} though our involutions are slightly different. To check (1), it suffices to use the $\omega$-invariance of $\langle\cdot,\cdot\rangle$. To check (2), we may suppose that $X=\mat(0,1_n,A,0)$ and use the fact that 
$$ \Ad(\omega)\mat(0,1_n,A,0)=\Ad\mat(1,,,A^{-1}) (X). $$
The extra ingredient compared to {\it{loc. cit.}} is the fact that $v_M^G(\Ad(w)(x))=v_M^G(x)$ for $x\in H(F)$. 
\end{proof}

\begin{lem}\label{iinvlem}
Let $X\in(\fm\cap\fs_\rs)(F)$ and $U\in\fs_\rs(F)$. Then we have the equality 
$$ \hat{i}_M^G(\eta,X,U^\omega)=\eta(X)\hat{i}_M^G(\eta,X,U). $$
\end{lem}

\begin{proof}
From Lemma \ref{lempf8.3}, we deduce that 
\begin{equation}\label{jinvfor}
 \hat{j}_M^G(\eta,X,U^\omega)=\eta(X)\hat{j}_M^G(\eta,X,U). 
\end{equation}
There exists $x\in H(F)$, $L\in\msl^{G,\omega}(M_0)$ and $Z\in(\fl\cap\fs_\rs)(F)_\el$ such that $U=\Ad(x)(Z)$. By \cite[Lemmas 8.3.(2) and 8.2]{MR4681295}, we have 
$$ \eta(X)\hat{i}_M^G(\eta,X,U)=\eta(\det(x))\eta(X)\hat{i}_M^G(\eta,X,Z)=\eta(\det(x))\eta(X)\hat{j}_M^G(\eta,X,Z). $$
Applying \eqref{jinvfor} to $X$ and $Z$, we have   
$$ \eta(X)\hat{j}_M^G(\eta,X,Z)=\hat{j}_M^G(\eta,X,Z^\omega). $$
Since $Z^\omega\in(\fl\cap\fs_\rs)(F)_\el$, by \cite[Lemma 8.2]{MR4681295} again, we obtain 
$$ \hat{j}_M^G(\eta,X,Z^\omega)=\hat{i}_M^G(\eta,X,Z^\omega). $$
Thus 
$$ \eta(X)\hat{i}_M^G(\eta,X,U)=\eta(\det(x))\hat{i}_M^G(\eta,X,Z^\omega). $$
We see that $U^\omega=\Ad(\omega x\omega^{-1})(Z^\omega)$, where $\omega x\omega^{-1}\in H(F)$. By \cite[Lemma 8.3.(2)]{MR4681295} again, we have 
$$ \hat{i}_M^G(\eta,X,U^\omega)=\eta(\det(\omega x\omega^{-1}))\hat{i}_M^G(\eta,X,Z^\omega)=\eta(\det(x))\hat{i}_M^G(\eta,X,Z^\omega). $$
Then the lemma follows. 
\end{proof}

\begin{prop}\label{vancommute}
Let $X\in(\fm\cap\fs_\rs)(F)$ and $U\in\fs_\rs(F)$. If $\eta(X)\neq\eta(U)$, then 
$$ \hat{i}_M^G(\eta, X, U)=0. $$
\end{prop}

\begin{proof}
Suppose that $U=\mat(0,U_1,U_2,0)$. Then $\omega U=\mat(U_2,,,U_1)\in H(F)$. Notice that 
$$ U^\omega=\mat(,U_2,U_1,)=\Ad(\omega U)(U). $$
By \cite[Lemma 8.3.(2)]{MR4681295}, we have 
$$ \hat{i}_M^G(\eta,X,U^\omega)=\eta(\det(\omega U))\hat{i}_M^G(\eta,X,U)=\eta(U)\hat{i}_M^G(\eta,X,U). $$
One may conclude by comparing this equality with Lemma \ref{iinvlem}. 
\end{proof}

\begin{remark}\label{rmkvancommute}
By the same argument, we can generalise the above proposition to the following product form. Let $L\in\msl^G(M)$, $X\in(\fm\cap\fs_\rs)(F)$ and $U\in(\fl\cap\fs)_\rs(F)$. If $\eta_L(X)\neq\eta_L(U)$, then 
$$ \hat{i}_M^L(\eta, X, U)=0. $$
\end{remark}

\begin{coro}\label{corvancommute}
If $f\in\CC_c^\infty(\fs(F))$ satisfies the additional condition in Definition \ref{defMass}, then $\hat{f}$ also satisfies this condition. 
\end{coro}

\begin{proof}
To begin with, we observe that it suffices to show the following assertion: for all $Q\in\msf^G(M)$ and $X\in(\fm\cap\fs_\rs)(F)$, we have 
$$ J_M^Q(\eta, X, \hat{f})=0 $$
unless $X$ comes potentially from $(\fm_{\wt{Q'}}\cap\fs'_\rs)(F)$. In fact, once we prove this, we can conclude by replacing $M$ in the assertion with any $L\in\msl^G(M)$. 

Now we show the above assertion. By \cite[Proposition 4.1.(4) and (8.1.1)]{MR4681295}, we obtain 
$$ J_M^Q(\eta, X, \hat{f})=J_M^{M_Q}(\eta, X, \hat{f}_Q^\eta)=\sum_{L\in\msl^{M_Q}(M)} \hat{I}_M^{L,M_Q,w}(\eta,X,f_Q^\eta), $$
where 
\[\begin{split}
	\hat{I}_M^{L,M_Q,w}(\eta,X,f_Q^\eta)=&\sum_{\{R\in\msl^{G,\omega}(M_0): R\subseteq L\}} |W_0^{R_n}||W_0^{L_n}|^{-1} \sum_{\fc\in\mst_\el^{\fr\cap\fs}} |W(R_H,\fc)|^{-1} \\
	&\int_{\fc_\reg(F)} J_L^{M_Q}(\eta,Z,f_Q^\eta) \hat{i}_M^L(\eta,X,Z) dZ. \\
\end{split}\]
By \cite[Proposition 4.1.(4)]{MR4681295} again, we have 
$$ J_L^{M_Q}(\eta,Z,f_Q^\eta)=J_L^Q(\eta,Z,f). $$
If $J_M^Q(\eta, X, \hat{f})\neq0$, then 
$$ J_L^Q(\eta,Z,f) \hat{i}_M^L(\eta,X,Z)\neq0 $$ 
for some $Z$. Since $J_L^Q(\eta,Z,f)\neq0$, by our assumption on $f$, we see that $Z$ comes potentially from $(\fm_{\wt{Q'}}\cap\fs'_\rs)(F)$. Since $\hat{i}_M^L(\eta,X,Z)\neq0$, by Remark \ref{rmkvancommute}, we have $\eta_L(X)=\eta_L(Z)$. As $L\subseteq M_Q$, it implies that $\eta_{M_Q}(X)=\eta_{M_Q}(Z)$. Thus $X$ also comes potentially from $(\fm_{\wt{Q'}}\cap\fs'_\rs)(F)$. 
\end{proof}



\section{\textbf{Construction of test functions}}

Let $F$ be a non-archimedean local field of characteristic zero and $E$ be a quadratic extension of $F$. Denote by $v_F(\cdot)$ the normalised valuation on $F$. Fix a uniformiser $\varpi$ of $\CO_F$. 

\subsection{Limit formulae}\label{seclim}

Recall that for $\fc\in\mst^\fs$ and $X,U\in(\fc\cap\fs_\rs)(F)$, we define $\gamma_\psi(X,U)$ by \eqref{gammaXY}. Similarly, for $\fc'\in\mst^{\fs'}$ and $Y,V\in(\fc'\cap\fs'_\rs)(F)$, we define $\gamma_\psi(Y,V)$ by the same formula. 

Let $M\in\msl^{G, \omega}(M_0)$. For all $L\in\msl^G(M), X\in(\fm\cap\fs_\rs)(F)$ and $U\in(\fl\cap\fs_\rs)(F)$, we define $\hat{a}_M^L(\eta, X, U):=0$ if $L\neq M$ and  
$$ \hat{a}_M^M(\eta, X, U):=\sum_{x\in (M_H)_U(F)\bs M_H(F), \Ad(x)(X)\in\fs_U(F)} \eta(\det(x)) \gamma_\psi(\Ad(x)(X),U)\psi(\langle\Ad(x)(X),U\rangle). $$

Let $M'\in\msl^{H'}(M'_0)$. For all $L'\in\msl^{H'}(M'), Y\in(\wt{\fm'}\cap\fs'_\rs)(F)$ and $V\in(\wt{\fl'}\cap\fs'_\rs)(F)$, we define $\hat{a}_{M'}^{L'}(Y, V):=0$ if $L'\neq M'$ and 
$$ \hat{a}_{M'}^{M'}(Y, V):=\sum_{x\in M'_V(F)\bs M'(F), \Ad(x)(Y)\in\fs'_V(F)} \gamma_\psi(\Ad(x)(Y),V)\psi(\langle\Ad(x)(Y),V\rangle). $$

\begin{prop}\label{proplimfor}
\leavevmode
\begin{enumerate}
	\item Let $M\in\msl^{G, \omega}(M_0)$, $L\in\msl^G(M), X\in(\fm\cap\fs_\rs)(F)$ and $U\in(\fl\cap\fs_\rs)(F)$. Then there exists $N\in\BZ_{\geq0}$ such that if $\mu\in F^\times$ satisfies $v_F(\mu)<-N$, we have the equality 
$$ \hat{i}_M^L(\eta, \mu X, U)=\hat{i}_M^L(\eta, X, \mu U)=\hat{a}_M^L(\eta, \mu X, U)=\hat{a}_M^L(\eta, X, \mu U). $$

	\item Let $M'\in\msl^{H'}(M'_0)$, $L'\in\msl^{H'}(M'), Y\in(\wt{\fm'}\cap\fs'_\rs)(F)$ and $V\in(\wt{\fl'}\cap\fs'_\rs)(F)$. Then there exists $N\in\BZ_{\geq0}$ such that if $\mu\in F^\times$ satisfies $v_F(\mu)<-N$, we have the equality 
$$ \hat{i}_{M'}^{L'}(\mu Y, V)=\hat{i}_{M'}^{L'}(Y, \mu V)=\hat{a}_{M'}^{L'}(\mu Y, V)=\hat{a}_{M'}^{L'}(Y, \mu V). $$
\end{enumerate}
\end{prop}

\begin{proof}
For $L\neq M$ and $L'\neq M'$, this is a product form of \cite[Lemmas 8.3.(3) and 8.8.(3) and Propositions 10.1 and 10.4]{MR4681295}. For $L=M$, (1) is a product form of \cite[Proposition 7.1.(i)]{MR3414387}. For $L'=M'$, (2) is a generalisation of \cite[Proposition 7.1.(ii)]{MR3414387} to the setting of \cite{MR2806111}, which can be proved in a similar way with the preparation in \cite{MR4681295}. 
\end{proof}

\subsection{Test functions}\label{secconsr}

\begin{lem}\label{lem75}
Let $X\in\fs_\rs(F)$ and $Y\in\fs'_\rs(F)$ be such that $X\da Y$. Let $y\in Hv$ be such that $\Ad(y)\circ\varphi(Y)=X$ (see Section \ref{ssec:centrss}). Let $V\in(\fs'_Y\cap\fs'_\rs)(F)$ and $U:=\Ad(y)\circ\varphi(V)$. Then we have

\begin{enumerate}
	\item $\langle X,U\rangle=\langle Y,V\rangle$; 

	\item $\gamma_\psi(\fh(F))^{-1}\gamma_\psi(X,U)=\gamma_\psi(\fh'(F))^{-1}\gamma_\psi(Y,V)$. 
\end{enumerate}
\end{lem}

\begin{proof}
This is a generalisation of \cite[Lemma 7.5]{MR3414387}. The first assertion results from Lemma \ref{lem74} because $\langle\cdot,\cdot\rangle$ is defined in the same way over $\ov{F}$ for both $\fg$ and $\fg'$ and is invariant under the adjoint action. The second assertion follows from Lemma \ref{lem74} and \cite[Proposition 7.1]{MR3414387}, which applies to a general symmetric pair, by the argument of \cite[(6) in p. 96]{MR1344131}. 
\end{proof}

\begin{prop}\label{testfunctions}
Fix a pair of matching Levi subgroups $M'\in\msl^{H'}(M'_0)$ and $M\in\msl^{G,\omega}(M_0)$ in the sense of Section \ref{ssec:mat_Levi}. Let $X_0\in(\fm\cap\fs_\rs)(F)$ and $Y_0\in(\wt{\fm'}\cap\fs'_\rs)(F)$ be such that $X_0\da Y_0$. 
Then there exist functions $f\in\CC_c^\infty(\fs(F))$ and $f'\in\CC_c^\infty(\fs'(F))$ satisfying the following conditions. 

\begin{enumerate}[(a)]
	\item If $X\in\Supp(f)$, there exists $Y\in(\fs'_{Y_0}\cap\fs'_\rs)(F)$ such that $X\da Y$. 

	\item If $Y\in\Supp(f')$, then $Y$ is $H'(F)$-conjugate to an element in $(\fs'_{Y_0}\cap\fs'_\rs)(F)$. 

	\item The weighted orbital integrals of $f$ and $f'$ vanish for nontrivial weights. 

	\item The functions $f$ and $f'$ are partially $G$-associated and satisfy the condition: if $X\in\fs_\rs(F)$ does not come from $\fs'_\rs(F)$, then 
	$$ J_G^G(\eta,X,f)=0. $$ 

	\item For $Q\in\msf^G(M)-\msp^G(M)$, 
	$$ J_M^Q(\eta, X_0, \hat{f})=J_{M'}^{Q'}(Y_0, \hat{f'})=0. $$

	\item We have the equality 
	$$ \gamma_{\psi}(\fh(F))^{-1} \kappa(X_0) J_G^G(\eta, X_0, \hat{f})=\gamma_{\psi}(\fh'(F))^{-1} J_{H'}^{H'}(Y_0, \hat{f'})\neq 0. $$
\end{enumerate}
\end{prop}

\begin{remark}
This is a generalisation of \cite[Proposition 7.6]{MR3414387}, where the conditions (a), (b), (d) and (f) have already appeared. In fact, the condition (d) means exactly that the functions $f$ and $f'$ are smooth transfers of each other in the sense of \cite[Definition 5.10.(ii)]{MR3414387}. Here the additional ingredients are Lemma \ref{lemlabesse} and Proposition \ref{proplimfor} (cf. \cite[\S6]{MR2164623}). 
\end{remark}

\begin{proof}[Proof of Proposition \ref{testfunctions}]
By Lemma \ref{lem74} applied to $Y_0$ and $X_0$, we have an isomorphism $\varphi: \fs'_{Y_0}(F)\ra\fs_{X_0}(F)$ such that $\varphi(Y_0)=X_0$. Choose $V_0\in(\fs'_{Y_0}\cap\fs'_\rs)(F)$ such that if we denote $U_0:=\varphi(V_0)\in(\fs_{X_0}\cap\fs_\rs)(F)$, then 
\begin{enumerate}[(i)]
	\item for all $i\in W(H,\fs_{X_0}), i\neq1$, we have $\langle i(X_0)-X_0, U_0\rangle\neq0$; 

	\item for all $i'\in W(H',\fs'_{Y_0}), i'\neq1$, we have $\langle i'(Y_0)-Y_0, V_0\rangle\neq0$; 

	\item $\kappa(U_0)=\kappa(X_0)$. 
\end{enumerate}

Fix $r\in\BZ_{\geq1}$ such that 
\begin{enumerate}[(i)]
	\item $1+\varpi^r\CO_F\subseteq {F^\times}^2$; 

	\item the sets $i((1+\varpi^r\CO_F)U_0)$ where $i\in W(H, \fs_{X_0})$ are mutually disjoint; 

	\item the sets $i'((1+\varpi^r\CO_F)V_0)$ where $i'\in W(H', \fs'_{Y_0})$ are mutually disjoint. 
\end{enumerate}

By Propositions \ref{proplimfor}, there exists $N\in\BZ_{\geq0}$ such that if $\mu\in F^\times$ satisfies $v_F(\mu)<-N$, then 
\begin{enumerate}[(i)]
	\item for $L_1,L_2\in\msl^G(M)$ with $L_1\subseteq L_2$, $i_1\in W(L_{1,H}, \fs_{X_0}, \fl_1\cap\fs)$ and $i_2\in W(H, \fs_{X_0}, \fl_2\cap\fs)$ (see \eqref{eq:weylcarlevi1} for the notations), we have
	$$ \hat{i}_{L_1}^{L_2}(\eta, i_1(X_0), i_2(\mu U_0))=\hat{a}_{L_1}^{L_2}(\eta, i_1(X_0), i_2(\mu U_0)); $$

	\item for $L'_1,L'_2\in\msl^{H'}(M')$ with $L'_1\subseteq L'_2$, $i'_1\in W(L'_1, \fs'_{Y_0}, \wt{\fl'_1}\cap\fs')$ and $i'_2\in W(H', \fs'_{Y_0}, \wt{\fl'_2}\cap\fs')$ (see \eqref{eq:weylcarlevi2} for the notations), we have
	$$ \hat{i}_{L'_1}^{L'_2}(i'_1(Y_0), i'_2(\mu V_0))=\hat{a}_{L'_1}^{L'_2}(i'_1(Y_0), i'_2(\mu V_0)). $$
\end{enumerate}
Fix such an integer $N$. 

Fix $\mu\in F^\times$ with $v_F(\mu)<-N$ such that
\begin{enumerate}[(i)]
	\item $\eta(\mu)=1$; 

	\item for all $i\in W(H,\fs_{X_0}), i\neq1$, the character $\lambda\mapsto\psi(\varpi^r\mu\lambda\langle i(X_0)-X_0, U_0\rangle)$ is nontrivial on $\CO_F$; 

	\item for all $i'\in W(H',\fs'_{Y_0}), i'\neq1$, the character $\lambda\mapsto\psi(\varpi^r\mu\lambda\langle i'(Y_0)-Y_0, V_0\rangle)$ is nontrivial on $\CO_F$. 
\end{enumerate}
Notice that the conditions (ii) and (iii) are possible because of the conditions (i) and (ii) on $U_0$ and $V_0$. 

We set 
$$ \omega'_0:=\mu(1+\varpi^r\CO_F)V_0. $$ 
Denote by $\fd'$ the $F$-vector space generated by $V_0$. Fix a complement $\fe'$ of $\fd'$ in $\fs'_{Y_0}(F)$. For $V\in\fs'_{Y_0}(F)$, denote by $V_{\fd'}$ its projection to $\fd'$ with respect to the decomposition $\fs'_{Y_0}(F)=\fd'\oplus\fe'$. 

We set 
$$ \omega_0:=\mu(1+\varpi^r\CO_F)U_0. $$
Denote by $\fd:=\varphi(\fd')$ the $F$-vector space generated by $U_0$. Let $\fe:=\varphi(\fe')$ be the complement of $\fd$ in $\fs_{X_0}(F)$. For $U\in\fs_{X_0}(F)$, denote by $U_\fd$ its projection to $\fd$ with respect to the decomposition $\fs_{X_0}(F)=\fd\oplus\fe$. 

Choose open compact neighbourhoods $\omega_\fe$ of $0$ in $\fe$ and $\omega_{\fe'}$ of $0$ in $\fe'$ which are small enough such that $\omega:=\omega_0\oplus\omega_\fe$ and $\omega':=\omega'_0\oplus\omega'_{\fe'}$ satisfy 
\begin{enumerate}[(i)]
	\item the sets $i(\omega)$ where $i\in W(H, \fs_{X_0})$ are mutually disjoint; 

	\item the sets $i'(\omega')$ where $i'\in W(H', \fs'_{Y_0})$ are mutually disjoint; 

	\item $\omega\subseteq(\fs_{X_0}\cap\fs_\rs)(F)$, $\omega'\subseteq(\fs'_{Y_0}\cap\fs'_\rs)(F)$ and $\varphi(\omega')=\omega$; 

	\item for $L_1,L_2\in\msl^G(M)$ with $L_1\subseteq L_2$, $i_1\in W(L_{1,H}, \fs_{X_0}, \fl_1\cap\fs), i_2\in W(H, \fs_{X_0}, \fl_2\cap\fs)$ and $U\in\omega$, we have 
	$$ \hat{i}_{L_1}^{L_2}(\eta, i_1(X_0), i_2(U))=\hat{i}_{L_1}^{L_2}(\eta, i_1(X_0), i_2(U_\fd))=\hat{a}_{L_1}^{L_2}(\eta, i_1(X_0), i_2(U_\fd)); $$

	\item for $L'_1,L'_2\in\msl^{H'}(M')$ with $L'_1\subseteq L'_2$, $i'_1\in W(L'_1, \fs'_{Y_0}, \wt{\fl'_1}\cap\fs'), i'_2\in W(H', \fs'_{Y_0}, \wt{\fl'_2}\cap\fs')$ and $V\in\omega'$, we have
	$$ \hat{i}_{L'_1}^{L'_2}(i'_1(Y_0), i'_2(V))=\hat{i}_{L'_1}^{L'_2}(i'_1(Y_0), i'_2(V_{\fd'}))=\hat{a}_{L'_1}^{L'_2}(i'_1(Y_0), i'_2(V_{\fd'})); $$

	\item the function $\kappa$ is constant on $\omega$. 
\end{enumerate}
Notice that the conditions (i) and (ii) follow from the conditions (ii) and (iii) on $r$. Besides, the conditions (iv) and (v) are assured by the local constancy of $\hat{i}_{L_1}^{L_2}$ and $\hat{i}_{L'_1}^{L'_2}$, $v_F(\mu)<-N$ and $r\geq 1$. Moreover, the condition (vi) results from the condition (i) on $\mu$ and the condition (i) on $r$. Combined with the condition (iii) on $U_0$, the condition (vi) says that the restriction of $\kappa$ to $\omega$ equals $\kappa(X_0)$. 

Define a function $f_\omega$ on $\omega$ by 
\begin{equation}\label{eq:def_f_omega}
 f_\omega(U):=\psi(-\langle X_0, U_\fd\rangle) 
\end{equation}
for all $U\in\omega$. By the conditions (i) and (vi) on $\omega$ and Lemma \ref{lemlabesse}.(1), we can construct a function $f\in\CC_c^\infty(\fs(F))$ such that 
\begin{enumerate}[(i)]
	\item $\Supp(f)\subseteq\Ad(H(F))(\omega)$; 

	\item for all $U\in\omega$, $\kappa(U)J_G^G(\eta, U, f)=f_\omega(U)$; 

	\item the weighted orbital integrals of $f$ vanish for nontrivial weights. 
\end{enumerate}

Define a function $f'_{\omega'}$ on $\omega'$ by 
$$ f'_{\omega'}(V):=\psi(-\langle Y_0, V_{\fd'}\rangle) $$
for all $V\in\omega'$. By the condition (ii) on $\omega'$ and Lemma \ref{lemlabesse}.(2), we can construct a function $f'\in\CC_c^\infty(\fs'(F))$ such that 
\begin{enumerate}[(i)]
	\item $\Supp(f')\subseteq\Ad(H'(F))(\omega')$; 

	\item for all $V\in\omega'$, $J_{H'}^{H'}(V, f')=f'_{\omega'}(V)$; 

	\item the weighted orbital integrals of $f'$ vanish for nontrivial weights. 
\end{enumerate}

We shall check that $f$ and $f'$ satisfy the conditions in the proposition. The conditions (a) and (b) result from the condition (iii) on $\omega$ and $\omega'$, the condition (i) on $f$ and the condition (i) on $f'$. The condition (c) is exactly the condition (iii) on $f$ and the condition (iii) on $f'$. The condition (d) is deduced from the condition (iii) on $\omega$ and $\omega'$, the conditions (i) and (ii) on $f$, the conditions (i) and (ii) on $f'$ and Lemma \ref{lem75}.(1). 

We now verify the condition (e). By \cite[Proposition 4.1.(4)]{MR4681295}, we write 
$$ J_M^Q(\eta, X_0, \hat{f})=J_M^{M_Q}(\eta, X_0, \hat{f}_Q^\eta). $$
Because of the condition (iii) on $f$ and Lemma \ref{lem:Lab_I.6.4}.(1), it follows from the definition \cite[(8.1.1)]{MR4681295} that 
$$ J_M^{M_Q}(\eta, X_0, \hat{f}_Q^\eta)=I_M^{M_Q}(\eta, X_0, \hat{f}_Q^\eta). $$
From the Weyl integration formula \cite[(7.1.2)]{MR4681295}, we deduce that
$$ I_M^{M_Q}(\eta, X_0, \hat{f}_Q^\eta)=\sum_{\fc\in\mst_0^{\fm_Q\cap\fs}} |W(M_{Q_H},\fc)|^{-1}\int_{\fc_\reg(F)} J_{M_Q}^{M_Q}(\eta, Z, f_Q^\eta) \hat{i}_M^{M_Q}(\eta, X_0, Z) dZ. $$
Consider $\fc\in\mst_0^{\fm_Q\cap\fs}$ and $Z\in\fc_\reg(F)$. Using \cite[Proposition 4.1.(4)]{MR4681295} again, we obtain 
$$ J_{M_Q}^{M_Q}(\eta, Z, f_Q^\eta)=J_{M_Q}^Q(\eta, Z, f)=J_G^G(\eta, Z, f). $$
Suppose that $J_G^G(\eta, Z, f)\neq0$. By the condition (i) on $f$, there exists $x\in H(F)$ such that $\Ad(x^{-1})(Z)\in\omega$. For such an $x$, we have 
$$ Z\in\Ad(x)(\omega)\subseteq\Ad(x)(\fs_{X_0})=\fc\subseteq\fm_Q\cap\fs. $$
By the condition (iv) on $\omega$ applied to $L_1=M$ and $L_2=M_Q$ as well as \cite[Lemma 8.3.(2)]{MR4681295}, if $M_Q\neq M$, we get $\hat{i}_M^{M_Q}(\eta, X_0, Z)=0$ and thus 
$$ J_M^Q(\eta, X_0, \hat{f})=0. $$
Similarly, if $M_{Q'}\neq M'$, we prove that $J_{M'}^{Q'}(Y_0, \hat{f'})=0$ with the help of \cite[Proposition 4.4.(4), (8.2.1), (7.2.2) and Lemma 8.8.(2)]{MR4681295}, Lemma \ref{lem:Lab_I.6.4}.(2), the conditions (iii) and (i) on $f'$ as well as the condition (v) on $\omega'$. 

The proof of the condition (f) is similar to the last paragraph of the proof of \cite[Proposition 7.6]{MR3414387} with the preparation in \cite{MR4681295}. We reproduce it here for completeness. By the Weyl integration formula \cite[(7.1.2)]{MR4681295} and the condition (i) on $f$, we have 
\[\begin{split}
 \kappa(X_0) J_G^G(\eta, X_0, \hat{f})&=\kappa(X_0) \sum_{\fc\in\mst_0^\fs} |W(H,\fc)|^{-1}\int_{\fc_\reg(F)} J_G^G(\eta, Z, f) \hat{i}_G^G(\eta, X_0, Z) dZ \\ 
 &=\kappa(X_0) |W(H,\fs_{X_0})|^{-1}\int_{(\fs_{X_0}\cap\fs_\rs)(F)} J_G^G(\eta, U, f) \hat{i}_G^G(\eta, X_0, U) dU. \\
\end{split}\]
By the condition (i) on $\omega$ and the condition (ii) on $f$, the last expression is equal to 
\begin{equation}\label{eq:1867}
 \kappa(X_0) \int_{\omega} \kappa(U) f_\omega(U) \hat{i}_G^G(\eta, X_0, U) dU. 
\end{equation}
Consider $U\in\omega$. By the condition (vi) on $\omega$ and the definition \eqref{eq:def_f_omega} of $f_\omega$, we have 
$$ \kappa(U) f_\omega(U)=\kappa(X_0)\psi(-\langle X_0, U_\fd\rangle). $$
By the condition (iv) on $\omega$ applied to $L_1=L_2=G$, we have 
$$ \hat{i}_G^G(\eta, X_0, U)=\hat{a}_G^G(\eta, X_0, U_\fd)=\sum_{i\in W(H,\fs_{X_0})} \eta(\det(i)) \gamma_\psi(i(X_0), U_\fd) \psi(\langle i(X_0), U_\fd\rangle). $$
It follows that \eqref{eq:1867} equals 
\begin{equation}\label{eq:1868}
 \sum_{i\in W(H,\fs_{X_0})} \eta(\det(i)) \vol(\omega_\fe) \int_{\omega_0} \gamma_\psi(i(X_0), U_\fd) \psi(\langle i(X_0)-X_0, U_\fd\rangle) dU_\fd. 
\end{equation}
For $U_\fd\in\omega_0=\mu(1+\varpi^r\CO_F)U_0$, from the condition (i) on $r$ and \cite[\S 25, Proposition 3]{MR0165033}, we see that 
$$ \gamma_\psi(i(X_0), U_\fd)=\gamma_\psi(i(X_0), \mu U_0). $$
By the condition (ii) on $\mu$, for $i\neq1$, 
$$ \int_{\omega_0} \psi(\langle i(X_0)-X_0, U_\fd\rangle) dU_\fd=\int_{\CO_F} \psi(\langle i(X_0)-X_0, \mu(1+\varpi^r\lambda)U_0\rangle) d\lambda=0. $$
Thus \eqref{eq:1868} is equal to 
$$  \vol(\omega_\fe) \gamma_\psi(X_0, \mu U_0) \vol(\omega_0)=\vol(\omega) \gamma_\psi(X_0, \mu U_0). $$
To sum up, we have 
\begin{equation}\label{eq:zhang7.6}
 \kappa(X_0) J_G^G(\eta, X_0, \hat{f})=\vol(\omega) \gamma_\psi(X_0, \mu U_0)\neq 0. 
\end{equation}
In fact, the above computation to deduce \eqref{eq:zhang7.6} is just rephrasing that in the last paragraph of the proof of \cite[Proposition 7.6]{MR3414387}. Similarly, we prove that 
$$ J_{H'}^{H'}(Y_0,\hat{f'})=\vol(\omega') \gamma_\psi(Y_0, \mu V_0)\neq 0 $$
by \cite[(7.2.2)]{MR4681295}, the conditions (i) and (ii) on $f'$, the conditions (ii) and (v) on $\omega'$, the condition (i) on $r$ and the condition (iii) on $\mu$. From the condition $\varphi(\omega')=\omega$ and our choice of Haar measures in Section \ref{ssec:centrss}, we know that $\vol(\omega)=\vol(\omega')$. By Lemma \ref{lem75}.(2), we have 
$$ \frac{\gamma_\psi(X_0, \mu U_0)}{\gamma_\psi(Y_0, \mu V_0)} = \frac{\gamma_\psi(\fh(F))}{\gamma_\psi(\fh'(F))}. $$
Then the equality in the condition (f) follows. 
\end{proof}


\section{\textbf{The unramified case}}\label{secwfle}

\subsection{}\label{ssec:unram_not}

Let $F$ be a non-archimedean local field of characteristic zero and $E$ be a quadratic extension of $F$. Assume that $F$ has odd residue characteristic and that $E/F$ is unramified. Assume that $(G',H')\simeq(GL_{2n},\Res_{E/F}GL_{n,E})$. In this case, the injection $M'\mapsto M$ from $\msl^{H'}(M'_0)$ into $\msl^{G,\omega}(M_0)$ in Section \ref{ssec:mat_Levi} is a bijection. For $M\in\msl^{G,\omega}(M_0)$, we shall always denote by $M'$ the unique element in $\msl^{H'}(M'_0)$ matching $M$. Notice that the $\BR$-linear spaces $\fa_M$ and $\fa_{M'}$ are naturally isomorphic. We shall choose compatible scalar products on them. There is a natural $\CO_F$-scheme structure on $\fs$ such that 
$$ \fs(\CO_F)=\left\{\mat(0,A,B,0): A,B\in\fg\fl_n(\CO_F)\right\}. $$
Let $\sigma$ be the nontrivial element in $\Gal(E/F)$. From the discussion in \cite[\S 10.2]{MR4350885}, we can and we shall identify $\fs'$ with $\fh'$ equipped with the Galois twisted conjugation of $H'$, i.e., $x\cdot Y:=x Y\sigma(x)^{-1}$ for all $x\in H'$ and $Y\in\fs'$. Moreover, there is a natural $\CO_F$-scheme structure on $\fs'$ such that 
$$ \fs'(\CO_F)=\fg\fl_n(\CO_E). $$
Denote by $f_0\in\CC_c^\infty(\fs(F))$ (resp. $f'_0\in\CC_c^\infty(\fs'(F))$) the characteristic function of $\fs(\CO_F)$ (resp. of $\fs'(\CO_F)$). 

\subsection{The weighted fundamental lemma}

\begin{lem}[see {\cite[Theorem 10.9]{MR4350885}}]\label{wfl}
For all $M\in\msl^{G,\omega}(M_0)$ and all $Q\in\msf^G(M)$, we have
\begin{enumerate}[(a)]
	\item for $X\in(\fm\cap\fs_\rs)(F)$ and $Y\in(\wt{\fm'}\cap\fs'_\rs)(F)$ such that $X\overset{M}{\da} Y$ (see Definition \ref{defbyinvM}), 
	$$ \kappa(X) J_M^Q(\eta, X, f_0)=J_{M'}^{Q'}(Y, f'_0); $$

	\item for $X\in(\fm\cap\fs_\rs)(F)$, 
	$$ J_M^Q(\eta, X, f_0)=0 $$
	unless $X$ comes potentially from $(\fm_{\wt{Q'}}\cap\fs'_\rs)(F)$ (see Definition \ref{defcomepotfrom}). 
\end{enumerate}
\end{lem}

\begin{remark}
Under the assumption $(G',H')\simeq(GL_{2n},\Res_{E/F}GL_{n,E})$, for all $M'\in\msl^{H'}(M'_0)$, we have 
$$ e_{M'}=(1, \cdots, 1)\in\{\pm1\}^\ell $$
with the notations in Section \ref{ssec:mat_Levi}. 
\end{remark}

\subsection{}

By \cite[Proposition 4.2]{MR3414387} and the base change to $\ov{F}$,  we know that any element $X\in\fs_\rs(F)$ (resp. $Y\in\fs'_\rs(F)$) is regular semi-simple in $\fg$ (reps. $\fg'$) in the usual sense, so $G_X$ (resp. $G'_Y$) is a maximal torus in $G$ (resp. in $G'$) defined over $F$. As usual, we call a torus over $F$ unramified if it splits over an unramified extension of $F$. 

Let $M\in\msl^{G,\omega}(M_0)$ and $Q\in\msf^G(M)$. For $f\in \CC_c^\infty(\fs(F))$ and $X\in (\fm\cap\fs_\rs)(F)$, define the weighted orbital integral 
$$ J_M^Q(X, f):=|D^\fs(X)|_F^{1/2} \int_{H_X(F)\bs H(F)} f(\Ad(x^{-1})(X)) v_M^Q(x) dx. $$
This is an analogue of \eqref{defnoninvwoi1} and \eqref{defnoninvwoi2} with $E$ replaced by $F\times F$. 

\begin{lem}\label{lem:chau71}
\leavevmode
\begin{enumerate}
\item Let $M\in\msl^{G,\omega}(M_0)$ and $X\in(\fm\cap\fs_\rs)(F)$. Suppose that 
	\begin{enumerate}[(a)]
		\item $X\in\fs(\CO_F)$ and $\det(X)\in\CO_F^\times$; 
	
		\item $T:=G_X$ is an unramified torus; 
	
		\item $X\in\ft(F)$ is integral of regular reduction in the sense of \cite[\S 7.2]{MR1440722}. 
	\end{enumerate}
Then for all $Q\in\msf^{G}(M)$, we have 
	\begin{displaymath}
	\kappa(X) J_{M}^{Q}(\eta, X, f_0) = J_{M}^{Q}(X, f_0) = \left\{ \begin{array}{ll}
	0, & \text{if $Q\notin\msp^{G}(M)$}; \\
	1, & \text{otherwise}. \\
	\end{array} \right.
	\end{displaymath}

\item Let $M'\in\msl^{H'}(M'_0)$ and $Y\in(\wt{\fm'}\cap\fs'_\rs)(F)$. Suppose that 
	\begin{enumerate}[(a)]
		\item $Y\in\fs'(\CO_F)$ and $\det(Y)\in\CO_E^\times$; 
	
		\item $T':=G'_Y$ is an unramified torus;  
	
		\item $Y\in\ft'(F)$ is integral of regular reduction in the sense of \cite[\S 7.2]{MR1440722}. 
	\end{enumerate}
Then for all $Q'\in\msf^{H'}(M')$, we have 
	\begin{displaymath}
	J_{M'}^{Q'}(Y, f'_0) = \left\{ \begin{array}{ll}
	0, & \text{if $Q'\notin\msp^{H'}(M')$}; \\
	1, & \text{otherwise}. \\
	\end{array} \right.
	\end{displaymath}
\end{enumerate}
\end{lem}

\begin{remark}\label{rmk:chau71}
If $F$ is a number field, then for a fixed $X$ or $Y$, the conditions (a)-(c) are satisfied at all but finitely many places. 
\end{remark}

\begin{proof}[Proof of Lemma \ref{lem:chau71}]
This is an analogue of \cite[Lemme 7.1]{MR2164623}. We shall first prove (1). By \cite[(3.1.1)]{MR4681295} and the condition (c), we have 
$$ |D^\fs(X)|_F=|D^\fg(X)|_F^{1/2}=1, $$
where $|D^\fg(X)|:=|\det(\ad(X)|_{\fg/\fg_X})|_F$ denotes the usual Weyl discriminant for $\fg(F)$. For $x\in H(F)$, by \cite[Lemme 7.2]{MR1440722}, $\Ad(x^{-1})(X)\in\fs(\CO_F)$ if and only if $x=tk$ with $t\in T(F)$ and $k\in K_G$. In this case, if $Q\notin\msp^{G}(M)$, then $v_M^Q(x)=v_M^Q(tk)=0$. Hence, $\kappa(X) J_{M}^{Q}(\eta, X, f_0)=J_{M}^{Q}(X, f_0)=0$ for $Q\notin\msp^{G}(M)$. Now suppose $Q\in\msp^G(M)$ and thus $v_M^Q=1$. By the condition (a), with conjugation by $H(\CO_F)$ if necessary, we may and we shall assume that $X$ is of the form $\mat(0,1_n,A,0)$. From the proof of \cite[Lemma 5.18]{MR3414387}, we know that 
\begin{equation}\label{eq:kot-cal-1}
 \kappa(X) J_{M}^{Q}(\eta, X, f_0)=\int_{GL_{n,A}(F)\bs GL_n(F)} \Psi_A(\Ad(g^{-1})(A)) dg, 
\end{equation}
where 
$$ \Psi_A(g):=\int_{GL_n(F)} \Phi_A(h, h^{-1}g)\eta(\det(h)) dh\in\CC_c^\infty(GL_n(F))$$
with $\Phi_A\in \CC_c^\infty(GL_n(F)\times GL_n(F))$ being the characteristic function of the set 
$$ \{(r,t)\in (GL_n(F)\cap\fg\fl_n(\CO_F))\times (GL_n(F)\cap\fg\fl_n(\CO_F)): |\det(rt)|_F=|\det(A)|_F\}. $$ 
Here we have $\det(A)\in \CO_F^\times$ by the condition (a), so $\Phi_A$ is the characteristic function of $GL_n(\CO_F)\times GL_n(\CO_F)$. Since $E/F$ is unramified, $\eta(\det(h))=1$ for $h\in GL_n(\CO_F)$. As $\vol(GL_n(\CO_F))=1$, we deduce that $\Psi_A$ is the characteristic function of $GL_n(\CO_F)$. Then the orbital integral \eqref{eq:kot-cal-1} is computed in \cite[Corollary 7.3]{MR858284} and equals $1$. The proof of $J_{M}^{Q}(X, f_0)=1$ for $Q\in\msp^G(M)$ is similar: it suffices to replace $\eta$ in the above argument by the trivial character of $F^\times$. 

Now we show (2). The vanishing of $J_{M'}^{Q'}(Y, f'_0)$ for $Q'\notin\msp^{H'}(M')$ can be shown in the same way as above. Let $Q'\in\msp^{H'}(M')$ and thus $v_{M'}^{Q'}=1$. From the proof of \cite[Lemma 5.18]{MR3414387}, we know that 
\begin{equation}\label{eq:kot-cal-2}
 J_{M'}^{Q'}(Y, f'_0)=\int_{GL_n(E)_{Y,\sigma}\bs GL_n(E)} \Xi_Y(h^{-1} Y \sigma(h)) dh, 
\end{equation}
where the $\sigma$-twisted centraliser of $Y$ in $GL_n(E)$ is denoted by 
$$ GL_n(E)_{Y,\sigma}:=\{h\in GL_n(E): h^{-1}Y\sigma(h)=Y\} $$ 
and we let $\Xi_Y\in\CC_c^\infty(GL_n(E))$ be the characteristic function of the set 
$$ \{r\in GL_n(E)\cap\fg\fl_n(\CO_E): |\det(r)|_E=|\det(Y)|_E\}. $$
Here we have $\det(Y)\in \CO_E^\times$ by the condition (a), so $\Xi_Y$ is the characteristic function of $GL_n(\CO_E)$. If $h^{-1}Y\sigma(h)\in GL_n(\CO_E)$, then $h^{-1}Y\sigma(Y)h\in GL_n(\CO_E)$. From \cite[Proposition 7.1]{MR858284},  we deduce 
$$ h\in GL_n(E)_{Y\sigma(Y)} \cdot GL_n(\CO_E), $$
where $GL_n(E)_{Y\sigma(Y)}$ is the centraliser of $Y\sigma(Y)$ in $GL_n(E)$. Since $Y\sigma(Y)$ is regular semi-simple in $GL_n(E)$, the proof of \cite[Lemma 1.1.(ii)]{MR1007299} implies that 
$$ GL_n(E)_{Y,\sigma}=GL_n(E)_{Y\sigma(Y)}. $$
Then the $\sigma$-twisted orbital integral \eqref{eq:kot-cal-2} is equal to 
$$\vol(GL_n(\CO_E))\cdot\vol(GL_n(E)_{Y\sigma(Y)}\cap GL_n(\CO_E))=1. $$
\end{proof}


\section{\textbf{Approximation of local data by global data}}\label{secapp}

Let $k'/k$ be a quadratic extension of number fields and $\BD$ be a central division algebra over $k$. For a place $v$ of $k$, denote $k'_v:=k'\otimes_k k_v$ and $\BD_v:=\BD\otimes_k k_v$. 

\begin{prop}\label{appglo}
Let $E/F$ be a quadratic extension of non-archimedean local fields of characteristic $0$. 

\begin{enumerate}[I)]
	\item Let $D$ be a central division algebra over $F$ containing $E$. Then there exists a quadratic extension of number fields $k'/k$, a central division algebra $\BD$ over $k$ containing $k'$, and a finite set $S$ of finite places of $k$ satisfying the following conditions. 
	\begin{enumerate}[(a)]
		\item The number field $k$ is totally imaginary.  
		\item $|S|\geq2$. 
		\item For all $v\in S$, we have compatible isomorphisms $k_v\simeq F$, $k'_v\simeq E$ and $\BD_v\simeq D$. 
		\item For all $v\notin S$, $k_v$ splits $\BD_v$. 
	\end{enumerate}

	\item Let $D$ be a central division algebra over $F$ such that $D\otimes_F E$ is a central division algebra over $E$. Then there exists a quadratic extension of number fields $k'/k$, a central division algebra $\BD$ over $k$ such that $\BD\otimes_k k'$ is a central division algebra over $k'$, and a finite set $S$ of finite places of $k$ satisfying the following conditions. 
	\begin{enumerate}[(a)]
		\item The number field $k$ is totally imaginary. 
		\item $|S|\geq2$. 
		\item For all $v\in S$, we have compatible isomorphisms $k_v\simeq F$, $k'_v\simeq E$ and $\BD_v\simeq D$. 
		\item For all $v\notin S$, $k_v$ splits $\BD_v$. 
	\end{enumerate}
\end{enumerate}
\end{prop}

\begin{proof}
By \cite[Proposition 9.1]{MR2164623}, there exists a number field $k$, a central simple algebra $\BD$ over $k$ and a finite set $S$ of finite places of $k$ such that 
\begin{enumerate}[(i)]
	\item the number field $k$ is totally imaginary; 
	\item $|S|\geq2$; 
	\item for all $v\in S$, we have compatible isomorphisms $k_v\simeq F$ and $\BD_v\simeq D$; 
	\item for all $v\notin S$, $k_v$ splits $\BD_v$. 
\end{enumerate}
From the condition (iii), we know that $\BD$ is a central division algebra over $k$. By \cite[Theorem 3.1]{Conrad}, there exists a quadratic extension $k'$ of $k$ such that $k'_v\simeq E$ for all $v\in S$. 

\begin{enumerate}[I)]
	\item We shall use \cite[Theorem 1.2]{MR3217641} to show that there exists an $k$-embedding of $k'$ into $\BD$. It is clear that there exists an $k_v$-embedding of $k'_v$ into $\BD_v$ for all place $v$ of $k$. Let $v'$ be a place of $k'$ and $v$ be the place of $k$ below $v'$. Denote by $c_v$ (resp. $d_v$) the capacity (resp. index) of $\BD_v$ (see \cite[Definition 2.1]{MR3217641}). Since $D$ is a central division algebra over $F$ of even degree, we deduce that $\BD$ is a central division algebra over $k$ of even degree. Then $c_vd_v$ is even. Define $x_{v'}:=\frac{c_v\cdot\gcd([k'_{v'}:k_v], d_v)}{2}\in\BQ_{>0}$. If $v\notin S$, i.e., $k_v$ splits $\BD_v$, then $d_v=1$ and $c_v$ is even, so $x_{v'}\in\BZ$. If $v\in S$, then $[k'_{v'}:k_v]=[E:F]=2$, which implies that $x_{v'}\in\BZ$ since $c_vd_v$ is even. We may use \cite[Theorem 1.2]{MR3217641} to conclude. 

	\item Let $v\in S$ and $v'$ be the unique place of $k'$ over $v$. Since $(\BD\otimes_k k')_{v'}\simeq \BD\otimes_k k_v\otimes_{k_v} k'_{v'}\simeq D\otimes_F E$ is a central division algebra over $E$, we know that $\BD\otimes_k k'$ is a central division algebra over $k'$. 
\end{enumerate}
\end{proof}


\section{\textbf{An infinitesimal variant of Guo-Jacquet trace formulae}}\label{secinfiGJ}

\subsection{}

Let $k'/k$ be a quadratic extension of number fields. Denote by $V=V_k$ (resp. $V_\infty$, $V_f$) the set of places (resp. archimedean places, finite places) of $k$. Let $\BA=\BA_k$ the ring of adèles of $k$ and $|\cdot|_\BA$ be the product of normalised local absolute values on the group of idèles $\BA^\times$. Let $\eta$ be the quadratic character of $\BA^\times/k^\times$ attached to $k'/k$. If $\bfX$ is an algebraic variety defined over $k$ and $v\in V$, denote by $\bfX_v$ the variety over $k_v$ obtained by extension of scalars. 

\subsection{}\label{ssec:glo_not}

Let $\bfG$ be a reductive group over $k$. Define a homomorphism $H_\bfG: \bfG(\BA)\ra\fa_\bfG$ by 
$$ \langle H_\bfG(x),\chi\rangle=\log(|\chi(x)|_\BA) $$
for all $x\in \bfG(\BA)$ and $\chi\in X(\bfG)_k$. Write $\bfG(\BA)^1$ for the kernel of $H_\bfG$ and $A_\bfG^\infty$ for the neutral component of the group of $\BR$-points of the maximal $\BQ$-split torus in $\Res_{k/\BQ} A_\bfG$. Fix a Levi $k$-factor $\bfM_0$ of a minimal parabolic $k$-subgroup of $\bfG$. Fix a maximal compact subgroup $K$ of $\bfG(\BA)$ which is admissible relative to $\bfM_0$ in the sense of \cite[p. 9]{MR625344}. In this paper, we choose the standard maximal compact subgroup when $\bfG(k)=\bfG\bfL_n(\BD)$, where $\BD$ is a central division algebra over a finite field extension of $k$. That is to say, $K:=\prod\limits_{v\in V} K_v$ where each $K_v$ is the standard maximal compact subgroup of $\bfG(k_v)$ described in Section \ref{ssec:stdmaxcpt}. For $\bfP\in\msf^\bfG(\bfM_0)$, we may extend the function $H_{\bfM_\bfP}$ to a map $H_\bfP: \bfG(\BA)\ra\fa_{\bfM_\bfP}$ using the decomposition $\bfG(\BA)=\bfM_\bfP(\BA)\bfN_\bfP(\BA)K$.  

Recall that we have fixed a Haar measure on $\fa_\bfG$ in Section \ref{ssec:gp_gen}. By the restriction of $H_\bfG$, the group $A_\bfG^\infty$ is isomorphic to $\fa_\bfG$ and is equipped with the Haar measure transported from $\fa_\bfG$. We shall choose compatible Haar measures on $\bfG(\BA)$ and $\bfG(\BA)^1$ such that the decomposition $\bfG(\BA)\simeq\bfG(\BA)^1\times A_\bfG^\infty$ preserves these measures. Fix the Haar measure on $K$ such that $\vol(K)=1$. For every unipotent subgroup $\bfN$ of $\bfG$, choose the unique Haar measure on $\bfN(\BA)$ such that $\vol(\bfN(F)\bs\bfN(\BA))=1$. We shall choose compatible Haar measures on $\bfG(\BA)$ and its Levi subgroups by requiring that for all $\bfP\in\msf^\bfG(\bfM_0)$ and all $f\in L^1(\bfG(\BA))$, we have 
$$ \int_{\bfG(\BA)} f(x) dx = \int_{\bfM_\bfP(\BA)\times \bfN_\bfP(\BA)\times K} f(mnk) dkdndm. $$
Moreover, we can and we shall require our Haar measure on $\bfG(\BA)$ to be the product of Haar measures on $\bfG(k_v)$. 

For $\bfM\in\msl^\bfG(\bfM_0)$ and $x\in\bfG(\BA)$, we define the $(\bfG,\bfM)$-family 
$$ v_\bfP(\lambda,x):=e^{-\lambda(H_\bfP(x))}, \forall \lambda\in i\fa_\bfM^\ast, \bfP\in\msp^\bfG(\bfM). $$
For $\bfQ\in\msf^\bfG(\bfM)$, we denote by $v_\bfM^\bfQ$ the associated Arthur's weight function on $\bfG(\BA)$ defined as in the local case \eqref{eqweifun}. 

Let $\bfL\in\msl^\bfG(\bfM_0)$. For all $(R_v)_{v\in V}\in\prod\limits_{v\in V}\msl^{\bfG_v}(\bfL_v)$, we shall define Arthur's coefficient $d_\bfL^\bfG((R_v)_{v\in V})\in\BR_{\geq0}$ as in \cite[p. 356]{MR928262} (see also \cite[\S 8.5]{MR2164623}). Note that the space $\fa_\bfL$ is naturally an $\BR$-linear subspace of $\fa_{\bfL_v}$. Consider the natural map 
$$ \bigoplus_{v\in V}\fa_{\bfL_v}^{R_v}\ra\fa_{\bfL}^{\bfG} $$
obtained by the orthogonal projections $\fa_{\bfL_v}^{R_v}\ra\fa_{\bfL}^{\bfG}$ (both viewed as subspaces of $\fa_{\bfL_v}$). If the map is not an isomorphism, we define $d_\bfL^\bfG((R_v)_{v\in V}):=0$. If the map is an isomorphism, then $R_v=\bfL_v$ for all but finitely many $v\in V$, and we define $d_\bfL^\bfG((R_v)_{v\in V})$ as the volume in $\fa_{\bfL}^{\bfG}$ of the parallelogram generated by the images of orthonormal bases of $\fa_{\bfL_v}^{R_v}$. 

If $d_\bfL^\bfG((R_v)_{v\in V})\neq0$, we can choose a section 
$$ (R_v)_{v\in V}\mapsto (Q_{R_v})_{v\in V}\in \prod_{v\in V} \msp^{\bfG_v}(R_v) $$
such that the splitting formula \cite[Corollary 7.4]{MR928262} of $(\bfG,\bfL)$-families holds. 

Let $\bfT$ be a torus over $k$. Recall that if $v\in V_f$, we choose in Section \ref{ssec:stdmaxcpt} a Haar measure on $\bfT(k_v)$ such that its maximal compact subgroup is of volume $1$. If $v\in V_\infty$, we identify $\bfT(k_v)$ with $\Hom(X^\ast(\bfT_v)_{k_v},\BC^\times)$ and equip it with the Haar measure deduced from a fixed Haar measure on $\BC^\times$. We shall choose the product measure on $\bfT(\BA)$. 

\subsection{}

Let $(\bfG,\bfH,\theta)$ be a symmetric pair defined over $k$. Recall that we write $\fs$ for the corresponding infinitesimal symmetric spaces. Denote by $\CS(\fs(\BA))$ the Bruhat-Schwartz space of $\fs(\BA)$. Let $\langle \cdot, \cdot \rangle$ be a $\bfG$-invariant $\theta$-invariant non-degenerate symmetric bilinear form on $\fg$. Fix a continuous and nontrivial unitary character $\Psi: \BA/k\ra\BC^\times$. For any place $v$ of $k$, we deduce a $\bfG_v$-invariant $\theta$-invariant non-degenerate symmetric bilinear form $\langle \cdot, \cdot \rangle_v$ on $\fg_v$, as well as a continuous and nontrivial unitary character $\Psi_v: k_v\ra\BC^\times$. Recall that $\fs(k_v)$ is equipped with the unique self-dual Haar measure with respect to $\Psi_v(\langle\cdot, \cdot\rangle_v)$. Then we equip $\fs(\BA)$ with the product measure. For $f\in\fs(\BA)$, we define its Fourier transform $\hat{f}\in\fs(\BA)$ by 
$$ \forall X\in\fs(\BA), \hat{f}(X):=\int_{\fs(\BA)} f(Y)\Psi(\langle X,Y\rangle) dY. $$

We shall write $(\bfG, \bfH)$ and $(\bfG', \bfH')$ for the symmetric pairs with respect to $k'/k$ defined in Sections \ref{symmpair1} and \ref{symmpair2} respectively, whose corresponding infinitesimal symmetric spaces are written as $\fs$ and $\fs'$ respectively. We also fix minimal Levi $k$-subgroups of these four groups as before and denote them by bold letters. Let $\langle\cdot,\cdot\rangle_\fs$ (resp. $\langle\cdot,\cdot\rangle_{\fs'}$) be the restriction to $\fs$ (resp. to $\fs'$) of the $\bfG$-invariant $\Ad(\omega_0)$-invariant (resp. $\bfG'$-invariant $\Ad(\alpha)$-invariant) non-degenerate symmetric bilinear form on $\fg$ (resp. on $\fg'$) defined by \eqref{bilform1} (resp. by \eqref{bilform2}). For any place $v$ of $k$, we denote by $\langle\cdot,\cdot\rangle_{\fs_v}$ (resp. $\langle\cdot,\cdot\rangle_{\fs'_v}$) the associated local bilinear form on $\fs_v$ (resp. to $\fs'_v$). 

\subsection{Trace formulae}\label{ssec:Jphi}

Denote by $\CO^{\fs}_\rs$ the set of $\bfH(k)$-orbits in $\fs_\rs(k)$. For $\phi\in\CS(\fs(\BA))$, define 
\begin{equation}\label{glodefJ^G(eta,f)}
 J^{\bfG}(\eta,\phi)=\sum_{\fo\in\CO^{\fs}_\rs} J_\fo^{\bfG}(\eta,\phi), 
\end{equation}
where $J_\fo^{\bfG}(\eta,\phi)$ is the constant term of \cite[(5.0.1) with $s=0$]{MR4424024}. From \cite[Theorem 4.14 and Corollary 5.9]{MR4424024}, we know that the right hand side of \eqref{glodefJ^G(eta,f)} is absolutely convergent. For $\fo\in\CO^{\fs}_\rs$, let $X\in\fo\cap(\fl\cap\fs_\rs)(k)_\el$ with $\bfL\in\msl^{\bfG,\omega}(\bfM_0)$. By \cite[Theorem 9.2]{MR4424024}, we have
$$ J_\fo^{\bfG}(\eta,\phi)=\vol(A_{\bfL}^\infty \bfH_X(k)\bs \bfH_X(\BA)) \int_{\bfH_X(\BA)\bs \bfH(\BA)} \phi(\Ad(x^{-1})(X)) \eta(\det(x)) v_\bfL^\bfG(x) dx. $$
For $\bfL\in\msl^{\bfG,\omega}(\bfM_0)$, we denote
$$ \tau(\bfH_X):=\vol(A_\bfL^\infty \bfH_X(k)\bs \bfH_X(\BA)) $$
for $X\in(\fl\cap\fs_\rs)(k)_\el$, and define the global weighted orbital integral 
\begin{equation}\label{eq:glo_woi_1}
 J_\bfL^\bfG(\eta, X, \phi):=\int_{\bfH_X(\BA)\bs \bfH(\BA)} \phi(\Ad(x^{-1})(X)) \eta(\det(x)) v_\bfL^\bfG(x) dx 
\end{equation}
for all $X\in(\fl\cap\fs_\rs)(k)$. 
From \cite[Lemma 3.13]{MR4681295}, we obtain 
\begin{equation}\label{leviJ^G}
 J^\bfG(\eta,\phi)=\sum_{\bfL\in\msl^{\bfG,\omega}(\bfM_0)} |W_0^{\bfL_n}||W_0^{\bfG\bfL_n}|^{-1} \sum_{X\in\Gamma_\el((\fl\cap\fs_\rs)(k))} \tau(\bfH_X) J_\bfL^\bfG(\eta, X, \phi). 
\end{equation}

Denote by $\CO^{\fs'}_\rs$ the set of $\bfH'(k)$-orbits in $\fs'_\rs(k)$. For $\phi'\in\CS(\fs'(\BA))$, define 
\begin{equation}\label{glodefJ^H'(f')}
 J^{\bfH'}(\phi')=\sum_{\fo\in\CO^{\fs'}_\rs} J_\fo^{\bfH'}(\phi'), 
\end{equation}
where $J_\fo^{\bfH'}(\phi')$ is the constant term of \cite[(5.0.1)]{MR4350885}. From \cite[Theorem 4.2 and Corollary 5.3]{MR4350885}, we know that the right hand side of \eqref{glodefJ^H'(f')} is absolutely convergent. For $\fo\in\CO^{\fs'}_\rs$, let $Y\in\fo\cap(\wt{\fl'}\cap\fs'_\rs)(k)_\el$ with $\bfL'\in\msl^{\bfH'}(\bfM'_0)$. By \cite[Theorem 9.2]{MR4350885}, we have
$$ J_\fo^{\bfH'}(\phi')=\vol(A_{\bfL'}^\infty \bfH'_Y(k)\bs \bfH'_Y(\BA)) \int_{\bfH'_Y(\BA)\bs \bfH'(\BA)} \phi'(\Ad(x^{-1})(Y)) v_{\bfL'}^{\bfH'}(x) dx. $$
For $\bfL'\in\msl^{\bfH'}(\bfM'_0)$, we denote
$$ \tau(\bfH'_Y):=\vol(A_{\bfL'}^\infty \bfH'_Y(k)\bs \bfH'_Y(\BA)) $$
for $Y\in(\wt{\fl'}\cap\fs'_\rs)(k)_\el$, and define the global weighted orbital integral 
\begin{equation}\label{eq:glo_woi_2}
 J_{\bfL'}^{\bfH'}(Y, \phi'):=\int_{\bfH'_Y(\BA)\bs \bfH'(\BA)} \phi'(\Ad(x^{-1})(Y)) v_{\bfL'}^{\bfH'}(x) dx 
\end{equation}
for all $Y\in(\wt{\fl'}\cap\fs'_\rs)(k)$. 
From \cite[Lemma 3.20]{MR4681295}, we obtain 
\begin{equation}\label{leviJ^H'}
 J^{\bfH'}(\phi')=\sum_{\bfL'\in\msl^{\bfH'}(\bfM'_0)} |W_0^{\bfL'}||W_0^{\bfH'}|^{-1} \sum_{Y\in\Gamma_\el((\wt{\fl'}\cap\fs'_\rs)(k))} \tau(\bfH'_Y) J_{\bfL'}^{\bfH'}(Y, \phi'). \end{equation}

\begin{prop}[see {\cite[Theorem 7.1]{MR4424024}} and {\cite[Theorem 7.1]{MR4350885}}]\label{propinfivar}
\leavevmode
\begin{enumerate}
	\item Let $\phi\in\CS(\fs(\BA))$ be such that $\Supp(\phi)\subseteq \fs_\rs(k_{v_1})$ and $\Supp(\hat{\phi})\subseteq \fs_\rs(k_{v_2})$ at some places $v_1, v_2$ of $k$. Then we have the equality
$$ J^\bfG(\eta, \phi)=J^\bfG(\eta,\hat{\phi}). $$

	\item Let $\phi'\in\CS(\fs'(\BA))$ be such that $\Supp(\phi')\subseteq \fs'_\rs(k_{v_1})$ and $\Supp(\hat{\phi'})\subseteq \fs'_\rs(k_{v_2})$ at some places $v_1, v_2$ of $k$. Then we have the equality
$$ J^{\bfH'}(\phi')=J^{\bfH'}(\hat{\phi'}). $$
\end{enumerate}
\end{prop}

\subsection{Weight functions}\label{ssec:wt_function}

We shall rewrite the weight functions appearing in \eqref{defnoninvwoi2} and \eqref{eq:glo_woi_2} to make them suitable for comparison with those in \eqref{defnoninvwoi1} and \eqref{eq:glo_woi_1}. We shall use the notations in Section \ref{ssec:sym2-I-II}. 

Let $F$ be a local field of characteristic zero or a number field. Let $E$ be a quadratic extension of $F$. In \textbf{Case II}, we can and we shall identify $G'$ with the group of invertible elements in the algebra 
$$ \left\{\mat(A,B,\ov{B},\ov{A}): A,B\in \fh'\right\}, $$
where $\ov{A}$ denotes the conjugate of $A$ under the nontrivial element of $\Gal(E/F)$. Then its subgroup $H'$ is identified with
$$ \left\{\mat(A,,,\ov{A}): A\in H'\right\}. $$

We shall choose $K_{H'}$ and $K_{G'}$ such that $K_{H'}=K_{G'}\cap H'(F)$ (resp. $K_{H'}=K_{G'}\cap H'(\BA)$) if $F$ is a local field of characteristic zero (resp. number field). We distinguish several cases. 

\begin{enumerate}
	\item $F$ is a non-archimedean local field of characteristic zero. In \textbf{Case I}, we choose $K_{H'}=GL_{r}(\CO_{D'})$ and $K_{G'}=GL_{r}(\CO_D)$. In \textbf{Case II}, we choose 
$$ K_{H'}=\left\{\mat(A,,,\ov{A}): A\in GL_{\frac{r}{2}}(\CO_{D\otimes_F E})\right\} $$ 
and $K_{G'}$ to be the group of invertible elements in the algebra 
$$ \left\{\mat(A,B,\ov{B},\ov{A}): A,B\in \fg\fl_{\frac{r}{2}}(\CO_{D\otimes_F E})\right\}. $$

	\item $F$ is an archimedean local field of characteristic zero. Then $(G'(F),H'(F))$ is isomorphic to either $(GL_r(\BH),GL_r(\BC))$ in \textbf{Case I} or $(GL_r(\BR),GL_{\frac{r}{2}}(\BC))$ in \textbf{Case II}. Here $\BH$ denotes the Hamiltonian quaternion algebra over $\BR$. In \textbf{Case I}, we choose $K_{H'}$ to be the unitary group and $K_{G'}$ to be the hyperunitary group. In \textbf{Case II}, there is a natural embedding $\fg\fl_{\frac{r}{2}}(\BC)\hookrightarrow \fg\fl_r(\BR)$ defined by 
	$$ A+Bi\mapsto\mat(A,B,-B,A), \forall A,B\in \fg\fl_{\frac{r}{2}}(\BR). $$
Notice that the transpose on $\fg\fl_r(\BR)$ induces the conjugate transpose on $\fg\fl_{\frac{r}{2}}(\BC)$. Then we choose $K_{H'}$ to be the unitary group and $K_{G'}$ to be the orthogonal group. 

	\item $F$ is a number field. We choose 
	$$ K_{H'}:=\prod_{v\in V_F} K_{H'_v} \text{ and } K_{G'}:=\prod_{v\in V_F} K_{G'_v} $$ 
	where 
	\begin{itemize}
		\item if $v\in V_F$ is inert in $E$, then $K_{H'_v}$ and $K_{G'_v}$ are described above; 
		
		\item if $v\in V_F$ splits in $E$, then 
		$$ (G'_v,H'_v)\simeq(GL_{2n,D_v},GL_{n,D_v}\times GL_{n,D_v}) $$ 
		for some integer $n$ and some central simple division algebra $D_v$ over $F_v$, and we choose $K_{H'_v}$ and $K_{G'_v}$ as the standard maximal compact subgroups in Section \ref{ssec:stdmaxcpt}. 
	\end{itemize}
\end{enumerate}

Let $M'\in\msl^{H'}(M'_0)$ and $Q'\in\msf^{H'}(M')$. Recall from Section \ref{ssec:sym2-I-II} that there is a natural $\BR$-isomorphism between $\fa_{M'}$ and $\fa_{\wt{M'}}$. We shall choose compatible scalar products on them such that 
$$ v_{M'}^{Q'}(x)=v_{\wt{M'}}^{\wt{Q'}}(x) $$
for all $x\in H'(F)$ (resp. $x\in H'(\BA)$) if $F$ is a local field of characteristic zero (resp. number field). Therefore, we can replace $v_{M'}^{Q'}$ (resp. $v_{\bfL'}^{\bfH'}$) with $v_{\wt{M'}}^{\wt{Q'}}$ (resp. $v_{\wt{\bfL'}}^{\bfG'}$) in \eqref{defnoninvwoi2} (resp. \eqref{eq:glo_woi_2}) and write $J_{\wt{M'}}^{\wt{Q'}}:=J_{M'}^{Q'}$ (resp. $J_{\wt{\bfL'}}^{\bfG'}:=J_{\bfL'}^{\bfH'}$) whenever it is more convenient. 

\subsection{Factorisation}\label{ssec:factorisation}

Let $v\in V$. If $v$ is inert in $k'$, denote by $\eta_v$ the quadratic character of $k_v^\times/N k'^\times_v$ and by $\kappa_v$ the transfer factor \eqref{eq:def_trans_fac}, both attached to $k'_v/k_v$. If $v$ splits in $k'$, let $\eta_v=1$ and $\kappa_v=1$. When there is no confusion, we shall sometimes abuse notation and write $\eta:=\eta_v$. 

Let $\phi:=\prod\limits_{v\in V} \phi_v\in\CS(\fs(\BA))$. Note that for all but finitely many $v\in V$, $\phi_v$ is the function $f_0$ defined in Section \ref{ssec:unram_not}. Let $\bfL\in\msl^{\bfG,\omega}(\bfM_0)$ and $X\in(\fl\cap\fs_\rs)(k)$. By the splitting formula of $(\bfG,\bfL)$-families applied to the weight function $v_\bfL^\bfG$ (see \cite[Corollary 7.4]{MR928262}) and our choice of Haar measures in Section \ref{ssec:glo_not}, 
\begin{equation}\label{eq:fact_woi1}
 J_\bfL^\bfG(\eta, X, \phi)=\sum_{(R_v)_{v\in V}\in\prod\limits_{v\in V}\msl^{\bfG_v}(\bfL_v)} d_\bfL^\bfG((R_v)_{v\in V}) \prod_{v\in V} \kappa_v(X) J_{\bfL_v}^{Q_{R_v}}(\eta, X, \phi_v). 
\end{equation}
It results from Lemma \ref{lem:chau71}.(1) and Remark \ref{rmk:chau71} that the sum and products in \eqref{eq:fact_woi1} are in fact finite. 

Let $\phi':=\prod\limits_{v\in V} \phi'_v\in\CS(\fs'(\BA))$. Note that for all but finitely many $v\in V$, $\phi'_v$ is the function $f'_0$ (when $v$ is inert in $k'$) or $f_0$ (when $v$ splits in $k'$) defined in Section \ref{ssec:unram_not}. Let $\bfL'\in\msl^{\bfH'}(\bfM'_0)$ and $Y\in(\wt{\fl'}\cap\fs'_\rs)(k)$. By the splitting formula of $(\bfG',\wt{\bfL'})$-families applied to the weight function $v_{\wt{\bfL'}}^{\bfG'}$ (see \cite[Corollary 7.4]{MR928262}) and our choice of Haar measures in Section \ref{ssec:glo_not}, 
\begin{equation}\label{eq:fact_woi2}
 J_{\bfL'}^{\bfH'}(Y, \phi')=\sum_{(\wt{R'_v})_{v\in V}\in\prod\limits_{v\in V}\msl^{\bfG'_v}(\wt{\bfL'_v})} d_{\wt{\bfL'}}^{\bfG'}((\wt{R'_v})_{v\in V}) \prod_{v\in V} J_{\wt{\bfL'_v}}^{Q_{\wt{R'_v}}}(Y, \phi'_v). 
\end{equation}
It results from Lemma \ref{lem:chau71} and Remark \ref{rmk:chau71} that the sum and products in \eqref{eq:fact_woi2} are in fact finite. 


\section{\textbf{Proof of Proposition \ref{parcommute}}}\label{secfinalproof}

\subsection{}\label{ssec:finalsetting}

Let $F$ be a non-archimedean local field of characteristic zero and $E$ be a quadratic extension of $F$. Fix a pair of matching Levi subgroups $M'\in\msl^{H'}(M'_0)$ and $M\in\msl^{G,\omega}(M_0)$ in the sense of Section \ref{ssec:mat_Levi}. Let $f\in\CC_c^\infty(\fs(F))$ and $f'\in\CC_c^\infty(\fs'(F))$ be partially $M$-associated (see Definition \ref{defparMass}) and satisfy the two conditions in Theorem \ref{thmcommute}. It is obvious from the definition that a pair of partially $M$-associated functions are also partially $L$-associated for all $L\in\msl^G(M)$. Therefore, to prove Proposition \ref{parcommute}, we fix $X_0\in(\fm\cap\fs_\rs)(F)$ and $Y_0\in(\wt{\fm'}\cap\fs'_\rs)(F)$ such that $X_0\overset{M}{\da} Y_0$ (see Definition \ref{defbyinvM}). Then it suffices to show that for all $Q\in\msf^G(M)$, we have the equality 
\begin{equation}\label{chaufor101}
 \gamma_{\psi}(\fh(F))^{-1} \kappa(X_0) J_M^Q(\eta, X_0, \hat{f})=\gamma_{\psi}(\fh'(F))^{-1} J_{M'}^{Q'}(Y_0, \hat{f'}). 
\end{equation}

\subsection{Global data}

Fix a quadratic extension of number fields $k'/k$, a central division algebra $\BD$ over $k$, and a finite set $S$ of finite places of $k$ satisfying the conditions in Proposition \ref{appglo}. Fix a place $w\in S$. Fix a continuous and nontrivial unitary character $\Psi: \BA/k\ra\BC^\times$ whose local component at $w$ is $\psi$. 

Define the global symmetric pairs $(\bfG, \bfH)$ and $(\bfG', \bfH')$ with respect to $k'/k$ and $\BD$ as in the local case. We shall use the notations in Section \ref{secinfiGJ} without further notice. There is a bijection $L\mapsto\bfL$ from $\msl^G(M_0)$ to $\msl^\bfG(\bfM_0)$ such that $\bfL_w\simeq L$ and we denote by $\bfM$ the image of $M$ under this bijection. Similarly, there is a bijection $L'\mapsto\bfL'$ from $\msl^{H'}(M'_0)$ to $\msl^{\bfH'}(\bfM'_0)$ such that $\bfL'_w\simeq L'$ and we denote by $\bfM'$ the image of $M'$ under this bijection. 

\subsection{Places}

It follows from  the condition (d) of Proposition \ref{appglo} that for $v\in V-S$, $(\bfG'_v, \bfH'_v)$ is identified with $(\bfG_v, \bfH_v)$ (resp. with $(\bfG\bfL_{2n, k_v},\Res_{k'_v/k_v}\bfG\bfL_{n,k'_v})$) if $v$ splits (resp. is inert) in $k'$. 

Fix a finite set $S_1\subseteq V$ and for $v\in V-S_1$ lattices $\fk_v\subseteq\fs(k_v)$ and $\fk'_v\subseteq\fs'(k_v)$ such that 
\begin{enumerate}[(i)]
	\item $V_\infty\sqcup S\subseteq S_1$; 

	\item if $v\in V-S_1$, then 
	\begin{itemize}
		\item $k_v$ has odd residue characteristic and $v$ is unramified in $k'$; 
		\item $\fk_v=\left\{\mat(0,A,B,0): A,B\in\fg\fl_n(\CO_{k_v})\right\}$; 
		\item $\fk'_v=\fk_v$ if $v$ splits in $k'$, while $\fk'_v=\fg\fl_n(\CO_{k'_v})$ if $v$ is inert in $k'$; 
		\item $\fk_v$ (resp. $\fk'_v$) is self-dual with respect to $\Psi_v(\langle \cdot, \cdot\rangle_{\fs_v})$ (resp. $\Psi_v(\langle \cdot, \cdot\rangle_{\fs'_v})$). 
	\end{itemize}
\end{enumerate}
For $v\in V-S_1$, notice that $\wh{1_{\fk_v}}=1_{\fk_v}$ (resp. $\wh{1_{\fk'_v}}=1_{\fk'_v}$), where $1_{\fk_v}$ (resp. $1_{\fk'_v}$) denotes the characteristic function of $\fk_v$ (resp. $\fk'_v$).  

\subsection{Orbits}\label{ssec:orbits}

Fix for each $v\in S_1-V_\infty$ an open compact non-empty subset $\Omega_v\subseteq(\wt{\fm'}\cap\fs'_\rs)(k_v)$ such that 
\begin{enumerate}[(i)]
	\item if $v=w$, then $Y_0\in\Omega_w\subseteq\Ad(\bfM'(k_w))((\fs'_{Y_0}\cap\fs'_\rs)(k_w))$ and, for all $Q\in\msf^G(M)$, $J_{M'}^{Q'}(\cdot, \hat{f'})$ is constant on $\Omega_w$ and $\kappa(\cdot) J_M^Q(\eta, \cdot, \hat{f})$ is constant on $\{X\in(\fm\cap\fs_\rs)(k_w): \exists Y\in\Omega_w, X\overset{M}{\da}Y\}$; 

	\item if $v\in S-\{w\}$, then $\Omega_v\subseteq(\wt{\fm'}\cap\fs'_\rs)(k_v)_\el$. 
\end{enumerate}
Notice that the condition (i) is achievable because of Lemma \ref{lem74}, \cite[Propositions 4.4.(2) and 4.1.(2)]{MR4681295} and the constancy of $J_{M'}^{Q'}(\cdot, \hat{f'})$ (resp. $\kappa(\cdot) J_M^Q(\eta, \cdot, \hat{f})$) on $\bfM'(k_w)$ (resp. $\bfM_\bfH(k_w)$)-orbits. Besides, the set $(\wt{\fm'}\cap\fs'_\rs)(k_v)_\el$ in the condition (ii) is non-empty by (1)$\Leftrightarrow$(2) in Proposition \ref{equivdefcomefrom} and is open in $(\wt{\fm'}\cap\fs'_\rs)(k_v)$ by Krasner's lemma. If $v\in S_1-V_\infty-S$, then $\Omega_v$ can be chosen arbitrarily. It is clear from the condition (i) that the restriction of $J_{M'}^{Q'}(\cdot, \hat{f'})$ to $\Omega_w$ equals $J_{M'}^{Q'}(Y_0, \hat{f'})$. Since $X_0\overset{M}{\da}Y_0$, the condition (i) also implies that the restriction of $\kappa(\cdot) J_M^Q(\eta, \cdot, \hat{f})$ to $\{X\in(\fm\cap\fs_\rs)(k_w): \exists Y\in\Omega_w, X\overset{M}{\da}Y\}$ equals $\kappa(X_0) J_M^Q(\eta, X_0, \hat{f})$. 

By the strong approximation theorem, there exists $Y^0\in(\wt{\fm'}\cap\fs')(k)$ such that 
\begin{enumerate}[(i)]
	\item for $v\in S_1-V_\infty$, $Y^0\in\Omega_v$; 

	\item for $v\in V-S_1$, $Y^0\in\wt{\fm'}(k_v)\cap\fk'_v$. 
\end{enumerate}
Combined with the condition (b) of Proposition \ref{appglo} and the condition (ii) on $\Omega_v$, the condition (i) on $Y^0$ implies that $Y^0\in(\wt{\fm'}\cap\fs'_\rs)(k)_\el$ because of Lemma \ref{lemsym2ell}. Choose an element $X^0\in(\fm\cap\fs_\rs)(k)_\el$ such that $X^0\overset{\bfM}{\da}Y^0$. 

\subsection{Choice of functions}\label{choffun}

Recall that if $v\in V$ splits in $k'$, we shall understand $\eta$ as the trivial character of $k_v^\times$, and then $\kappa$ is also trivial in this case. 

For each $v\in V$, we fix functions $\phi_v\in\CS(\fs(k_v))$ and $\phi'_v\in\CS(\fs'(k_v))$ as follows. 
\begin{enumerate}[(i)]
	\item If $v=w$, let $\phi_w=f$ and $\phi'_w=f'$. 

	\item For $v\in S_1-V_\infty-\{w\}$, 
	\begin{itemize}
		\item if $v$ is inert in $k'$, let $\phi_v$ and $\phi'_v$ verify the conditions in Proposition \ref{testfunctions} with respect to $\bfM'_v$, $\bfM_v$, $X^0$ and $Y^0$; 
		
		\item if $v$ splits in $k'$, we identify $(\bfH_v, \fs_v)$ and $(\bfH'_v, \fs'_v)$ by the condition (d) in Proposition \ref{appglo}, and choose $\phi_v=\phi'_v$ verifying an analogue of Proposition \ref{testfunctions}. 
	\end{itemize} 

	\item If $v\in V-S_1$, let $\phi_v=1_{\fk_v}$ and $\phi'_v=1_{\fk'_v}$. 

	\item For $v\in V_\infty$, we identify $(\bfH_v, \fs_v)$ and $(\bfH'_v, \fs'_v)$ by the condition (a) in Proposition \ref{appglo}, and choose $\phi_v=\phi'_v$ such that 
	\begin{itemize}
		\item $J_{\bfG_v}^{\bfG_v}(\eta, X^0, \wh{\phi_v})=J_{\bfH'_v}^{\bfH'_v}(Y^0, \wh{\phi'_v})\neq 0$; 

		\item if $X\in\fs(k)$ is $\bfH(k_u)$-conjugate to an element in $\Supp(\wh{\phi_u})$ for all $u\in V$, then $X$ is $\bfH(k)$-conjugate to $X^0$; 

		\item if $Y\in\fs'(k)$ is $\bfH'(k_u)$-conjugate to an element in $\Supp(\wh{\phi'_u})$ for all $u\in V$, then $Y$ is $\bfH'(k)$-conjugate to $Y^0$. 
	\end{itemize}
\end{enumerate}
Notice that the condition (iv) is possible: with the preparation \cite[Proposition 3.3]{MR4424024} and \cite[Proposition 3.5]{MR4350885}, the proof of existence of such functions is sketched in \cite[p. 1874]{MR3414387} and similar to the argument of \cite[Lemme 10.7]{MR1440722}. 

In the rest of the proof of Proposition \ref{parcommute}, we shall consider the global test functions 
$$ \phi:=\prod\limits_{v\in V} \phi_v\in\CS(\fs(\BA)) \text{ and } \phi':=\prod\limits_{v\in V} \phi'_v\in\CS(\fs'(\BA)). $$

\subsection{Comparison of $J^\bfG(\eta, \phi)$ and $J^{\bfH'}(\phi')$}\label{ssec:Compare_J_J'}

\subsubsection{}

Recall the geometric expansions \eqref{leviJ^G} of $J^\bfG(\eta, \phi)$ and \eqref{leviJ^H'} of $J^{\bfH'}(\phi')$. As in \cite[Lemme 10.1]{MR2164623}, we simplify them in the following lemma for our choice of test functions in Section \ref{choffun}. 

\begin{lem}\label{lem101}
For our choice of $\phi$ and $\phi'$, we have 
$$ J^\bfG(\eta, \phi)=\sum_{\bfL\in\msl^\bfG(\bfM)} |\Tran_{\bfG\bfL_n}(\bfM_n, \bfL_n)|^{-1} \sum_{X\in\Gamma_\el((\fl\cap\fs_\rs)(k))} \tau(\bfH_X) J_\bfL^\bfG(\eta, X, \phi) $$
and 
$$ J^{\bfH'}(\phi')=\sum_{\bfL'\in\msl^{\bfH'}(\bfM')} |\Tran_{\bfH'}(\bfM', \bfL')|^{-1} \sum_{Y\in\Gamma_\el((\wt{\fl'}\cap\fs'_\rs)(k))} \tau(\bfH'_Y) J_{\bfL'}^{\bfH'}(Y, \phi'). $$
\end{lem}

\begin{proof}
For the first formula, we start with \eqref{leviJ^G}. Fix $\bfL\in\msl^{\bfG,\omega}(\bfM_0)$ and $X\in\Gamma_\el((\fl\cap\fs_\rs)(k))$ such that $J_\bfL^\bfG(\eta, X, \phi)\neq0$. By the condition (b) in Proposition \ref{appglo}, there exists $u\in S-\{w\}$. Then $X$ is $\bfH(k_u)$-conjugate to an element $X^\ast\in\Supp(\phi_u)$. By the condition (a) in Proposition \ref{testfunctions} of $\phi_u$, there exists $Y^\ast\in(\fs'_{Y^0}\cap\fs'_\rs)(k_u)$ such that $X^\ast\da Y^\ast$. Since $X^0\da Y^0$, by Lemma \ref{lem74}, there exists $X^\dagger\in(\fs_{X^0}\cap\fs_\rs)(k_u)$ such that $X^\dagger\da Y^\ast$. It follows that $X^\ast$ and $X^\dagger$ are $\bfH(k_u)$-conjugate. Thus we have proved that $X$ is $\bfH(k_u)$-conjugate to an element $X^\dagger\in(\fs_{X^0}\cap\fs_\rs)(k_u)$. 

Choose $\bfR\in\msl^{\bfG,\omega}(\bfM_0)$ such that $\bfR\subseteq \bfL$ and that $X$ is $\bfL_\bfH(k_u)$-conjugate to an element $X^\#\in(\fr\cap\fs_\rs)(k_u)_\el$. Then $X^\dagger$ and $X^\#$ are $\bfH(k_u)$-conjugate. It follows that there exists $x\in\bfH(k_u)$ and $Z\in(\fr\cap\fs_\rs)(k_u)_\el$ such that $Z=\Ad(x)(X^0)$. But by the condition (ii) on $\Omega_u$, the condition (i) on $Y^0$ and $X^0\overset{\bfM}{\da} Y^0$ (see Section \ref{ssec:orbits}), we know that $X^0\in(\fm\cap\fs_\rs)(k_u)_\el$. Thus by \cite[Lemma 3.13.(1)]{MR4681295}, there exists $w\in\left\{\mat(\omega_n,,,\omega_n): \omega_n\in W_0^{\bfG\bfL_{n,u}}\right\}$ such that $\Ad(w)(\bfM_u)=\bfR_u$. Since we may identify $W_0^{\bfG\bfL_{n,u}}$ with $W_0^{\bfG\bfL_n}$, we deduce that $\Tran_{\bfG\bfL_n}(\bfM_n, \bfL_n)\neq\emptyset$. That is to say, we may restrict the sum on $\msl^{\bfG,\omega}(\bfM_0)$ in \eqref{leviJ^G} to those elements conjugate to an element in $\msl^\bfG(\bfM)$ under $\left\{\mat(\omega_n,,,\omega_n): \omega_n\in W_0^{\bfG\bfL_n}\right\}$. 

We shall replace the sum on $\msl^{\bfG,\omega}(\bfM_0)$ in \eqref{leviJ^G} with a sum on $\msl^\bfG(\bfM)$. On the one hand, the number of elements in $\msl^{\bfG,\omega}(\bfM_0)$ conjugate to $\bfL$ under $\left\{\mat(\omega_n,,,\omega_n): \omega_n\in W_0^{\bfG\bfL_n}\right\}$ is 
$$ |\Norm_{W_0^{\bfG\bfL_n}}(\bfL_n)\bs W_0^{\bfG\bfL_n}|. $$ 
On the other hand, the number of elements in $\msl^\bfG(\bfM)$ conjugate to $\bfL$ under $\left\{\mat(\omega_n,,,\omega_n): \omega_n\in W_0^{\bfG\bfL_n}\right\}$ is $$ |W^{\bfG\bfL_n}(\bfL_n)|^{-1}|\Tran_{\bfG\bfL_n}(\bfM_n,\bfL_n)|. $$ 
Since $W^{\bfG\bfL_n}(\bfL_n)=\Norm_{W_0^{\bfG\bfL_n}}(\bfL_n)/W_0^{\bfL_n}$ (see \cite[(7.12.2)]{MR2192014}), we obtain 
$$ |W_0^{\bfL_n}||W_0^{\bfG\bfL_n}|^{-1}\cdot|\Norm_{W_0^{\bfG\bfL_n}}(\bfL_n)\bs W_0^{\bfG\bfL_n}|\cdot|W^{\bfG\bfL_n}(\bfL_n)||\Tran_{\bfG\bfL_n}(\bfM_n,\bfL_n)|^{-1}=|\Tran_{\bfG\bfL_n}(\bfM_n,\bfL_n)|^{-1}. $$
Then the first formula of the lemma follows. 

The second formula can be proved in a similar way with the help of \eqref{leviJ^H'}, the condition (b) in Proposition \ref{testfunctions} of $\phi'_u$ and \cite[Lemma 3.20.(1)]{MR4681295}. The only additional ingredient is the condition (c) in Proposition \ref{appglo}, by which we may identify $W_0^{\bfH'_u}$ with $W_0^{\bfH'}$. 
\end{proof}

\subsubsection{}\label{ssec:scalartpdt}

Recall from Section \ref{ssec:mat_Levi} that we have bijections $\msl^{\bfH'}(\bfM')\ra\msl^{\bfG}(\bfM)$ and $\msf^{\bfH'}(\bfM')\ra\msf^{\bfG}(\bfM)$. If $(\bfG'_v,\bfH'_v)$ is identified with $(\bfG_v,\bfH_v)$ for certain $v\in V$, then $\bfM'_v=\bfM_v\cap \bfH_v$ is a Levi subgroup of $\bfH'_v$. In this case, to unify notations, we define
$$ e_{\bfM'_v}:=(1,\cdots,1)\in\{\pm1\}^\ell $$
as in \eqref{eq:def_e_M'}. Moreover, for $L_v\in\msl^{\bfG_v}(\bfM_v)$, resp. $Q_v\in\msf^{\bfG_v}(\bfM_v)$, we shall write 
$$ L'_v:=L_v\cap \bfH_v\in\msl^{\bfH'_v}(\bfM'_v) \text{, resp. } Q'_v:=Q_v\cap \bfH_v\in\msf^{\bfH'_v}(\bfM'_v), $$
and 
$$ \wt{L'_v}:=L_v\in\msl^{\bfG'_v}(\wt{\bfM'_v}) \text{, resp. } \wt{Q'_v}:=Q_v\in\msf^{\bfG'_v}(\wt{\bfM'_v}). $$

Let $\bfL\in\msl^\bfG(\bfM)$ and $\bfL'\in\msl^{\bfH'}(\bfM')$ be a pair of matching Levi subgroups. Then the map 
$$ (R_v)_{v\in V}\mapsto(\wt{R'_v})_{v\in V} $$ 
defines a bijection from $\prod\limits_{v\in V}\msl^{\bfG_v}(\bfL_v)$ to $\prod\limits_{v\in V}\msl^{\bfG'_v}(\wt{\bfL'_v})$. From Section \ref{ssec:sym2-I-II}, there is a natural $\BR$-linear isomorphism $\fa_{\bfL'}\simeq\fa_{\wt{\bfL'}}$. By \cite[Proposition 2.15 in Chapter IV]{milneCFT}, the $\BR$-linear spaces $\fa_{\wt{\bfL'}}$, $\fa_{\bfL}$, $\fa_{\wt{\bfL'_v}}$ and $\fa_{\bfL_v}$ for $v\in V$ are all naturally isomorphic. By modifying the scalar products on these Euclidean spaces, we may and we shall suppose that all these $\BR$-linear isomorphisms are isometries. Consider $(R_v)_{v\in V}\in\prod\limits_{v\in V}\msl^{\bfG_v}(\bfL_v)$. We obtain the following commutative diagram induced by orthogonal projections. 
$$     \xymatrix{  \bigoplus\limits_{v\in V} \fa_{\wt{\bfL'_v}}^{\wt{R'_v}} \ar[r]^{\sim}  \ar[d] & \bigoplus\limits_{v\in V} \fa_{\bfL_v}^{R_v} \ar[d]   \\
   \fa_{\wt{\bfL'}}^{\bfG'} \ar[r]^{\sim} & \fa_\bfL^\bfG   }  $$
By our choice of scalar products, we have the equality 
\begin{equation}\label{eq:equal-d}
 d_{\wt{\bfL'}}^{\bfG'}((\wt{R'_v})_{v\in V})=d_\bfL^\bfG((R_v)_{v\in V}). 
\end{equation}
We shall also choose compatible sections 
$$ (\wt{R'_v})_{v\in V}\mapsto (Q_{\wt{R'_v}})_{v\in V}\in \prod_{v\in V} \msp^{\bfG'_v}(\wt{R'_v}) $$ 
and 
$$ (R_v)_{v\in V}\mapsto (Q_{R_v})_{v\in V}\in \prod_{v\in V} \msp^{\bfG_v}(R_v) $$
in the sense that for all $v\in V$, 
$$ Q_{\wt{R'_v}}=\wt{(Q_{R_v})'}. $$ 

\begin{prop}\label{prop102}
For our choice of $\phi$ and $\phi'$, we have the equality 
$$ J^\bfG(\eta, \phi)=J^{\bfH'}(\phi'). $$
\end{prop}

The rest of Section \ref{ssec:Compare_J_J'} is devoted to the proof of Proposition \ref{prop102}. 

\subsubsection{}

Let $\bfL\in\msl^\bfG(\bfM)$ and $\bfL'\in\msl^{\bfH'}(\bfM')$ be a pair of matching Levi subgroups. There is a natural bijection between $\Tran_{\bfG\bfL_n}(\bfM_n, \bfL_n)$ and $\Tran_{\bfH'}(\bfM', \bfL')$ since both of them are understood as permutations. Let $X\in\Gamma_\el((\fl\cap\fs_\rs)(k))$. By Proposition \ref{prop:equivcomefromLevi} and Lemma \ref{lemsym2ell}, $X$ comes from $\fs'_\rs(k)$ if and only if there exists $Y\in\Gamma_\el((\wt{\fl'}\cap\fs'_\rs)(k))$ such that $X\overset{\bfL}{\da}Y$. Moreover, it follows from Section \ref{ssec:matorb} that there is a natural injection from $\Gamma_\el((\wt{\fl'}\cap\fs'_\rs)(k))$ into $\Gamma_\el((\fl\cap\fs_\rs)(k))$. 

We first consider $X\in\Gamma_\el((\fl\cap\fs_\rs)(k))$ and $Y\in\Gamma_\el((\wt{\fl'}\cap\fs'_\rs)(k))$ such that $X\overset{\bfL}{\da}Y$. Then there is an isomorphism $\bfH_X\simeq\bfH'_Y$ over $k$ by Lemma \ref{lem74}. Thus we have 
\begin{equation}\label{eq:tau_equal}
 \tau(\bfH_X)=\tau(\bfH'_Y) 
\end{equation}
by our choices of Haar measures on tori in Section \ref{ssec:glo_not} and of scalar products on $\fa_\bfL\simeq\fa_{\bfL'}$ in Section \ref{ssec:scalartpdt}. Since $X$ and $Y$ have $\bfL_v$-matching orbits at each $v\in V$, we see that 
$$ J_\bfL^\bfG(\eta, X, \phi)=J_{\bfL'}^{\bfH'}(Y, \phi') $$ 
by \eqref{eq:fact_woi1}, \eqref{eq:fact_woi2}, \eqref{eq:equal-d} and our choice of $\phi$ and $\phi'$ in Section \ref{choffun}. In particular, here we have used our assumption on $f$ and $f'$ (see Section \ref{ssec:finalsetting}) and the weighted fundamental lemma (see Lemma \ref{wfl}). 

Hence by Lemma \ref{lem101}, to show Proposition \ref{prop102}, it suffices to show that $J_\bfL^\bfG(\eta, X, \phi)=0$ for any $X\in\Gamma_\el((\fl\cap\fs_\rs)(k))$ that does not come from $\fs'_\rs(k)$. 

\subsubsection{}\label{ssec:contradict}

We now assume that $X\in\Gamma_\el((\fl\cap\fs_\rs)(k))$ does not come from $\fs'_\rs(k)$ but $J_\bfL^\bfG(\eta, X, \phi)\neq0$. By the factorisation \eqref{eq:fact_woi1} of $J_\bfL^\bfG(\eta, X, \phi)$, there exists $(R_v)_{v\in V}\in\prod\limits_{v\in V}\msl^{\bfG_v}(\bfL_v)$ such that 
\begin{equation}\label{eq:contra_nonvan}
 d_\bfL^\bfG((R_v)_{v\in V}) \prod_{v\in V} \kappa_v(X) J_{\bfL_v}^{Q_{R_v}}(\eta, X, \phi_v)\neq0. 
\end{equation}
Fix such an $(R_v)_{v\in V}$. We can and we shall choose a representative of $X$ in the form of 
$$ \mat(0,1_n,A,0)\in (\fl\cap\fs_\rs)_\el(k), $$
which will again be denoted by $X$. 

\begin{lem}\label{lemRsigne}
For all $v\in V$, we have $\eta_{R_v}(X)=e_{R'_v}$. 
\end{lem}

\begin{proof}
It follows from our definitions in Sections \ref{ssec:factorisation} and \ref{ssec:scalartpdt} if $v\in V$ splits in $k'$ (in particular, if $v\in V_\infty$). For $v\in V-S_1$ which is inert in $k'$, it results from Lemma \ref{wfl}.(b). For $v\in S_1-V_\infty-\{w\}$ which is inert in $k'$, it is deduced from Propositions \ref{testfunctions}.(a)(c) and \ref{prop:equivcomefromLevi} and Corollary \ref{corequivdefcomefrom}.(1). For $v=w$, it is because $f$ satisfies the additional condition in Definition \ref{defMass}. 
\end{proof}

\begin{lem}\label{lemLsigne}
For all $v\in V$, we have $\eta_{\bfL_v}(X)=e_{\bfL'_v}$. 
\end{lem}

\begin{proof}
Let $S_2$ be the subset of $V_f$ consisting of all places that are inert in $k'$. The assertion is trivial for $v\in V-S_2$ by our definitions in Sections \ref{ssec:factorisation} and \ref{ssec:scalartpdt}. Consider $v\in S_2$. By Corollaries \ref{idcar3} and \ref{idcar5}, $\eta_{\bfL_v}(X)$ and $e_{\bfL'_v}$ are understood as $\varepsilon_{X}^{\bfL_v}$ and ${\varepsilon'}^{\bfL_v}$ respectively. By Proposition \ref{proppontdual}.(1), we view $\varepsilon_{X}^{\bfL_v}$ and ${\varepsilon'}^{\bfL_v}$ as complex characters of $Z_{\wh{\bfL_v}}[2]=Z_{\wh{\bfL}}[2]$. Denote the character obtained by their quotient by $\zeta_v:=\varepsilon_{X}^{\bfL_v}/{\varepsilon'}^{\bfL_v}$. To unify notations, for $v\in V-S_2$, we define $\varepsilon_{X}^{\bfL_v}={\varepsilon'}^{\bfL_v}=\zeta_v$ as the trivial character. 

For all $v\in S_2$, by Lemma \ref{lemRsigne}, we know that the character $\zeta_v$ is trivial on $Z_{\wh{R_v}}[2]$. The fundamental exact sequence of global class field theory for $k'/k$ (see \cite[Example 4.4.(a) in Chapter VIII]{milneCFT} for example) implies that the product $\prod\limits_{v\in S_2} \varepsilon_{X}^{\bfL_v}$ is trivial on $Z_{\wh{\bfL}}[2]$. Additionally, because of Lemma \ref{lem:kottsign} and the product formula of Kottwitz signs \cite[the last proposition]{MR697075}, the product $\prod\limits_{v\in S_2} {\varepsilon'}^{\bfL_v}$ is trivial on $Z_{\wh{\bfL}}[2]$. In sum, we have a global condition: the product $\prod\limits_{v\in S_2} \zeta_v$ is trivial on $Z_{\wh{\bfL}}[2]$. 

We have to show that $\zeta_v$ is trivial on $Z_{\wh{\bfL}}[2]$ for $v\in S_2$. Since $d_\bfL^\bfG((R_v)_{v\in V})\neq0$, we have $\bigoplus\limits_{v\in V}\fa_{\bfL_v}^{R_v}=\fa_{\bfL}^{\bfG}$. Let $v\in S_2$. From $\fa_{\bfL_v}^{R_v}\oplus\left(\bigoplus\limits_{v'\in V-\{v\}}\fa_{\bfL_{v'}}^{R_{v'}}\right)=\fa_{\bfL}^{\bfG}$, by taking the orthogonal complements, we obtain 
$$ \fa_{R_v}^{\bfG_v}\oplus\left(\bigcap_{v'\in V-\{v\}}\fa_{R_{v'}}^{\bfG_{v'}}\right)=\fa_{\bfL}^{\bfG}. $$
It follows that $Z_{\wh{\bfL}}=A_v A^v$ and $A_v\cap A^v=Z_{\wh{\bfG}}$, where $A_v:=Z_{\wh{R_v}}$ and $A^v:=\bigcap\limits_{v'\in V-\{v\}}Z_{\wh{R_{v'}}}$. 

We claim that $Z_{\wh{\bfL}}[2]=A_v[2] A^v[2]$. To see this, let $s=a_v a^v\in Z_{\wh{\bfL}}[2]$, where $a_v\in A_v$ and $a^v\in A^v$. Since $s^2=1$, we have 
$$ a_v^2=(a^v)^{-2}\in A_v\cap A^v=Z_{\wh{\bfG}}. $$ 
Because $z\mapsto z^2$ is a surjective endomorphism of $Z_{\wh{\bfG}}$, there exists $y\in Z_{\wh{\bfG}}$ such that 
$$ a_v^2=(a^v)^{-2}=y^2. $$ 
Thus $s=y^{-1}a_v\cdot (ya^v)$ with $y^{-1}a_v\in A_v[2]$ and $ya^v\in A^v[2]$. We have shown our claim. 

Now, let $s=a_v a^v\in Z_{\wh{\bfL}}[2]$, where $a_v\in A_v[2]$ and $a^v\in A^v[2]$. Since the character $\zeta_v$ is trivial on $Z_{\wh{R_v}}[2]$, we have 
$$ \zeta_v(s)=\zeta_v(a_v)\zeta_v(a^v)=\zeta_v(a^v). $$ 
By the global condition above, we have 
$$ \zeta_v(a^v)=\prod\limits_{v'\in S_2-\{v\}} \zeta_{v'}^{-1}(a^v). $$ 
But for $v'\in S_2-\{v\}$, the character $\zeta_{v'}$ is trivial on $A^v[2]$. We have proved $\zeta_v(s)=1$ for all $s\in Z_{\wh{\bfL}}[2]$ and thus the lemma. 
\end{proof}

\begin{lem}\label{lemloceve}
For all $v\in V$, $X$ comes from $\fs'_\rs(k_v)$. 
\end{lem}

\begin{proof}
The assertion is trivial if $v\in V$ splits in $k'$ (in particular, if $v\in V_\infty$). For $v\in S_1-V_\infty-\{w\}$ which is inert in $k'$, it is deduced from Propositions \ref{testfunctions}.(a)(c). For $v=w$, it is because $f$ satisfies the two conditions in Theorem \ref{thmcommute}. 

It suffices to consider $v\in V-S_1$ which is inert in $k'$. In this case, recall that $(\bfG'_v, \bfH'_v)$ is identified with $(\bfG\bfL_{2n, k_v},\Res_{k'_v/k_v}\bfG\bfL_{n,k'_v})$ by the condition (d) of Proposition \ref{appglo}. We are in a similar situation as the base change for $GL_n$, and the following argument is similar to that of \cite[Lemma III.3.4]{MR1339717}. There exists $\bfL_0\in\msl^{\bfG,\omega}(\bfM_0), \bfL_0\subseteq\bfL$ and $y\in \bfL_\bfH(k_v)$ such that $Z:=\Ad(y^{-1})(X)\in(\fl_{0}\cap\fs_\rs)(k_v)_\el$. By Lemma \ref{lemLsigne}, we have 
\begin{equation}\label{eq:sign_all_1}
 \eta_{\bfL_v}(Z)=\eta_{\bfL_v}(X)=e_{\bfL'_v}=(1, \cdots, 1). 
\end{equation}
Suppose that $X$ does not come from $\fs'_\rs(k_v)$. Then $Z$ does not come from $\fs'_\rs(k_v)$. By Proposition \ref{prop:equivcomefromLevi} and Corollary \ref{corequivdefcomefrom}.(1), we have 
\begin{equation}\label{eq:sign_not_all_1}
 \eta_{\bfL_{0,v}}(Z)\neq e_{\bfL'_{0,v}}=(1,\cdots,1). 
\end{equation}
From the descent formula (cf. \cite[Proposition 4.1.(5)]{MR4681295} and \cite[Lemma I.6.2]{MR1339717}), we obtain 
$$ J_{\bfL_{v}}^{Q_{R_{v}}}(\eta, X, \phi_{v})=\eta(\det(y)) J_{\bfL_{v}}^{Q_{R_{v}}}(\eta, Z, \phi_{v})=\eta(\det(y)) \sum_{L_1\in\msl^{R_v}(\bfL_{0,v})} d_{\bfL_{0,v}}^{R_v}(\bfL_v, L_1) J_{\bfL_{0,v}}^{Q_1}(\eta, Z, \phi_v) $$
where $d_{\bfL_{0,v}}^{R_v}(\bfL_v, L_1)$ is defined in \cite[p.356]{MR928262} and $Q_1\in\msp^{\bfG_v}(L_1)$ is given in \cite[Lemma I.1.2]{MR1339717}. In particular, both of $L_1$ and $Q_1$ are $\omega$-stable. 

Consider $L_1\in\msl^{R_v}(\bfL_{0,v})$ such that $d_{\bfL_{0,v}}^{R_v}(\bfL_v, L_1)\neq0$. Then $d_{\bfL_{0,v,n}}^{R_{v,n}}(\bfL_{v,n}, L_{1,n})\neq0$ with the notations defined by \eqref{eq:notationM_n}. By \cite[Lemma 10.1]{MR928262}, the natural map 
$$ X(\bfL_{v,n})_{k_v}\oplus X(L_{1,n})_{k_v}\ra X(\bfL_{0,v,n})_{k_v} $$
is surjective. Then \eqref{eq:sign_all_1} and \eqref{eq:sign_not_all_1} imply that 
$$ \eta_{L_1}(Z)\neq e_{L'_1}=(1, \cdots, 1). $$
From Lemma \ref{wfl}.(b), we deduce that $J_{\bfL_{0,v}}^{Q_1}(\eta, Z, \phi_v)=0$. 

Hence, we have $J_{\bfL_{v}}^{Q_{R_{v}}}(\eta, X, \phi_{v})=0$, which contradicts our assumption \eqref{eq:contra_nonvan}. 
\end{proof}

\subsubsection{}

In this section, we shall prove the following statement which contradicts our assumption at the beginning of Section \ref{ssec:contradict}. This will complete the proof of Proposition \ref{prop102}. 

\begin{prop}\label{glocom}
The element $X$ comes from $\fs'_\rs(k)$. 
\end{prop}

\begin{proof}
We start with two lemmas. 

\begin{lem}\label{lem2451}
Let $F_0$ be a field. Let $Z_C\subseteq A\subseteq B\subseteq C$ be reductive groups defined over $F_0$, where $Z_C$ denotes the centre of $C$. Suppose that $H^1(F_0,A_\beta)$ is a singleton for all inner forms $A_\beta$ of $A$. Then the natural map 
$$ H^1(F_0,A/Z_C)\ra H^1(F_0,B/Z_C) $$
is injective. 
\end{lem}

\begin{proof}[Proof of Lemma \ref{lem2451}]
We begin with the following commutative diagram with exact rows. 
$$    \xymatrix{ 0\ar[r] & Z_C \ar[r] \ar@{=}[d]  & A \ar[r]  \ar[d] & A/Z_C \ar[r] \ar[d] & 0  \\
 0\ar[r] & Z_C \ar[r] &B\ar[r] & B/Z_C \ar[r] & 0 } $$
Then we obtain the following commutative diagram of pointed sets. 
$$    \xymatrix{   H^1(F_0,A/Z_C) \ar[r]  \ar[d]^g & H^2(F_0,Z_C) \ar@{=}[d]   \\
   H^1(F_0,B/Z_C) \ar[r] & H^2(F_0,Z_C)   } $$
Since $H^1(F_0,A_\beta)$ is a singleton for all inner forms $A_\beta$ of $A$, the map 
$$ H^1(F_0,A/Z_C) \ra H^2(F_0,Z_C) $$
is injective by \cite[Corollary (28.13)]{MR1632779}. As the above diagram is commutative, we deduce that $g$ is injective. 
\end{proof}

\begin{lem}\label{lem2452}
Let $F_0$ be a number field. Let $Z_B\subseteq A\subseteq B$ be reductive groups defined over $F_0$, where $Z_B$ denotes the centre of $B$. Suppose that $H^1(F_0,A_\beta)$ is a singleton for all inner forms $A_\beta$ of $A$. Then the map 
$$ H^1(F_0,A/Z_B)\ra H^1(\BA_{F_0},A/Z_B) $$
is injective. 
\end{lem}

\begin{proof}[Proof of Lemma \ref{lem2452}]
We use the commutative diagram of pointed sets: 
$$    \xymatrix{  H^1(F_0,A/Z_B) \ar[r]^i \ar[d]^g  & H^1(\BA_{F_0},A/Z_B)  \ar[d]    \\
   H^1(F_0,B_\ad) \ar[r]^h &H^1(\BA_{F_0},B_\ad)    } $$
By Lemma \ref{lem2451}, the map $g$ is injective. By the Hasse principle for $B_\ad$ (see \cite[Theorem 6.22]{MR1278263}), the map $h$ is injective. Since the diagram is commutative, we deduce that $i$ is injective. 
\end{proof}

%
%

Return to the proof of Proposition \ref{glocom}. Let $t\in H^1(k,\bfT/\bfR)$ be the class of the Galois $1$-cocycle $t_\sigma$ in Lemma \ref{lem1231} associated to $X$. Let $u\in H^1(k,\bfH_0/Z_\bfG)$ be the class of the Galois $1$-cocycle $u_\sigma$ in Section \ref{ssec:classEg'} associated to $\fg'$. Let $t'\in H^1(\BA,\bfT/\bfR)$ (resp. $u'\in H^1(\BA,\bfH_0/Z_\bfG)$) be the image of $t$ (resp. $u$). By Lemmas \ref{lemloceve} and \ref{lem1251}, there exists $v'\in H^1(\BA,\bfT/Z_\bfG)$ with images $t'$ and $u'$. 

By \cite[Proposition 1.6.12]{MR1695940}, we have the following commutative diagram of pointed sets with exact rows. 
$$     \xymatrix{  H^1(k,\bfT/Z_\bfG) \ar[r] \ar[d]  & H^1(\BA,\bfT/Z_\bfG) \ar[r]  \ar[d] & H^1(\BA/k,\bfT/Z_\bfG) \ar[d]^g   \\
   H^1(k,\bfH_0/Z_\bfG) \ar[r]^i & H^1(\BA,\bfH_0/Z_\bfG) \ar[r] & H^1_\ab(\BA/k,\bfH_0/Z_\bfG)   }  $$
Let $v''\in H^1(\BA/k,\bfT/Z_\bfG)$ (resp. $u''\in H^1_\ab(\BA/k,\bfH_0/Z_\bfG)$) be the image of $v'$ (resp. $u'$). Then $u''=0$. By Shapiro's lemma and global class field theory (see \cite[Theorem 5.1.(b) in Chapter VII]{milneCFT} for example), we know that $H^1(\BA/k, \bfT)$ is trivial. Then it results from a variant of Lemma \ref{lem2451} for $H^1_\ab(\BA/k,\cdot)$ that the map $g$ is injective. Thus $v''=0$ and there exists $v\in H^1(k,\bfT/Z_\bfG)$ with image $v'$. By Lemma \ref{lem2452}, the map $i$ is injective. Since the square on the left above is commutative, the class $v$ maps to $u$. 

We also have the following commutative diagram. 
$$     \xymatrix{  H^1(k,\bfT/Z_\bfG) \ar[r] \ar[d]  & H^1(\BA,\bfT/Z_\bfG) \ar[d]    \\
   H^1(k,\bfT/\bfR) \ar[r]^j & H^1(\BA,\bfT/\bfR)    }  $$
By \eqref{eq:coh_opp_tor} and its adelic variant, the map $j$ is given by 
$$ \prod_{i\in I_0} k_i^\times/N_{k'_i/k_i}({k'_i}^\times)\ra \prod_{i\in I_0} \BA_{k_i}^\times/N_{k'_i/k_i}(\BA_{k'_i}^\times), $$
where $k_i=k[\lambda]/(\chi_i(\lambda))$ with $\chi_i$ being an irreducible polynomial over $k$, and $k'_i=k_i\otimes_k k'$ is a field. It is known to be injective from the fundamental exact sequence of global class field theory for $k'_i/k_i$ (see \cite[Example 4.4.(a) in Chapter VIII]{milneCFT} for example). Since the diagram is commutative, the class $v$ maps to $t$. 

Hence, we can finish the proof of Proposition \ref{glocom} using Lemma \ref{lem1251}. 
\end{proof}

\subsection{End of the proof}

\begin{lem}\label{lem106}
For our choice of $\phi$ and $\phi'$, we have 
$$ J^\bfG(\eta, \hat{\phi})=\tau(\bfH_{X^0}) J_\bfM^\bfG(\eta, X^0, \hat{\phi}) $$
and 
$$ J^{\bfH'}(\hat{\phi'})=\tau(\bfH'_{Y^0}) J_{\bfM'}^{\bfH'}(Y^0, \hat{\phi'}). $$
\end{lem}

\begin{proof}
The first formula results from \eqref{glodefJ^G(eta,f)} for $\hat{\phi}$ and the condition (iv) in Section \ref{choffun}, while the second formula results from \eqref{glodefJ^H'(f')} for $\hat{\phi'}$ and the condition (iv) in Section \ref{choffun}. 
\end{proof}

\begin{proof}[Proof of Proposition \ref{parcommute}]
Combining Propositions \ref{propinfivar} and \ref{prop102}, Lemma \ref{lem106} and \eqref{eq:tau_equal} applied to $X^0\overset{\bfM}{\da}Y^0$, for our choice of $\phi$ and $\phi'$, we have the equality 
$$ J_\bfM^\bfG(\eta, X^0, \hat{\phi})=J_{\bfM'}^{\bfH'}(Y^0, \hat{\phi'}). $$
But by the splitting formulae \eqref{eq:fact_woi1} and \eqref{eq:fact_woi2} applied to $\hat\phi$ and $\hat{\phi'}$, the equality of Arthur's coefficients \eqref{eq:equal-d}, the conditions (iii) and (iv) in Section \ref{choffun} as well as Lemma \ref{wfl}.(a), the difference of two sides can be written as 
\begin{equation}\label{chau1023}
\begin{split} 
&\sum_{(R_v)_{v\in V}\in\prod\limits_{v\in V}\msl^{\bfG_v}(\bfM_v)} d_\bfM^\bfG((R_v)_{v\in V}) \prod_{v\in (V-S_1)\cup V_\infty} \kappa_v(X^0) J_{\bfM_v}^{Q_{R_v}}(\eta, X^0, \wh{\phi_v}) \\ 
&\times \left[\prod_{v\in S_1-V_\infty} \kappa_v(X^0) J_{\bfM_v}^{Q_{R_v}}(\eta, X^0, \wh{\phi_v}) - \prod_{v\in S_1-V_\infty} J_{\wt{\bfM'_v}}^{Q_{\wt{R'_v}}}(Y^0,\wh{\phi'_v}) \right]. \\ 
\end{split}
\end{equation}
By the condition (e) in Proposition \ref{testfunctions}, we may suppose that $R_v=\bfM_v$ for all $v\in S_1-V_\infty-\{w\}$. Moreover, by the condition (f) in Proposition \ref{testfunctions}, we have 
$$ J_{\wt{\bfM'_v}}^{Q_{\wt{\bfM'_v}}}(Y^0, \wh{\phi'_v})=\gamma_{\Psi_v}(\fh'(k_v)) \gamma_{\Psi_v}(\fh(k_v))^{-1} \kappa_v(X^0) J_{\bfM_v}^{Q_{\bfM_v}}(\eta, X^0, \wh{\phi_v})\neq 0 $$
for all $v\in S_1-V_\infty-\{w\}$. 

Recall from \cite[\S 30, Proposition 5]{MR0165033} the product formula 
$$ \prod_{v\in V} \gamma_{\Psi_v}(\fh(k_v))=\prod_{v\in V} \gamma_{\Psi_v}(\fh'(k_v))=1. $$
For $v\in V_\infty$, it follows from the condition (a) of Proposition \ref{appglo} and \cite[\S 26]{MR0165033} that 
$$ \gamma_{\Psi_v}(\fh(k_v))=\gamma_{\Psi_v}(\fh'(k_v))=1. $$
For $v\in V-S_1$, since $\fh(k_v)$ and $\fh'(k_v)$ possess self-dual lattices $(\fg\fl_n\oplus\fg\fl_n)(\CO_{k_v})$ and $\fg\fl_n(\CO_{k'_v})$ respectively, by \cite[\S 20, Théorème 5]{MR0165033}, we have 
$$ \gamma_{\Psi_v}(\fh(k_v))=\gamma_{\Psi_v}(\fh'(k_v))=1. $$
Hence, 
$$ \prod_{v\in S_1-V_\infty-\{w\}} \gamma_{\Psi_v}(\fh'(k_v)) \gamma_{\Psi_v}(\fh(k_v))^{-1}=\gamma_{\Psi_w}(\fh(k_w))\gamma_{\Psi_w}(\fh'(k_w))^{-1}. $$
Then the expression \eqref{chau1023} equals 
\begin{equation}\label{chau1024}
\begin{split} 
&\sum_{(R_v)_{v\in V}\in\prod\limits_{v\in V}\msl^{\bfG_v}(\bfM_v)} d_\bfM^\bfG((R_v)_{v\in V}) \prod_{v\in V-\{w\}} \kappa_v(X^0) J_{\bfM_v}^{Q_{R_v}}(\eta, X^0, \wh{\phi_v}) \\ 
&\times \left[\kappa_w(X^0) J_{\bfM_w}^{Q_{R_w}}(\eta, X^0, \wh{\phi_w}) - \gamma_{\Psi_w}(\fh(k_w)) \gamma_{\Psi_w}(\fh'(k_w))^{-1} J_{\wt{\bfM'_w}}^{Q_{\wt{R'_w}}}(Y^0,\wh{\phi'_w}) \right]. \\ 
\end{split}
\end{equation}

We shall prove \eqref{chaufor101} by induction on the dimension of $G$. By the condition (i) on $\Omega_w$, the condition (i) on $Y^0$ and $X^0\overset{\bfM}{\da}Y^0$ (see Section \ref{ssec:orbits}), we may replace $X_0$ and $Y_0$ with $X^0$ and $Y^0$ respectively in \eqref{chaufor101}. Since $\phi_w$ and $\phi'_w$ are partially $\bfM_w$-associated, by parabolic descent (see \cite[Propositions 4.1.(4) and 4.4.(4)]{MR4681295}, \eqref{eq:pardes1} and \eqref{eq:pardes2}), we see that $\phi_{w,Q_{R_w}}^\eta\in\CC_c^\infty((\fr_w\cap\fs_w)(k_w))$ and $\phi'_{w,(Q_{R_w})'}\in\CC_c^\infty((\wt{\fr'_w}\cap\fs'_w)(k_w))$ are partially $\bfM_w$-associated. As part of the induction hypothesis, we may suppose that for $R_w\neq \bfG_w$, we have the equality 
$$ \kappa_w(X^0) J_{\bfM_w}^{R_w}(\eta,X^0,\hat{\phi}_{w,Q_{R_w}}^\eta)=\gamma_{\Psi_w}((\fr_{w}\cap\fh_w)(k_w))\gamma_{\Psi_w}(\fr'_w(k_w))^{-1}J_{\wt{\bfM'_w}}^{\wt{R'_w}}(Y^0, \hat{\phi'}_{w,(Q_{R_w})'}). $$
Since the difference between the quadratic form on $\fh(k_w)$ (resp. $\fh'(k_w)$) and its restriction on $(\fr_{w}\cap\fh_w)(k_w)$ (resp. $\fr'_w(k_w)$) is a split quadratic form, by \cite[\S 25, Proposition 3]{MR0165033}, we have $\gamma_{\Psi_w}((\fr_{w}\cap\fh_w)(k_w))=\gamma_{\Psi_w}(\fh(k_w))$ (resp. $\gamma_{\Psi_w}(\fr'_w(k_w))=\gamma_{\Psi_w}(\fh'(k_w))$). By parabolic descent again, we see that 
\begin{equation}\label{eq:ind-hyp}
 \kappa_w(X^0) J_{\bfM_w}^{Q_{R_w}}(\eta,X^0,\wh{\phi_w})=\gamma_{\Psi_w}(\fh(k_w))\gamma_{\Psi_w}(\fh'(k_w))^{-1}J_{\wt{\bfM'_w}}^{Q_{\wt{R'_w}}}(Y^0, \wh{\phi'_w}), 
\end{equation}
which implies \eqref{chaufor101} for all $Q_{R_w}\neq \bfG_w$. Actually, such an argument shows \eqref{chaufor101} for all $Q_w\in\msl^{\bfG_w}(\bfM_w), Q_w\neq \bfG_w$, so it suffices to prove \eqref{chaufor101} for $Q_w=\bfG_w$. 

With the equality \eqref{eq:ind-hyp}, we may and we shall suppose $R_w=\bfG_w$ in \eqref{chau1024}. If  $d_\bfM^\bfG((R_v)_{v\in V})\neq0$, then $\bigoplus\limits_{v\in V}\fa_{\bfM_v}^{R_v}=\fa_{\bfM}^{\bfG}$. But we know $\fa_{\bfM_w}^{\bfG_w}\simeq\fa_\bfM^\bfG$. Hence, $d_\bfM^\bfG((R_v)_{v\in V})=0$ unless $R_v=\bfM_v$ for all $v\in V-\{w\}$, in which case $d_\bfM^\bfG((R_v)_{v\in V})=1$. That is to say, the sum in \eqref{chau1024} is reduced to only one term. We obtain 
\[\begin{split} 
&\prod_{v\in V-\{w\}} \kappa_v(X^0) J_{\bfM_v}^{Q_{\bfM_v}}(\eta, X^0, \wh{\phi_v}) \\ 
&\times\left[\kappa_w(X^0) J_{\bfM_w}^{\bfG_w}(\eta, X^0, \wh{\phi_w}) - \gamma_{\Psi_w}(\fh(k_w)) \gamma_{\Psi_w}(\fh'(k_w))^{-1} J_{\wt{\bfM'_w}}^{\bfG'_w}(Y^0,\wh{\phi'_w}) \right]=0. \\ 
\end{split}\]
To conclude, it suffices to show that $\kappa_v(X^0) J_{\bfM_v}^{Q_{\bfM_v}}(\eta, X^0, \wh{\phi_v})$ does not vanish for each $v\in V-\{w\}$. For $v\in V_\infty$, it results from the condition (iv) in Section \ref{choffun}. For $v\in S_1-V_\infty-\{w\}$, it follows from the condition (f) in Proposition \ref{testfunctions}. For $v\in V-S_1$, since $Y^0\in(\wt{\fm'}\cap\fs'_\rs)(k_v)\cap\fk'_v$ (see Section \ref{ssec:orbits}), we have 
$$ J_{\wt{\bfM'_v}}^{Q_{\wt{\bfM'_v}}}(Y^0, \wh{\phi'_v})\geq\vol(\bfH'_{v,Y^0}(k_v)\bs\bfH'_{v,Y^0}(k_v)\cdot\bfH'_v(\CO_{k_v}))>0. $$
Then the nonvanishing of $\kappa_v(X^0) J_{\bfM_v}^{Q_{\bfM_v}}(\eta, X^0, \wh{\phi_v})$ is a consequence of Lemma \ref{wfl}.(a) and $X^0\overset{\bfM}{\da}Y^0$. 
\end{proof}


\bibliography{References}
\bibliographystyle{plain}

\medskip

\begin{flushleft}
Johns Hopkins University \\
Department of Mathematics, 404 Krieger Hall \\
3400 N. Charles Street, Baltimore, MD 21218 \\
United States \\
\medskip
E-mail: hli213@jhu.edu \\
\end{flushleft}

\end{document}